\documentclass[11pt]{article}
\usepackage{amsmath}
\usepackage{amssymb}
\usepackage{amsthm}
\usepackage{graphicx}
\usepackage{mathtools}
\usepackage{float}
\usepackage{tabularx}
\usepackage{array}
\usepackage{wrapfig}
\usepackage{multirow}
\usepackage{tabu}
\usepackage{booktabs}
\usepackage{pgfplots}
\usepackage{pgfplotstable}
\usepackage{qtree}
\usepackage{xcolor}
\usepackage{listings}
\usepackage{caption}
\usepackage{siunitx}
\usepackage{makecell}
\usepackage{bm}
\usepackage{subfig}
\usepackage{titlesec}
\usepackage{url}
\usepackage{algorithm}
\usepackage{algpseudocode}
\usepackage{cancel} 
\usepackage{natbib}
\usepackage[bookmarks=false,hidelinks]{hyperref} 
\usepackage{appendix}
\usepackage{verbatim}
\usepackage{enumitem} 
\DeclareCaptionFont{white}{\color{white}}
\DeclareCaptionFormat{listing}{%
\parbox{\textwidth}{\colorbox{gray}{\parbox{\textwidth}{#1#2#3}}\vskip-4pt}}

\pgfplotsset{compat=1.16}

\newtheorem{theorem}{Theorem}[section]
\newtheorem{lemma}[theorem]{Lemma}
\newtheorem{corollary}{Corollary}[theorem]
\newtheorem{assumption}{Sampling Scheme}[section]
\newtheorem{remark}{Remark}[section]

\date{}
\begin{document}

\author{Xiaoran Jiang \\ 
  Department of Statistics \\ 
    George Mason University \\
    Fairfax, VA, 22030
 \and 
 Anand N. Vidyashankar \\
    Department of Statistics \\ 
     George Mason University \\ 
     Fairfax, VA, 22030 } 

\title{Ancestral Inference and Learning for Branching Process in Random Environments}
\maketitle

\begin{abstract}    
Ancestral inference for branching processes in random environments involves determining the ancestor distribution parameters using the population sizes of descendant generations. In this paper, we introduce a new methodology for ancestral inference utilizing the generalized method of moments. We demonstrate that the estimator's behavior is critically influenced by the coefficient of variation of the environment sequence. Furthermore, despite the process's evolution being heavily dependent on the offspring means of various generations, we show that the joint limiting distribution of the ancestor and offspring estimators of the mean, under appropriate centering and scaling, decouple and converge to independent Gaussian random variables when the ratio of the number of generations to the logarithm of the number of replicates converges to zero. Additionally, we provide estimators for the limiting variance and illustrate our findings through numerical experiments and data from Polymerase Chain Reaction experiments and COVID-19 data.
\end{abstract}

\section{Introduction}
\label{sec-intro}
Branching processes and their variants are used to model various biological, biochemical, and epidemic processes. 
For example, these methods have been used as a model for quantitative polymerase chain reaction (qPCR) experiments and for tracking the spread of COVID-19 cases in communities during the early stages of the pandemic (see \cite{Fran-Vidya2024}; \cite{Yanev2020}; \cite{Atanasov2021}). 
Estimating the mean number of ancestors based on the data from descendant generations is known as ancestral inference, and it arises in many applications. One important application is in polymerase chain reaction (PCR) experiments, a biochemical experiment used to amplify the number of copies of a specific DNA sequence, referred to as the target DNA. This experiment involves multiple rounds of thermal cycling, an experimental process of heating and cooling that creates the conditions necessary for DNA replication. In Section \ref{sec-da}, we discuss the details of this experiment.  Quantitation refers to estimating the initial number of DNA molecules in the context of PCR experiments and is a particular case of the ancestral inference problem. Existing methods for quantitation include linear or non-linear regression models (\cite{LIVAK2001402}; \cite{Goll2006}; \cite{Hsu-Sherina20}) for the fluorescence data. Alternatively, a supercritical Galton-Watson (GW) branching process is a mechanistic model used for analyzing qPCR data (see, for instance, \cite{L09A}; \cite{Follestad2010}; \cite{Bret-Vidya2011}).   

This paper describes new inferential techniques for estimating the mean number of ancestors generating branching processes in random environments (BPRE). Specifically, we consider a replicated BPRE model in which replicates are independent and identically distributed (i.i.d.) random variables initiated by a random number of i.i.d. ancestors.

\vspace{0.1in}

\noindent{\bf{Probability model:}} We begin with an introduction to our model. Let $\{ Z_{0,j}: 1 \le j \le J \}$ be a collection of i.i.d. non-negative discrete random variables with distribution $    P(Z_{0,1}=r)=\mu(r)$, $r=1,2,\cdots $ representing the initial number of ancestors defined on a probability space $(\Omega, \mathcal{F},P)$. Each ancestor initiates a BPRE as follows: let $\{\xi_{n,j,k}: n \ge 0, j \ge 1, k \ge 1\}$ denote a collection of random variables which are i.i.d. conditioned on a $\zeta_{n,j}$ with distribution $p(r|\zeta_{n,j})$. The random variables $\{\zeta_{n,j}: n \ge 0, j \ge 1\}$ are i.i.d. and assume values on $\mathfrak{B}$ which is equipped with the Borel $\sigma$-field $\mathcal{B}(\mathfrak{B})$ generated by the topology of weak convergence (see \cite{athreya2004branching}). Hence, with probability one (w.p.1.),
\begin{align*}
    \bm{P}(\xi_{n,j,k}=r|\zeta_{n,j})=p(r|\zeta_{n,j}).
\end{align*}
We refer to $\bm{p}( \cdot|\zeta_{n,j})=(  p(0|\zeta_{n,j}),  p(1|\zeta_{n,j}) ,   p(2|\zeta_{n,j})  , \cdots )   $ as the offspring distribution of the $n^{\text{th}}$ generation in the $j^{\text{th}}$ replicate conditioned on the environment. The replicated branching process in i.i.d. environments is defined iteratively as follows: for each $n\ge 0$, $j=1,2 \cdots J$,
\begin{align*}
    Z_{n+1,j}=\sum_{k=1}^{ Z_{n,j}} \xi_{n,j,k}.
\end{align*}
The random variable $\xi_{n,j,k} $ is interpreted as the number of children produced by the $k^{\text{th}}$ parent in the $n^{\text{th}} $ generation of $j^{\text{th}} $ replicate with the offspring distribution $ \bm{p}( \cdot|\zeta_{n,j})$. $Z_{n,j}$ denotes the total population size of children in the $n^{\text{th}} $ generation of $j^{\text{th}} $ replicate. In our analysis, $J$ will depend on $n$, and we suppress it in the notations when there is no scope for confusion.
Let $m_A=\bm{E}[Z_{0,1}]$, $\sigma^2_A=\bm{Var}[Z_{0,1}] $ and
\begin{align*}
    m_{n,j}=\sum_{r\ge 1} r p(r|\zeta_{n,j}), \quad \sigma^2_{n,j} = \sum_{j\ge 1} j^2 p(j|\zeta_{n,j}) -   m_{n,j}^2
\end{align*}
denote the offspring mean and variance conditioned on the $n^{\text{th}}$ generation offspring distribution. We refer to them as
conditional offspring mean and conditional offspring variance. Since $\{ \zeta_{n,j}: n \ge 0, j \ge 1\} $ forms an i.i.d. collection, it follows that $ \{ (m_{n,j},\sigma^2_{n,j} ) : n \ge 0, j \ge 1\} $ forms an i.i.d. collection of random vectors.  Let $m^*=\bm{E} [m_{0,1} ]$ and $\sigma^{2*}= \bm{Var} [m_{0,1}] $ denote the expectation and variance of $m_{0,1} $. Let $\gamma^{2*}=\bm{E}[\sigma_{0,1}^2]  $ and for any $d >1 $, let $ m_{d}^*=\bm{E} [m_{0,1}^d ]$. Notice that while $m_{n,j}$ is the conditional offspring mean, $m^* $ represents the marginal offspring mean. Similarly, $\gamma^{2*} $ represents the marginal offspring variance.
We assume throughout the paper that the following two conditions hold:

\begin{enumerate}[label=\textbf{(A\arabic*)}, ref=\textbf{(A\arabic*)}]
    \item $p(0|\zeta_{n,j})=0$, for all $ n\ge 0 $ and $j \ge 1$. \label{AssumA1}
    \item  $\bm{P}(p(1|\zeta_{n,j})  <1)=1$, for all $ n\ge 0 $ and $j \ge 1$. \label{AssumA2}
\end{enumerate}

\noindent
The condition \ref{AssumA1} ensures that the process is supercritical and  that $Z_{n,j} \to \infty $ w.p.1.(see, for instance, \cite{athreya2004branching}; \cite{Kersting2017DiscreteTB}; \cite{Tanny1977}). 
The condition \ref{AssumA2} ensures that the process is non-degenerate. As is well-known, when the offspring distribution is fixed across all generations, one obtains the GW process (see \cite{athreya2004branching}).

The next two assumptions are required for asymptotic analyses, discussed in later sections.

\begin{enumerate}[resume,label=\textbf{(A\arabic*)}, ref=\textbf{(A\arabic*)}]
\item $\gamma^{2*} < \infty $. \label{Assum-gamma}
\item   $m_{4+2\delta}^*< \infty$  and $\bm{E}[ Z_{0,1}^{4+2\delta} ]  < \infty    $ for some $\delta>0 $. \label{Assum-mo4+2delta}
\end{enumerate}
We note here that \ref{Assum-gamma} 
and  \ref{Assum-mo4+2delta} are not required for all our results and will be specified when necessary.

We begin with a brief description of estimating $m_A$. The simplest 
estimator, using $J$ replicates is $\bar{Z}_{0} \coloneqq J^{-1} ( Z_{0,1}+ Z_{0,2}+\cdots+ Z_{0,J})$. 
However, if  $Z_{0,j}$'s are unobservable and only the population size after a certain number of generations is available, then one way to estimate the parameters of the ancestor distribution would be by using the information contained in the descendant generations. In the GW case, the information about $m_A$ is contained in the martingale limit $W_j$ of $m^{*(-n)}Z_{n,j}$ and this information can be extracted by using the generalized method of moments. This idea was utilized in \cite{Bret-Vidya2011} to perform inference for the quantitation parameters. However, the GW model is not adequate to capture all sources of variability in PCR experiments as demonstrated in Section \ref{sec-da}. For this reason, we undertake a detailed analysis of modeling such data using BPRE and develop new technical tools to address the analytical challenges.

We now turn to another motivating example where the proposed model and the method arise naturally. When modeling the growth of disease spread such as COVID-19 and smallpox, most research concentrated on where and when the disease began to spread. On the other hand, our focus is on quantifying the mean number of people who initiated the disease spread. However, data at the beginning of the disease spread is unreliable for many reasons, such as lack of awareness and limited testing capabilities. Hence, it is critical to develop inferential methods that rely solely on data from the $\tau^{\text{th}} $ generation onward when the data of the number of patients becomes more reliable. 

We now turn to a more detailed description of the Polymerase chain reaction and dynamic of COVID-19 growth.

\subsection{Polymerase chain reaction}
\label{sec-intro-PCR}
We first discuss the ideas behind PCR experiments. As described above, the PCR experiment consists of multiple rounds of thermal cycling. Although theoretically, the number of molecules doubles in every cycle, in practice, the probability of a molecule being duplicated is less than one. The probability of a molecule duplicating ($p$) is known as the efficiency of the reaction, and the average number of DNA molecules produced by a single strand of DNA is referred to as the amplification rate $m=1+p$.

\begin{figure}[ht]
\centering
\includegraphics [height=8.0cm]{ 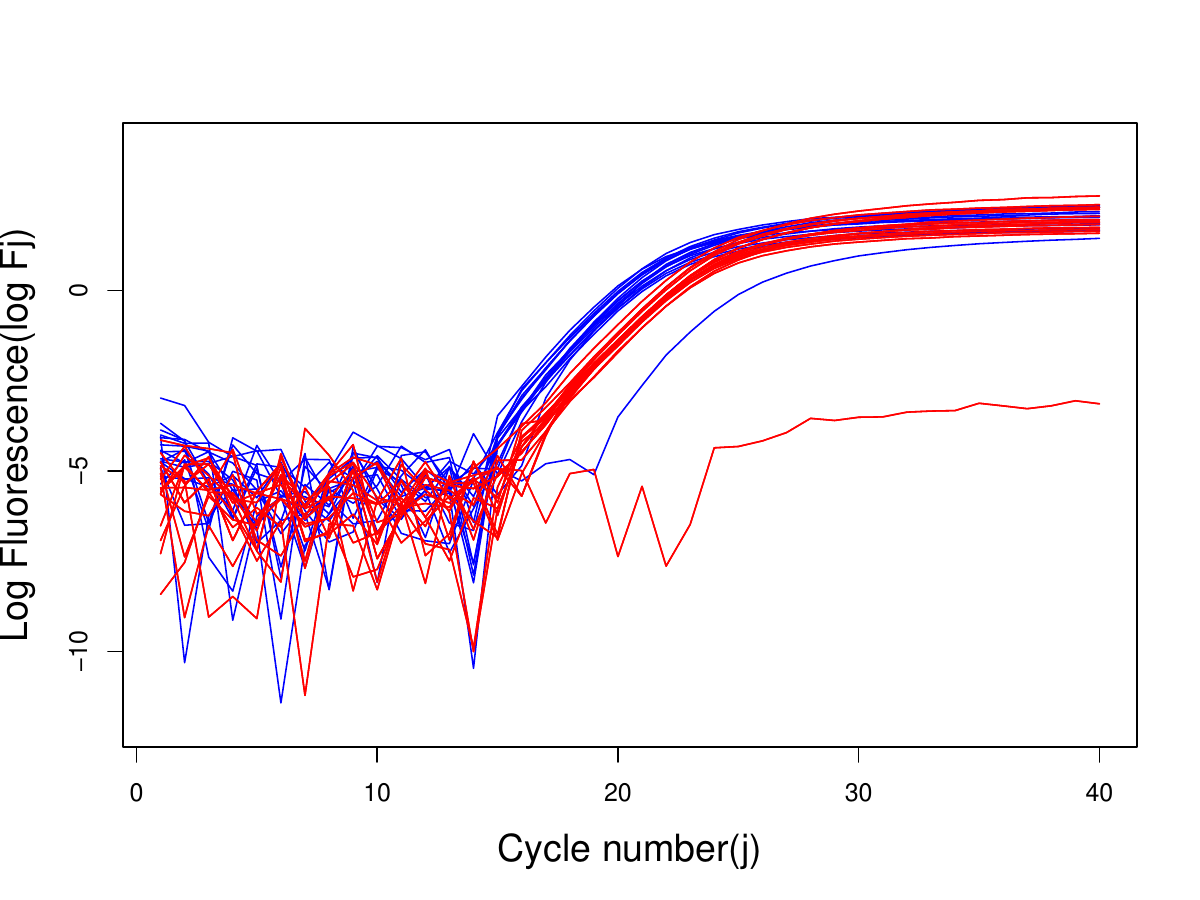 } 
\caption{Example of the PCR dataset}
\label{LHflo}
  \end{figure}

There are two types of quantitation: absolute quantitation and relative quantitation. Absolute quantitation involves the estimation of the copy number of a target DNA. In contrast, relative quantitation concerns the ratio of the copy numbers of a target DNA to that of a reference DNA, whose initial number and amplification process is ``well-understood”. The reference DNA is referred to as a calibrator.

The amount of DNA molecules is typically too small to be detected by standard instruments. A quantitative PCR (qPCR) experiment is a technique that allows for the measurement of DNA generation at every cycle using the fluorescent reporter. The cycles of a PCR can be classified into three phases: the initial phase, the exponential phase, and the plateau phase. In the initial phase, the amount of the target DNA is not substantial compared with the background, and hence, the fluorescence data is noisy. In the exponential phase, the amount of DNA increases, and the fluorescence measurement becomes reliable. As the reaction components become limiting, the efficiency of the reaction decreases, leading to a saturation phenomenon referred to as the plateau phase. Figure \ref{LHflo}, which plots fluorescence data from a qPCR experiment where the target material is luteinizing hormone (LH) on the log scale from 32 reactions, illustrates these three phases. 

A typical qPCR experiment generates data from multiple reactions, with each reaction occurring in a separate well. Replicating the same PCR experiment in multiple wells under identical conditions produces data that can be represented as  $\{F_{n,j}, n\ge 1, j \ge 1 \} $, where $F_{n,j} $ denotes the fluorescence intensity from the $n^\text{th}$ cycle of the $j^\text{th}$ replicate. The fluorescence intensity is proportional to the amount of DNA molecules. Thus,  as in PCR literature, we assume that $F_{n,j}=cZ_{n,j} $, where $c$ depends on the amplicon size and the calibration factor, which represents the number of nanograms of double-stranded DNA per fluorescence unit. More details concerning the experiment and scientific details can be found in \cite{Bret-Vidya2011} and the references therein.

Variability is introduced at several stages of the experiment. First, \textbf{between reaction variability}, such as pipetting error (\cite{Curry2002FactorsIR}), arises from the changes in the volume of PCR supplies that are pipetted from the master mix into the wells. In random effects models, the efficiency of $j^\text{th}$ reaction is treated as a random variable $p_j $ to account for between reaction variability. 
Second, \textbf{within reaction variability} occurs. For example, when the reaction components, such as dNTPs, become limiting, it leads to the plateau phase (\cite{Saha2007}), which indicates that the amplification rate varies with each cycle. Consequently, it is critical to consider both within and between reaction variability when modeling PCR experiments. The GW model with random effects does not capture \textbf{within reaction variability}.
\subsection{Dynamics of COVID-19 growth}

Another new aspect that is beginning to attract attention concerns the COVID-19 dataset. While most research had concentrated on where and when COVID-19 began to spread, our focus is on quantifying the mean number of people who initiated the disease spread, whom we refer to as ``ancestral subjects". For this analysis, we focus on the weekly confirmed COVID-19 cases at the beginning of 2020. The spread of the disease depends on many factors, such as population size and policies. Policies can vary significantly between different states but tend to remain consistent within a state over a period of time. Meanwhile, within a state like Virginia (see Figure \ref{VA-pop-fig}), the population sizes of different counties can differ greatly. Therefore, for the COVID-19 dataset, we can treat counties with similar population sizes as replicates and then compare the number of ancestral subjects between different groups of counties.

\begin{figure}[ht]
\centering
\includegraphics [width=\linewidth]{ 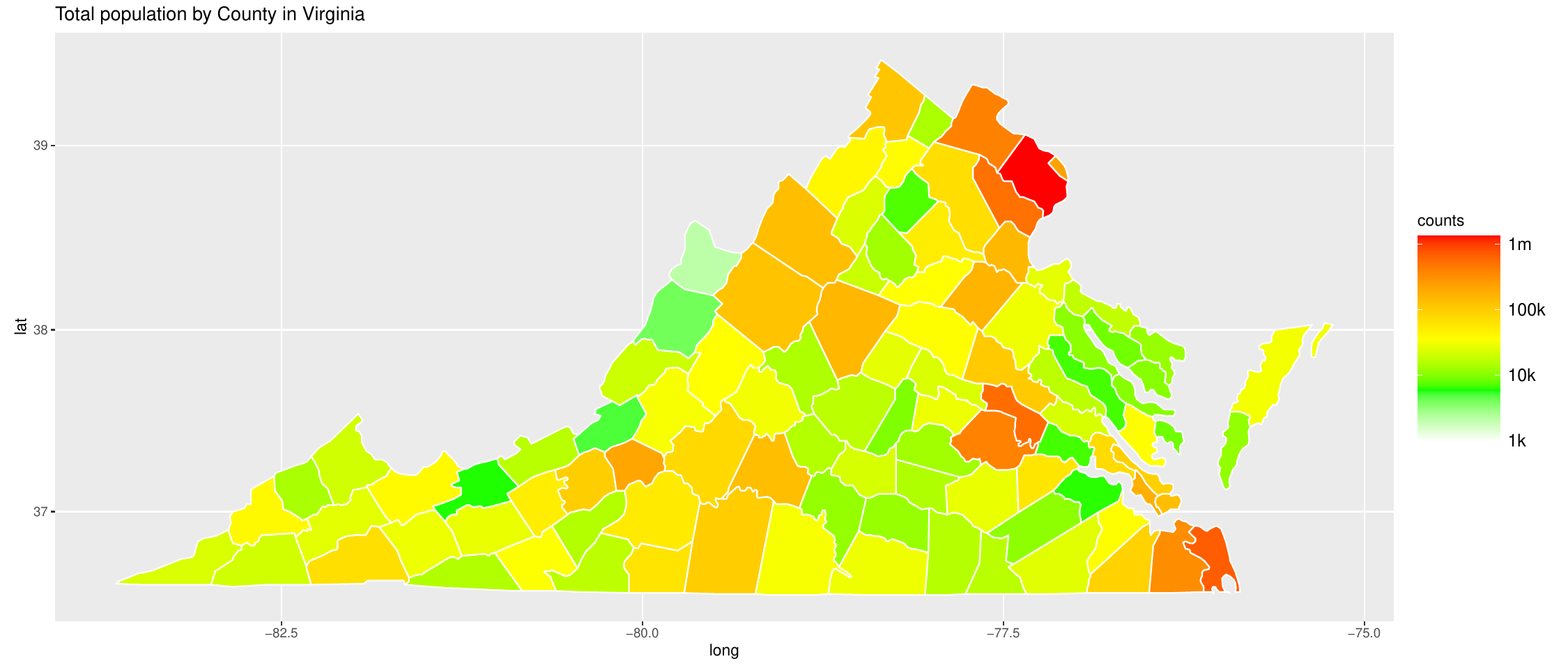 } 
\caption{VA Population}
\label{VA-pop-fig}
  \end{figure}

Consider, for example, the accumulated patients for counties in Virginia grouped by population size in Figure \ref{gpp_plot}. 
Initially, the disease did not receive sufficient attention, and accurate testing methods were not widely available (and perhaps promoted), which explains the low reported data for the first few weeks. As testing became more accessible and the seriousness of the disease was realized, the growth rate of reported cases increased significantly for a few days to weeks. However, many of these cases involved patients who were infected in previous weeks, rendering the data from this period an unreliable reflection of the growth of the disease. This pattern resembles the initial phase in PCR experiments, where observed data may be unreliable. Eventually, after resolving the historical cases from previous weeks, the growth rate stabilized, and the data became more reliable.

After some time, when the spread of COVID-19 was controlled, the growth rate began to decline. Thus, the progression of COVID-19 dynamics mirrors the three phases in a PCR dataset, though for different reasons. Consequently, we utilize data from the period when the growth rate remained stable before it began to decline to estimate the initial number of infected subjects. 

\begin{figure}[ht]
    \centering
    \includegraphics[width=0.75\linewidth]{ 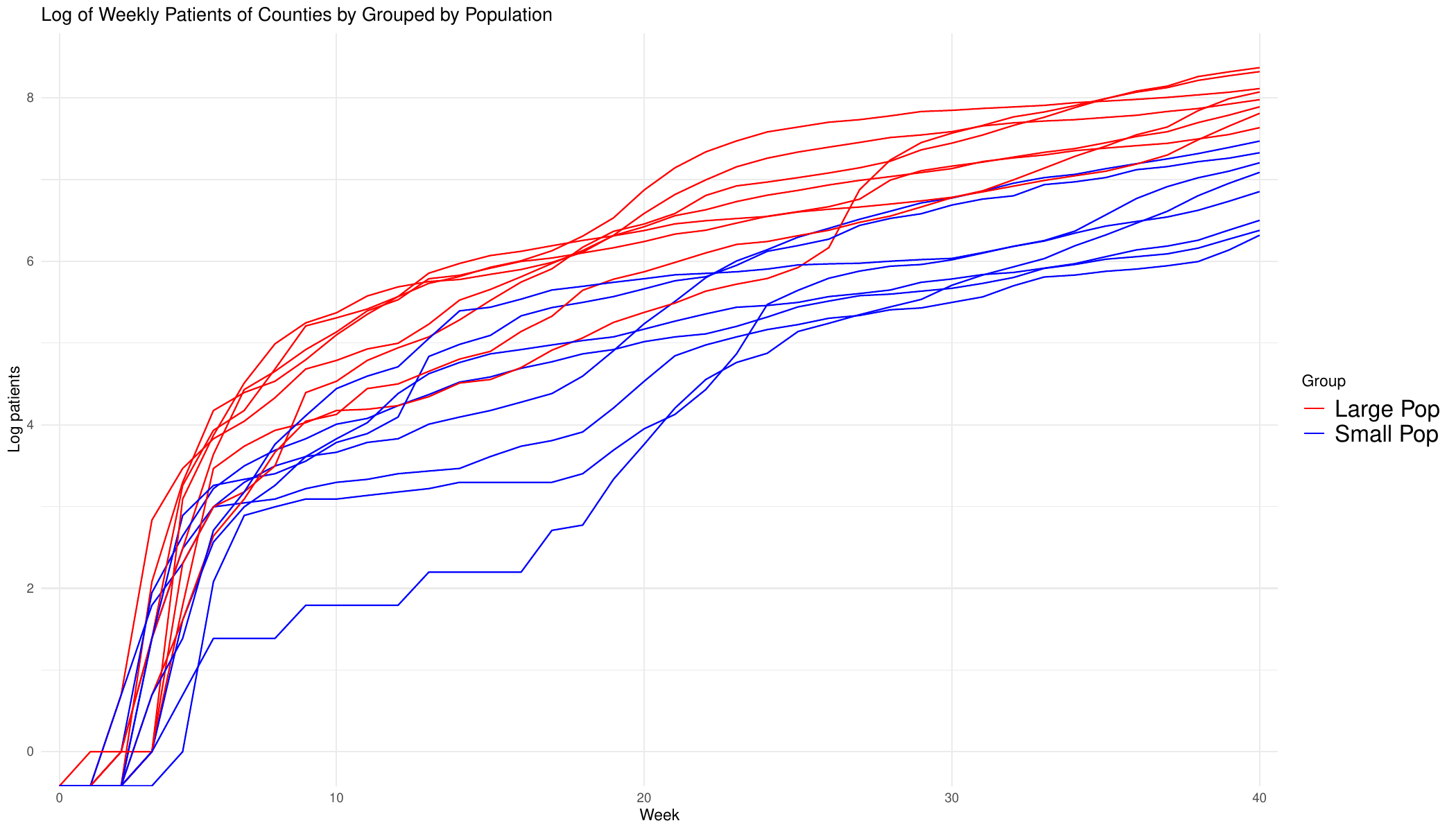}
    \caption{Weekly Accumulated Patients for Counties in VA Grouped by Population}
    \label{gpp_plot}
\end{figure}

Comparing the number of ancestral subjects between different groups of counties is similar to relative quantitation in PCR. For each group, we assume the earliest reported case for the entire state marks the starting point of the pandemic. The accumulated cases in $l^\text{th}$ week for $j^\text{th}$ county is modeled as $Z_{l,j} $. The transmission rate is represented by $m_{l,j} $, which depends on the unobservable environmental variables $\zeta_{l,j}$.

\subsection{Our contributions}

In this paper, we generalize from these two examples by considering an asymptotic framework where the number of generations $n$ diverges to infinity. Since different data structures are available, 
we begin with a description of the data structures studied in the paper. 
Let $\tau$ denote the starting generation number and $n$ denote the final generation number of the observable data.  We allow $\tau$ to depend on $n$, yielding a rich class of data structures.

\begin{assumption}
\label{assum-tau_n}
Let $\tau_n$ denote the initial generation number when the data are observable. Three distinct scenarios describe different sampling regimes: (i) $\tau_n \equiv \tau$,  (ii) the difference $n-\tau_n=\delta_{m}$ is constant, and (iii) $\tau_n \to \infty$ and $n-\tau_n \to \infty $. 
\end{assumption}
\noindent
Case (i) ensures the availability of all data starting from generation $\tau$ till generation $n$ of the BPRE.  By taking $\tau=1$, one obtains all the data from all the $n$ generations of the BPRE. Case (ii) allows observations to be confined to a fixed number of generations within a ``moving window" of length $\delta_{m} $. Case (iii) expands on the concept of the ``moving window" from case (ii), assuming the length of the window also diverges; for instance, $\lim_{n \to \infty}\frac{n-\tau_n}{n}=\frac{1}{2} $. Figure \ref{Assumc1} - Figure \ref{Assumc3} provide examples of these three cases.  In these figures, the green parts represent the observable data from $\tau_n $ to $n$ generations. We emphasize here that cases (i) and (iii) are typically observed in biological/biochemical experiments such as PCR, while case (ii) is typically observed in epidemiological studies and the spread of infectious diseases.

\begin{figure}[ht]
\centering
\begin{minipage}{.3\textwidth}
\centering
\includegraphics[height=4.5cm]{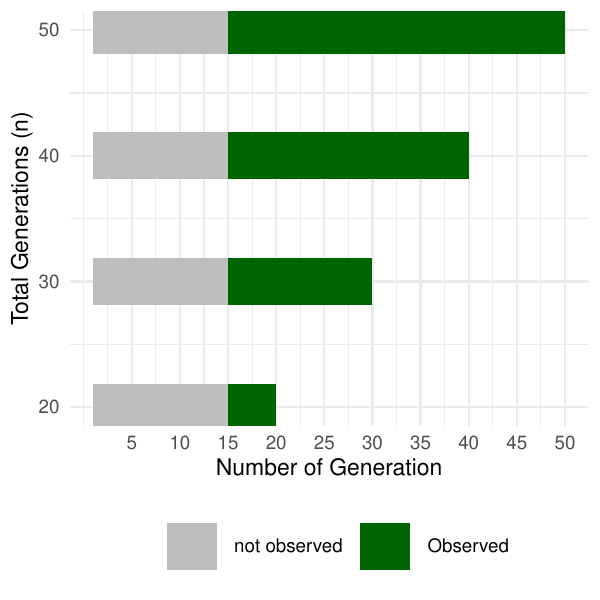} 
\caption{case (i)}
\label{Assumc1}
\end{minipage}\hfill
\begin{minipage}{.3\textwidth}
\centering
\includegraphics[height=4.5cm]{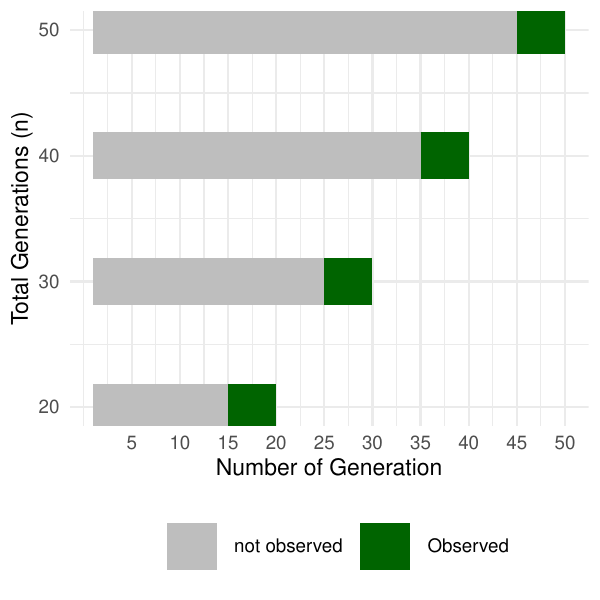} 
\caption{case (ii)}
\label{Assumc2}
\end{minipage}\hfill
\begin{minipage}{.3\textwidth}
\centering
\includegraphics[height=4.5cm]{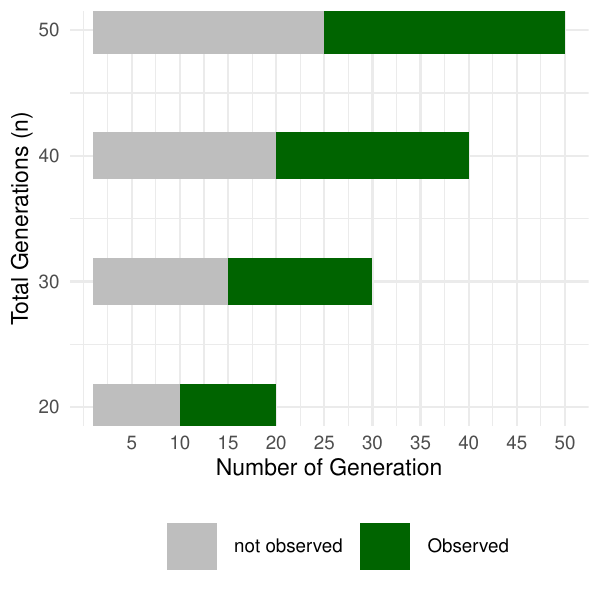} 
\caption{case (iii)}
\label{Assumc3}
\end{minipage}
\end{figure}
This paper is structured as follows: 
Section \ref{sec-mr} presents the main results of the paper concerning the estimator of the ancestral mean. We establish that in all three cases described in sampling scheme \ref{assum-tau_n}, the asymptotic normality holds under the scaling $(J_n\hat{r}^{-2n})^{\frac{1}{2}}$, where $\hat{r}^2$ (which is necessarily larger than one by Jensen's inequality) represents an estimate of the coefficient of variation (CV).
This result is in \emph{sharp contrast} to the GW case where the rate is $J_n^{\frac{1}{2}}$ (see \cite{hanlon2009contributions}). 

The estimate of the ancestral mean, $\hat{m}_{A,n}$, involves an estimate of the marginal offspring mean, $\hat{m}_n$. Theorem \ref{consis-CLT-mo} provides the limiting distribution for the centered  $\hat{m}_n$ holds using the scaling $(J_n(n-\tau_n))^{\frac{1}{2}}$ with the limiting variance coinciding with the anticipated marginal variance.  A natural next question concerns the joint distribution of the $\hat{m}_n$ and $\hat{m}_{A,n}$. As mentioned previously, since the estimate of $m_A$ utilizes an estimate of $m^*$, it is natural to use data from a select few generations to estimate $m^*$ and then estimate $m_A$ using the remaining generations. Now, interpreting in the language of statistical learning, the following question is pertinent: \emph{should one learn $m^*$ from initial or later generations}? We address this issue and establish a difference in the asymptotic variance in Section \ref{sec-slj}. Specifically, we show that learning $m^*$ from later generations yields an estimator of $m_A$ with a smaller asymptotic variance in all three cases considered in sampling scheme \ref{assum-tau_n}. Additionally, we also establish that, as in the GW case, the estimators are asymptotically independent. Applications of these results for quantification in RT-PCR experiments are described in Section \ref{sec-slj}.

The proofs of the main results rely on the rate of convergence of $(\hat{m}m^{*(-1)})^n$, where $m^*$ is the marginal offspring mean. This problem has some general interest. In the case of positive i.i.d. random variables, $X_1, X_2, \cdots X_n$, with expectation $m^*$, the behavior of $(\bar{X}_n/ m^*)^{a_n}$ relies on the finiteness of the moments of $X_1$. However, in the GW and the BPRE cases, the behavior is captured via the behavior of the harmonic moments. These results are studied in Section \ref{sec-other} of the manuscript, while the proofs of the main results are included in Section \ref{sec-pf}. Implementations of our estimators in data examples require an estimate of the limiting variance. We consider several parametric models and derive consistent estimators of the asymptotic variance in Section \ref{sec-var}. In Section \ref{sec-ne}, we provide results from several numerical experiments and a novel Bootstrap algorithm for constructing a confidence interval. Section \ref{sec-da} contains data analysis from two PCR experiments and data from a COVID-19 study.  Section \ref{sec-dis} contains some concluding remarks.

\section{Main Results}
\label{sec-mr}
In this section, we state our main results and their asymptotic properties concerning the ancestral mean $m_A $ and marginal offspring mean $m^*$.

\subsection{Ancestral mean  estimation}
We begin with the observation that, since $\bm{E}[{Z_{n,j}}]= m^{*n} m_A $, a natural estimator of $m_A$ is
\begin{align*}
    \hat{m}_{A}&=\frac{1}{m^{*n}} \frac{1}{J} \sum_{j=1}^{J}  Z_{n,j}  . 
\end{align*}
Next, using an estimator $\hat{m}$ of $m^*$, specified below, a method of moments estimator for $m_A$ is,
\begin{align*}
    \hat{m}_{A,n,J}=\frac{1}{\hat{m}^n} \frac{1}{J} \sum_{j=1}^{J}  Z_{n,j} = \left(\frac{m^*}{\hat{m}}\right)^n \frac{1}{J} \sum_{j=1}^{J}  \frac{Z_{n,j}}{m^{*n}} . 
\end{align*}
While the above estimator only uses data from the $n^{th}$ generation, one can incorporate all available data by replacing $Z_{n,j}$ by $\sum_{l=\tau}^n Z_{l,j}$ yielding
\begin{align}
   \hat{m}_{A, n} \coloneqq \hat{m}_{A,\tau_n,n,J_n}= \frac{1}{\hat{\mathcal{N}}_n} \left( \frac{1}{J_n} \sum_{j=1}^{J_n}   \sum_{l=\tau_n}^n Z_{l,j} \right) = \frac{\mathcal{N}_{n}^*}{\hat{\mathcal{N}}_n}   \left(  \frac{1}{J_n}   \sum_{j=1}^{J_n}  \frac{1}{\mathcal{N}_{n}^*} \sum_{l=\tau_n}^n Z_{l,j} \right) ,   \label{def-ma_taunj}
\end{align}
where $\mathcal{N}^{*}_{n}$ represents the scaling factor and is given by
\begin{align*}
\mathcal{N}_{n}^*=  \frac{m^{*\tau_n}(m^{*(n-\tau_n+1)} -1)}{(m^*-1) } . 
\end{align*}
Using $\hat{\mathcal{N}}_n$ to represent the estimate of $\mathcal{N}_n^*$, the ratio $\mathcal{N}_{n}^*\hat{\mathcal{N}}_{n}^{-1}$ represents the behavior of the ratio of the marginal mean (of the sum $\sum_{l=\tau}^{n} Z_{l,j}$) to its estimate. While several estimators of $m$ are possible, we use  
\begin{align}
\hat{m}_n \coloneqq \hat{m}_{\tau_n,n,J_n}=\frac{1}{J_n} \sum_{j=1}^{J_n}  \frac{1}{n-\tau_n} \sum_{l=\tau_n+1}^n\frac{Z_{l,j}}{Z_{l-1,j}},        
     \label{def-mo_taunj}
\end{align}
which has the intuitive interpretation involving the average number of children produced by the $(l-1)^{\text{th}}$ generation parents. Evidently, the asymptotic properties of (\ref{def-ma_taunj})  will depend on the behavior of $\mathcal{N}_n^{*(-1)} \sum_{l=\tau}^n Z_{l,j}$.  A key observation is that the variance of $\mathcal{N}_n^{*(-1)} \sum_{l=\tau}^n Z_{l,j}$ grows as $n$ increase leading to a comparison of the marginal mean behavior to that of the marginal second moment, namely the CV. This new aspect is peculiar to the BPRE and is not typically encountered in the GW case. Our first main result concerning the asymptotic properties of the estimator in (\ref{def-ma_taunj}), details this issue and investigates its behavior under all three cases described in sampling scheme \ref{assum-tau_n}. Before we state the Theorem, we need a few more notations. Let 
\begin{align}
\hat{m}_{2,n} =\frac{1}{J_n} \sum_{j=1}^{J_n} \frac{1}{n-\tau_n} \sum_{l=\tau_n+1}^n \left(\frac{Z_{l,j}}{Z_{l-1,j}} \right)^2,   \quad \text{and} ~~   \hat{r}_{n} = \sqrt{\frac{\hat{m}_{2,n}}{\hat{m}_{n}^2} }  \label{def-r_taunJ}
\end{align}
denote the estimates of the marginal second moment and the square root of the CV. Let $r^{*2}\coloneqq m_2^*/m^{*2}$ denote the coefficient of variation. Let
\begin{align*}
\mathfrak{D}=    \frac{ m_A\gamma^{2*}}{1-(r^{*2}m^*)^{-1}}  +     \frac{(m_A^2+\sigma_A^2) \sigma^{2*}}{1-r^{*(-2)}}, \;\; \sigma_{I}^2 \coloneqq \frac{\mathfrak{D}(m^*-1)^2(m_{2}^*+m^*)}{m^{2*}(m_{2}^*- m^*)(m_{2}^*-1) }, ~\text{and}
\end{align*}
\begin{align*}
\sigma_{F}^2(\delta) \coloneqq  \frac{\mathfrak{D}(m^*-1)m^{*(2\delta+2)} }{m_{2}^{*(\delta+1)}(m^{*(\delta +2)} -m^*)^2} \left( \frac{2m^* m_{2}^{*(\delta+1)} - 2m^{*(\delta+2)}}{ m_{2}^*-m^*  }  - \frac{(m^*+1)(m_{2}^{*(\delta+1)}-1 ) }{m_{2}^*-1} \right). 
\end{align*}
Also, recall that $\delta_m$ is the constant $(n-\tau_n)$ described in case (ii) in the sampling scheme. 
\begin{theorem}
\label{consis-CLT4mataunJ}
Let \ref{AssumA1}-\ref{Assum-gamma} hold. Additionally, assume that $\frac{n}{ \log J_n} \to 0 $ as $ n \to \infty$, and $\tau_n  $ satisfies one of the cases (i)-(iii) in sampling scheme \ref{assum-tau_n}. Then, $\hat{m}_{A,n}  \xrightarrow[]{p} m_A$. Furthermore, if \ref{Assum-mo4+2delta} holds, then   $        \sum_{l=0}^{n-1} \bm{E}[Z_l^{-1}] \coloneqq \lambda_n$ converges to $\lambda $ as $n \to \infty$ and 
\begin{align*}
      \frac{1}{\hat{r}_{n}^{n}} \sqrt{J_n} ( \hat{m}_{A,n} -m_A ) \xrightarrow[]{d} N(0,  \sigma^2_{\tau} ), ~\text{where},
 \end{align*}
\begin{align}
\label{def-sigma_tau2}
    \sigma^2_{\tau}=
    \begin{cases}
           \exp \left(\frac{(\lambda-\lambda_\tau)\gamma^{2*}}{m_{2}^*}   \right) \sigma_{I}^2 \quad &\text{Under case (i)}  \\
        \sigma_{F}^2(\delta_{m}) \quad &\text{Under case (ii)}\\
        \sigma_{I}^2  \quad &\text{Under case (iii)}
    \end{cases}
\end{align}
and $\sigma_{I}^2 $ and $ \sigma_{F}^2(\delta) $ are positive and finite. 
\end{theorem} 
\begin{remark} 
\label{remark-asy-var}
The formulae for the limiting variances in all three cases depend on the parameters of the marginal distribution.  The notation $\sigma^2_F$ represents the limiting variance based on a fixed number of generations in the estimation of $m_A$, whereas $\sigma^2_I$ represents the case when the number of generations involved in the estimate diverges. Below, in Lemma \ref{var-sum-zn}, we will show that $\sigma^2_F(\delta_m)$ increases to $\sigma^2_I$ as $\delta_m \nearrow \infty$. Using this and that $\lambda >\lambda_{\tau}$, it follows that 
\begin{align}
\label{mataunJ-limit-var}
    \exp \left(\frac{(\lambda-\lambda_\tau)\gamma^{2*}}{m_{2}^*}   \right) \sigma_{I}^2 > \sigma_{I}^2  > \sigma_{F}^2(\delta_{m}).
\end{align}
We notice that the limiting variance is minimized when a finite number of generations are involved in the estimator of $m_A$ (case (ii)). However, two pertinent issues arise, namely, the effect on the estimation of $m^*$ and on the learning that we study in Section \ref{sec-slj} below.
\end{remark}

We now turn to the asymptotic behavior of the estimator for $m^*$. Since for any $l,j$, $\bm{E}[ Z_{l+1,j}  / Z_{l,j}] =m^*$, $\hat{m}_{n}$ in (\ref{def-mo_taunj}) is a method of moments estimator of $m^*$. Our next Theorem concerns the consistency and asymptotic normality of $\hat{m}_{n} $ in all three cases described in the sampling scheme \ref{assum-tau_n}.
\begin{theorem}
\label{consis-CLT-mo}
Let \ref{AssumA1}-\ref{Assum-gamma} hold and $ \tau_n $ satisfy one of the cases (i)-(iii) in sampling scheme \ref{assum-tau_n}. If $J_n(n-\tau_n) \to \infty$, then 
$\hat{m}_{n}\xrightarrow[]{p} m^* $ and
\begin{align*}
 \sqrt{J_n(n-\tau_n)} (\hat{m}_{n} - m^*) \xrightarrow[]{d} N(0, \sigma^{2*})    . 
\end{align*} 
Moreover, if $J_n(n-\tau_n)=O(n^{\alpha})$ for some $\alpha>1$, then $\hat{m}_{n}\to m^* $ w.p.1.
\end{theorem}
The estimator $\hat{m}_{n} $ and analysis differ from those in \cite{Li-Vidya2019} where they obtain the estimator by dividing ``the total number of children by the total number of parents''; that is, 
\begin{align}
  \hat{m}_{\tau,n,J}^{(0)}  =\frac{\sum_{j=1}^{J}  \sum_{l=\tau+1}^n Z_{l,j}  }{\sum_{j=1}^{J}  \sum_{l=\tau+1}^n Z_{l-1,j} }. \label{mo_estld}
\end{align}
Note that we can rewrite (\ref{mo_estld}) as 
\begin{align}
 \hat{m}_{\tau,n,J}^{(0)}  
  &= \sum_{j=1}^{J}   w_j^{(B)}\sum_{l=\tau+1}^n   w_l^{(W)} \frac{Z_{l,j}}{Z_{l-1,j}}, \quad \text{where}  \label{mo_estld-2}    
\end{align}
\begin{align*}
   w_j^{(B)}= \frac{ \sum_{l=\tau+1}^n Z_{l-1,j} }{\sum_{j'=1}^{J}  \left( \sum_{l=\tau+1}^n Z_{l-1,j'} \right) }  \quad \text{and} \quad
    w_l^{(W)} =\frac{Z_{l-1,j}}{\sum_{l'=\tau+1}^n Z_{l'-1,j}}.
\end{align*}
In (\ref{mo_estld-2}), note that $  w_j^{(B)}$ represents the \emph{between replicate weight} and $ w_l^{(W)} $ represents the \emph{ within replicate weight}. 
However, our approach assigns equal weight to each ratio of children to parents across all generations and replicates, implying that $ w_j^{(B)} \equiv \frac{1}{J} $ and $w_l^{(W)} \equiv \frac{1}{n-\tau}$. We will demonstrate that this ``equal between and within replicate weight" guarantees the asymptotic properties of $\hat{m}_{A,n}$.
Furthermore, while (\ref{mo_estld}) assumes $\tau $ and $n$ to remain constant, the asymptotic properties of our estimator can be derived when the starting point $\tau_n $ is subject to different conditions depending on the number of generations  $n $, which is also allowed to diverge to infinity.

\subsection{Other data structure}

Before turning to the proof of the asymptotic properties of $\hat{m}_{A,n}  $, we first study the simple data structure where only the last two generations $(Z_{n-1,j}, Z_{n,j})$ are observable for each replicate. This is common in many real-world problems.  Under this data structure, our estimators for $m_A $ and $m^*$ are
\begin{align}
     \tilde{m}_{A,n}= \frac{1}{J_n} \sum_{j=1}^{J_n} \frac{Z_{n,j}}{{ \tilde{m}_{n}^{n}}} ,  \quad \tilde{m}_{n}=\frac{1}{J_n} \sum_{ j=1 }^{J_n}   \frac{ Z_{n,j}  } {Z_{n-1,j}}   \label{def-m1}.
\end{align}
Notice that although $\tilde{m}_{n} = \hat{m}_{n-1,n,J_n} $, we only use the last generation to estimate $m_A $, which is different from the estimator $\hat{m}_{A,n-1,n,J_n} $. The next theorem illustrates the asymptotic properties of $ \tilde{m}_{A,n}$, and its proof is a corollary to the proof of Theorem \ref{consis-CLT4mataunJ}.
\begin{theorem}
\label{consis-CLT4ma}
Let\ref{AssumA1}-\ref{Assum-gamma} hold and assume that $\frac{n}{ \log J_n} \to 0 $. Then, $\tilde{m}_{A,n} \xrightarrow[]{p} m_A$ as $ n \to \infty$.
Moreover, if \ref{Assum-mo4+2delta} holds, then
\begin{align} 
     \frac{1}{\tilde{r}_{n}^n} \sqrt{J_n} ( \tilde{m}_{A,n} -m_A ) \xrightarrow[]{d} N(0,  \frac{\mathfrak{D} }{m_{2}^*} ), \quad \text{where} \label{ma-clt}
\end{align}
\begin{align}
\tilde{m}_{2,n}= \frac{1}{J_n} \sum_{ j=1 }^{J_n}    \left(  \frac{ Z_{l,j}  } {Z_{l-1,j}}    \right)^2  , \quad  \text{and}~~ \tilde{r}_{n}^2=\frac{\tilde{m}_{2,n}}{  \tilde{m}_{n}^2}  . \label{def-m_21-r1}
\end{align}
\end{theorem} 
\noindent In Section \ref{sec-var} below we provide estimators for the variance of the ancestor ($\sigma_A^2$) under this data structure. We will prove their consistency under the assumptions \ref{AssumA1}-\ref{Assum-mo4+2delta}.

The asymptotic properties of  
$m^{*(-n)} Z_{n,1} $ rely on the $Var(Z_{n,1})$ which 
by Proposition 1 in \cite{Li-Vidya2019} is given by
\begin{align}
   \bm{Var} (Z_{n,1})=m_{2}^{*(n-1)} \mathfrak{D}_n + m^{*2n} \sigma_A^2, \quad{where} \label{varzn1}
\end{align}
\begin{align}
\mathfrak{D}_n= m_A\gamma^{2*} \left( \frac{1-(r^{*2}m^*)^{-n}}{1-(r^{*2}m^*)^{-1}} \right) +   (m_A^2+\sigma_A^2) \sigma^{2*} \left( \frac{1-r^{*(-2n)}}{1-r^{*(-2)}} \right), \text{ and }
  \mathfrak{D}_n \nearrow \mathfrak{D}  . \label{varzn1-D}
\end{align}  
This implies that variance of $m^{*(-n)} Z_{n,1} $  diverges like $r^{*n} $. 

The scaling factor $r^{*n} $ captures the variability introduced by the random environment $\zeta_{l,j} $ 
 and differs from the GW process model studied in \cite{Bret-Vidya2011} where  $r^{*2}= \frac{m_{2}^*}{m^{*2}} =1 $. Consequently, the scaling factor does not influence the estimation process in the GW model. 

This difference underscores one of the main challenges in establishing the asymptotic properties of our estimator within the BPRE framework. Specifically, when applying an estimator for $r^* $ in the BPRE model, unlike the case for the marginal offspring mean as stated in Theorem \ref{mo^n to 1}, the ratio $(r^{*(-1)} \hat{r}_{\tau_n,n,J_n})^n$ does not converge to 1 but to a fixed value that needs to be determined depending on the behavior of $\tau_n$. This is carried out in Theorem \ref{ro^n to 1}. We now turn to characterizing the asymptotic variance alluded to in Theorem \ref{consis-CLT4mataunJ}.
\begin{lemma}
\label{var-sum-zn}
Let \ref{AssumA1}-\ref{Assum-mo4+2delta} hold and let $\frac{n}{\log J_n}\to 0 $. Then, 
    \begin{align*}
    \bm{Var} \left[\frac{1}{r^{*n}} \frac{1}{\mathcal{N}_{n}^*} \sum_{l=\tau_n}^n Z_{l,j} \right] \to \begin{cases}
        \sigma_{I}^2 , \quad  & \text{under cases(i) and (iii)} \\
        \sigma_{F}^2(\delta_{m}) &  \text{under cases(ii)} 
    \end{cases} ,
\end{align*}
where $ \sigma_{I}^2$ and $\sigma_{F}^2(\delta) $ is specified in Remark \ref{remark-asy-var}. Furthermore, as $\delta_{m} \to \infty$, $\sigma^2_F(\delta_{m}) \to \sigma_{I}^2$.  $\sigma^2_F(0)= \frac{\mathfrak{D}}{m_2^*} $, which is the asymptotic variance of $\tilde{m}_{A,n} $.
\end{lemma}
The details of the proof of this lemma are outlined in Appendix \ref{apd-c-var-sum-zn}.

\begin{remark}[Comparison to GW case] \label{compareGW}
    Under the GW process, the offspring distributions are i.i.d., implying that $m^{*2}=m_2^* $ and $\sigma^{2*}=0 $. Consequently, this yields $\sigma_{I}^2=\frac{\mathfrak{D}}{m^{2*}} $ and $\sigma_{F}^2(\delta) =\frac{\mathfrak{D}}{m^{2*}}$ for any $\delta $. Referring to (\ref{varzn1-D}), the limit of $\mathfrak{D}_n $ will be $\frac{ m_A\gamma^{2*} m^* }{m^*-1}$. Additionally, because $r^*=1 $, we no longer need to estimate it, and therefore, the exponential term described in Case 1 of (\ref{def-sigma_tau2}) is omitted. According to (\ref{varzn1}), the limiting variance is $\frac{ m_A\gamma^{2*}  }{m^*(m^*-1)}+ \sigma_A^2 $, which corresponds to the limiting variance of Theorem 1 in \cite{Bret-Vidya2011}.
   This comparison highlights that, unlike the GW process, the limiting variance for BPRE depends on an additional factor, $r^* $.
\end{remark}

\noindent
\section{Statistical Learning, Joint Distributions, and Applications }
\label{sec-slj}
We now turn to the joint distribution of $\hat{m}_{A,n} $ and $ \hat{m}_n$. Motivated by results in Section \ref{sec-mr}, we cast this problem in terms of learning and divide the observations within a BPRE into two parts: one for estimating $m_A$ and the other for estimating $m^*$. This leads to the question: can one learn the parameters of the offspring distribution first and then use that to estimate the ancestor mean? It is to be noted here that the process evolves as a non-homogeneous branching process, conditioned on the environmental sequence.  Alternatively, heuristically speaking, it may be more ``efficient'' to use the initial generation to estimate the ancestor mean since more information may be available during those generations. It turns out that as the population is growing exponentially fast, choices of generations for estimating the parameters become an important issue.
Specifically, we establish that the rates of estimation for $\hat{m}_{A,n}$ and the limiting variances can be different depending on the sampling scheme. This is the content of our next two results.
Before we state the result, we need several notations. Let $\tau_{1,n}<\tau_{2,n}<n$ and set $\delta_{m,n}=\tau_{2,n}- \tau_{1,n} $ and $\delta_{m_A,n}=n- \tau_{2,n} $. Define
\begin{align*}
\hat{m}_{\tau_1,\tau_2,n}=\frac{1}{J_n(\tau_2-\tau_1)} \sum_{j=1}^{J_n} \sum_{l=\tau_1+1}^{\tau_2}\frac{Z_{l,j}}{Z_{l-1,j}}, ~~\text{and}~~
    \hat{m}_{A,\tau_1,\tau_2,n}=  \frac{1}{J_n \hat{\mathcal{N}}_{\tau_1,\tau_2,n}} \sum_{j=1}^{J_n}  \sum_{l=\tau_2}^n Z_{l,j} , 
\end{align*}
where to streamline the notation, we have suppressed $n$ in $\tau_{1,n} $ and $\tau_{2,n} $ to $\tau_1 $ and $\tau_2 $ respectively. Thus, the marginal offspring mean is estimated using data from $\tau_{1}^{\text{th}}$ to $\tau_{2}^{\text{th}}$ generation. In contrast, the ancestor mean is estimated using data from $\tau_{2}^{\text{th}} $ to $n^{\text{th}}$ generation.   
We will assume that the proportion of generations used to estimate $ m^*$, namely, $\frac{\delta_{m,n}}{n-\tau_{1,n}} $, converges to a fixed positive constant as $n \to \infty$.
As before,
\begin{align*}
\hat{\mathcal{N}}_{\tau_1,\tau_2,n}&= \frac{ \hat{m}_{\tau_1,\tau_2,n}^{\tau_2}(\hat{m}_{\tau_1,\tau_2,n}^{n-\tau_2 +1} -1)     } {\hat{m}_{\tau_1,\tau_2,n}-1   }.
\end{align*}
\begin{theorem}
\label{joint-CLT}
Let \ref{AssumA1}-\ref{Assum-mo4+2delta} hold and that $\frac{n}{ \log J_n} \to 0 $. $\tau_{2,n}-\tau_{1,n} \coloneqq \delta_{m,n} $ and $n-\tau_{2,n} \coloneqq \delta_{m_A,n} $, and $\frac{\delta_{m,n}}{n-\tau_{1,n}} \to p_{m} $, $\frac{\delta_{m_A,n}}{n-\tau_{1,n}} \to p_{m_A} $, where $p_{m} $ and $p_{m_A} $ are positive constants. Then,
\begin{align*}
\begin{pmatrix} 
\sqrt{(n-\tau_{1,n})J_n} & 0 \\
0& \frac{1}{\hat{r}^n} \sqrt{J_n}  
 \end{pmatrix}
      \begin{pmatrix}
     \hat{m}_{\tau_1,\tau_2,n} -  m^* \\
    \hat{m}_{A,\tau_1,\tau_2,n} - m_A
    \end{pmatrix}
    \xrightarrow[]{d}
    N \left( \begin{pmatrix} 
    0 \\ 0 \end{pmatrix} ,
     \begin{pmatrix}
    \frac{\sigma^{2*}}{p_{m} } & 0 \\
    0 &  \sigma^2_{\tau}
    \end{pmatrix}
    \right)
\end{align*}    
and 
\begin{align*}
    \sigma^2_{\tau}=
    \begin{cases}
           \exp \left(\frac{(\lambda-\lambda_\tau)\gamma^{2*}}{m_{2}^*}   \right) \sigma_{I}^2 \quad &\text{if case (i) is true}  \\
        \sigma_{F}^2(p_{m_A} \delta_{m}) \quad &\text{if case (ii) is true}\\
        \sigma_{I}^2  \quad &\text{if case (iii) is true}
    \end{cases}.
\end{align*}
\end{theorem}
We now turn to the alternate case described previously and accordingly use data from    $\tau_{2}^{\text{th}} $ to $n^{\text{th}}$  for estimating $m^*$ while data from $\tau_{1}^{\text{th}}$ to $\tau_{2}^{\text{th}}$ generations for estimating $m_A$. We use prime to denote the estimators in this case. That is,
\begin{align*}
\hat{m}_{\tau_1,\tau_2,n}^{\prime} &=\frac{1}{J_n} \sum_{j=1}^{J_n} \frac{1}{n-\tau_2} \sum_{l=\tau_2+1}^{n}\frac{Z_{l,j}}{Z_{l-1,j}} , \quad \hat{m}_{A,\tau_1,\tau_2,n}^{\prime} =    \frac{1}{J_n \hat{\mathcal{N}}_{\tau_1,\tau_2,n}^{\prime}} \sum_{j=1}^{J_n} \sum_{l=\tau_1}^{\tau_2} Z_{l,j} ,  
\end{align*}
and the corresponding estimate of the mean of $\sum_{j=1}^{J_n} \sum_{l=\tau_1}^{\tau_2} Z_{l,j}$ is
\begin{align*}
\hat{\mathcal{N}}_{\tau_1,\tau_2,n}^{\prime} &= \frac{ \hat{m}_{\tau_1,\tau_2,n}^{\prime \tau_1}( \hat{m}_{\tau_1,\tau_2,n}^{\prime \tau_2-\tau_1 +1} -1)     }{ \hat{m}_{\tau_1,\tau_2,n}^{\prime}-1   } .
\end{align*}
The next Theorem describes our result under this setting.
\begin{theorem}
\label{joint-CLT-2}
     Under \ref{AssumA1}-\ref{Assum-mo4+2delta}, if $\frac{n}{ \log J_n} \to 0 $, $\tau_{2,n}-\tau_{1,n} \coloneqq \delta_{m_A,n} $ and $n-\tau_{2,n}\coloneqq\delta_{m,n} $,  and $\frac{\delta_{m,n}}{n-\tau_{1,n}}\to p_{m} $,  $\frac{\delta_{m_A,n}}{n-\tau_{1,n}}\to p_{m_A} $, where $p_{m} $ and $p_{m_A} $ are constants, we have
\begin{align*}
\begin{pmatrix} 
\sqrt{(n-\tau_{1,n})J_n} & 0 \\
0& \frac{1}{\hat{r}^{\tau_{2,n}}} \sqrt{J_n}  
 \end{pmatrix}
      \begin{pmatrix}
    \hat{m}_{\tau_1,\tau_2,n}^{\prime} - m^* \\
  \hat{m}_{A,\tau_1,\tau_2,n}^{\prime}    -m_A
    \end{pmatrix}
    \xrightarrow[]{d}
    N \left( \begin{pmatrix} 
    0 \\ 0 \end{pmatrix} ,
     \begin{pmatrix}
     \frac{\sigma^{2*}}{p_{m} } & 0 \\
    0 & \sigma^2_{\tau'}
    \end{pmatrix}
    \right)
\end{align*}    
and 
\begin{align*}
    \sigma^2_{\tau'}=
    \begin{cases}
      \sigma_{I}^2 \quad &\text{if case (i) is true} \\
        \sigma_{F}^2(p_{m_A} \delta_{m}) \quad &\text{if case (ii) is true }  \\
        \sigma_{I}^2  \quad &\text{if case (iii) is true}
    \end{cases}.
\end{align*}
\end{theorem}

\begin{remark}
\label{learning-benefit}
   We notice here that the limiting covariance matrix in both Theorem \ref{joint-CLT} and Theorem \ref{joint-CLT-2} is diagonal; that is, the limit distributions are asymptotically independent. Comparing the covariance matrices of Theorem \ref{joint-CLT} and \ref{joint-CLT-2}, the difference is in the values of $\sigma^2_{\tau}$ and $\sigma^2_{\tau'}$. From the proof, it will follow that when we use the later generations for estimating $m^*$ in Theorem \ref{joint-CLT-2}, the estimator behaves as if it is case (iii) in Theorem \ref{joint-CLT}. The reasoning behind this difference, in intuitive terms, is that the use of later generations in case (i) of Theorem \ref{joint-CLT-2} behaves like case (iii) of Theorem Theorem \ref{joint-CLT}. Turning to the rate, we notice that the scaling factor is different for estimating $m_A$. Due to the exponential growth of $Z_{l,j}$, the marginal growth rate behaves like $r^{*\tau_2}$.
\end{remark}

\subsection{Relative Quantitation for PCR}

As previously discussed, relative quantitation involves comparing the amounts of two sets of genetic materials: the calibrator and the target. We denote by $Z_{n,j,C}$ for amount of calibrator in the $n^{\text{th}}$ cycle and the $j^{\text{th}}$ replicate. The initial number is $ m_{A,C}$. Similarly, let $Z_{n,j,T}$  and $m_{A,T} $  denote the respective quantities for the target material. Our primary interest lies in estimating the ratio $R=\frac{m_{A,T}}{m_{A,C}}$. Since only the fluorescence intensities $F_{n, j, T}$ and $F_{n, j, C}$ are observed, we adopt the convention (as in \cite{Bret-Vidya2011} that $Z_{n, j, T}= c F_{n, j, T}$ where $c$ is a constant. We estimate $R$ using
\begin{align}
    \hat{R}_{n}=\frac{\hat{m}_{A,n,T}}{\hat{m}_{A,n,C}} \label{r_est},
\end{align}
where, $ \hat{m}_{A,n,T}$  and $\hat{m}_{A,n,C} $ is defined following (\ref{def-ma_taunj}). Our next result concerns the asymptotic behavior of the relative quantitation parameter, $\hat{R}_{n}$, under the BPRE framework. For notational simplicity, we denote $ \hat{m}_{A,n,T}$ and $\hat{m}_{A,n,C} $ by $\hat{m}_{A,T} $ and $\hat{m}_{A,C} $. Also, let $r_T^*$ and $r_C^*$ denote the CV of the target and the calibrator.
\begin{theorem}
\label{relative-theorem}
Under \ref{AssumA1}-\ref{Assum-mo4+2delta}, if $\frac{n}{\log J_n} \to 0 $,
\begin{align*}
 \frac{1}{r_{T}^{*n}+r_{C}^{*n}}  \sqrt{J_n} (\hat{R}_{n}-R) 
 \xrightarrow[]{d} \begin{cases}
     N\left( 0,  \frac{R^2 \sigma_{\tau_{T}}^2 }{m_{A,T}^2 }\right) \quad &\text{if } r_{T}^*>r_{C}^*\\
     N\left( 0,  \frac{R^2 \sigma_{\tau_{C}}^2 }{m_{A,C}^2 }\right) \quad &\text{if } r_{T}^*< r_{C}^*\\
      N \left( 0, \frac{R^2}{4}\left(\frac{\sigma_{\tau_{T}}^2 }{m_{A,T}^2}+\frac{\sigma_{\tau_{C}}^2 }{m_{A,C}^2 } \right) \right) \quad &\text{if }  r_{T}^*=r_{C}^*
 \end{cases} ,
\end{align*}
where $\sigma_{\tau_{T}}^2 $ and $\sigma_{\tau_{C}}^2 $ are defined in (\ref{def-sigma_tau2}).
\end{theorem}
The details of the proof of this theorem are outlined in Appendix \ref{apd-C-relative-theorem}.
\begin{remark}
Here, we observe that the limiting variance is determined by  $\max(r_{T}^*,r_{C}^*) $. This indicates that the group with greater variability in the offspring distribution plays a dominant role in governing the overall variance.
    In GW process, since $r_{T}^*=r_{c}^*=1 $, and as detailed in  \ref{compareGW}, the limiting variance converges to $ N \left( 0, \frac{R^2}{4}\left(\frac{\sigma_{\tau_{T}}^2 }{m_{A,T}^2}+\frac{\sigma_{\tau_{C}}^2 }{m_{A,C}^2 } \right) \right) $, where $\sigma_{\tau_{T}}^2 $ and $\sigma_{\tau_{C}}^2 $ has the same form in Remark \ref{compareGW}. This expression for the limiting variance also aligns with the result of Theorem 2 in \cite{Bret-Vidya2011}. 
\end{remark}

\subsection{Dynamics of COVID-19 growth}
Turning to the dynamics of COVID-19 evolution, as described in the introduction, one can estimate the mean initial number of potentially infected subjects in communities of similar sizes. In this case, setting $Z_{n, j_C,C}$ and $Z_{n, j_T, T}$ as the number of infected patients in county $j_C$ from state $C$ and county $j_T $ from state $T$ during week $n$ and $Z_{n, j_T, T}$,
one can estimate the mean initial number of infected subjects that potentially initiated the disease. 
Using Theorem \ref{relative-theorem}, one can obtain the standard error of the estimate and confidence intervals for the ratio and compare the policies for stopping the spread work.

\section{Probability Bounds and Convergence Rates}
\label{sec-other}
In this section, we describe several required results in the proofs of the main Theorems described in Section \ref{sec-mr} above and are of independent interest.
Throughout the rest of the paper, we will use the notation $\mathcal{F}_{n,j} $ to denote the  $\sigma$-field generated by $(\zeta_{0,j},\zeta_{1,j},\cdots,\zeta_{n-1,j}, Z_{0,j},\cdots,Z_{n,j}) $. Our first result concerns the harmonic moments of generation sizes of the BPRE.  In the following, we will always assume $\tau_n $ satisfies one of the cases in sampling scheme \ref{assum-tau_n}.
\begin{lemma}
\label{sum1/zn} 
Under \ref{AssumA1} and \ref{AssumA2} there exist positive constants $0 < c_H < \infty $ and $0<a<1$ such that,
  \begin{align*}
      \bm{E}\left[  \frac{1}{Z_n} \right] \le c_H a^n \text{ for all } n \ge 1 \text{ and } \; \sum_{l=0}^\infty \left[ \bm{E}  \left( \frac{1}{Z_l} \right) \right]^p  < \infty \text{ for all }p>0.
  \end{align*}
Specifically, when $p=1 $ recalling $\lambda_n \coloneqq \sum_{l=0}^{n-1} \bm{E}[Z_n^{-1}]$,   $  \lim_{n \to \infty }  \lambda_n < \lambda< \infty $.
\end{lemma}

\begin{proof} For  $Z_0=1 $, let $P(Z_{1}=1|Z_{0}=1)=\gamma_1$, and we let $r_1 $ be the solution (which exists under \ref{AssumA2}) of the equation  $\bm{E}[ m_{0,1}^{-r_1}]=\gamma_1$ . Let
\begin{align*}
    A_{1,n}=
    \begin{cases}
        \gamma_1^n \quad &\text{if } r_1<1 \\
        n \gamma_1^n \quad &\text{if } r_1=1\\
        [\bm{E}(m_{0,1}^{-1})]^n\quad &\text{if } r_1>1
    \end{cases}.
\end{align*}
Then by Theorem 2.1 in \cite{Grama-Liu2017}, we have 
\begin{align*}
   \lim_{n \to \infty} \frac{\bm{E}(Z_n^{-1}|Z_0=1)}{ A_{1,n}} \to c_{1}< \infty.
\end{align*}
Because both $\gamma_1<1 $ and $\bm{E}(m_{0,1}^{-1})<1 $, we can always find a constant $0< a < 1$, such that $  \frac{A_{1,n} }{ a^n } \to 0$, and
\begin{align*}
      \lim_{n \to \infty} \frac{\bm{E}(Z_n^{-1}|Z_0=1)}{ a^n} \to 0.
\end{align*}
This implies that there exists a constant $c_H $ such that, 
 \begin{align*}
     \bm{E} \left[  \frac{1}{Z_n}  \middle|Z_0=1 \right] \le  c_H a^n , \;\;\text{for all $n$}.
  \end{align*}
Meanwhile, since $Z_n= \sum_{l=1}^{Z_0} Z_{n,l}^{\prime}  $, where each $Z_{n,l}^{\prime} $ is a BPRE starting with a single ancestor, it follows that for any $k>1$, $\bm{E}[Z_n^{-1}|Z_0=k ] \le \bm{E}[Z_n^{-1} | Z_0=1] $, implying   
 \begin{align*}
     \left[ \bm{E} \left(  \frac{1}{Z_n} \right) \right]^{p} &=    \left[ \sum_{k=1}^\infty \bm{E} \left[  \frac{1}{Z_n} \middle|Z_0=k  \right] \bm{P}(Z_0=k) \right]^{p} \\ 
     &\le \left[ \sum_{k=1}^\infty \bm{E} \left[  \frac{1}{Z_n}  \middle|Z_0=1 \right] \bm{P}(Z_0=k) \right]^{p} =\left[ \bm{E} \left[  \frac{1}{Z_n}  \middle|Z_0=1 \right] \right]^{p} \le  ( c_H a^n)^{p} , \;\; \text{for all $n$}.
  \end{align*}
 let $p=1$, we have $ \bm{E}\left[  Z_n^{-1} \right] \le c_H a^n $ for all $n$. And for any $p>0$, 
  \begin{align*}
       \sum_{l=0}^\infty \left[ \bm{E}  \left( \frac{1}{Z_l} \right) \right]^{p} \le   \sum_{l=0}^\infty( c_H a^n)^{p}   < \infty .
  \end{align*}
\end{proof}

\noindent
Let $Y_{l,j}\coloneqq \frac{Z_{l+1,j}}{Z_{l,j}}-m_{l,j} $, the next theorem tells that the variance of the infinite sum converges:
\begin{lemma}
\label{varsumYl-finite}
Assume that the conditions of Lemma \ref{sum1/zn} 
 hold. Then as $n \to \infty $, the following holds:
\begin{align*}
   & \limsup_{n \to \infty} \bm{Var} \left[   \sum_{l=\tau_n+1}^{n}   Y_{l,j} \right]  \le c_\tau, \quad  \text{if $\tau_n \equiv \tau$ is fixed (case (i) in sampling scheme \ref{assum-tau_n}), and}\\
   & \lim_{n \to \infty}  \bm{Var} \left[   \sum_{l=\tau_n+1}^{n}   Y_{l,j} \right] =0,  \quad  \text{if $\tau_n \to \infty$  (case (ii) and (iii) in sampling scheme \ref{assum-tau_n})}.
\end{align*}
Furthermore,  $0 < c_\tau \le  \frac{\gamma^{*2} c_H a^{\tau}  }{1-a} $.
\end{lemma}

\begin{proof}
Let the generic element of $\{\frac{ Z_{l+1,j}  } {Z_{l,j}} \}$, $\{\xi_{l,n,j}  \}$, $\{ \sigma^2_{n,j} \} $ and $\{ m_{n,j} \}$ be denoted by $\frac{ Z_{l+1}  } {Z_{l}}$, $\xi_{l,n} $, $\sigma^2_{n}$ and $m_{n}$. Let $Y_{l,j} =     \frac{ Z_{l+1,j}  } {Z_{l,j}} - m_{l,j}  $ and denote its generic element by $Y_l$. Using the conditional independence of $\xi_{l,n} $, we have for any $\tau_n$,
\begin{align*}
\bm{Var}\left[    \sum_{l=\tau_n+1}^n   Y_l  \right]  
    &=  \bm{Var}\left[   \sum_{l=\tau_n+1}^{n-1}   Y_l  +  \bm{E}   [ Y_n | \mathcal{F}_n]  \right]  +   \bm{E}\left[   \frac{1}{Z_n^2} \bm{Var}  [ Z_{n+1} |  \mathcal{F}_n  ] \right]  \nonumber \\
    &=  \bm{Var} \left[   \sum_{l=\tau_n+1}^{n-1}   Y_l \right] + \bm{E}\left[   \frac{1}{Z_n}   \sigma^2_{n} \right] 
    = \bm{Var}\left[    \sum_{l=\tau_n+1}^{n-1}   Y_l  \right] + \gamma^{*2}\bm{E}\left[  \frac{1}{Z_n} \right],
\end{align*}
where the penultimate equality follows by conditioning on $\mathcal{F}_{(n-1)}$. By iteration and using Lemma \ref{sum1/zn}, for some $0<a<1$, it follows that
    \begin{align}
         \bm{Var} \left[   \sum_{l=\tau_n+1}^{n}   Y_{l,j} \right] &=           \gamma^{*2}  \sum_{l=\tau_n+1}^n\bm{E} \left[  \frac{1}{Z_l} \right] 
         \le c_H \gamma^{*2} \sum_{l=\tau_n+1}^n a^l. \label{varsumYl-finite-p1}
    \end{align}
Now, when $\tau_n=\tau$, then by Lemma \ref{sum1/zn} it follows that
\begin{align*}
\bm{Var} \left[\sum_{l=\tau+1}^{\infty} Y_{l,j} \right]=c_H \gamma^{*2}  \sum_{l\ge \tau+1}\bm{E} \left[  \frac{1}{Z_l} \right] \coloneqq c_{\tau}.
\end{align*}
Next, when $\tau_n \nearrow \infty$, 
the last term on the right-hand side (RHS) of (\ref{varsumYl-finite-p1}) converges to 0, being the tail of a convergent sum. Thus, in all cases of the sampling scheme \ref{assum-tau_n} the variance convergence follows. This completes the proof of the Lemma. 
\end{proof}

The next theorem is a key result examining the asymptotic behavior of the ratio $(\hat{m}_{\tau_n,n,J_n}  / {m^*} )^n $. We demonstrate that this ratio converges to 1 exponentially fast, ensuring convergence when raised to the $n^\text{th}$ power. The proof of this theorem requires that $J_n \nearrow \infty$ such that  $\frac{n^2}{J_n(n-\tau_n)}  \to 0$. 
When there is no scope for confusion,
we will suppress $\tau_n$ and $J_n $in the subscripts; for example, $\hat{m}_{\tau_n,n,J_n} $ will be simplified to $ \hat{m}_{n}$:

\begin{theorem}
\label{mo^n to 1}
Let \ref{AssumA1}-\ref{Assum-gamma} hold. Also, assume that for $\tau_n $ in all the cases in sampling scheme \ref{assum-tau_n} that $ \frac{n^2}{J_n(n-\tau_n)} \to 0$ and that $m_2^*=\bm{E}[m_{0,1} ]^2 < \infty $. 
Then, as $n \to \infty$
     \begin{align*}
    \left(\frac{\hat{m}_{n} } {m^*} \right)^n  \xrightarrow[]{p} 1.
\end{align*}   
\end{theorem}
\begin{proof}
We will show that $ n\log \frac{ \hat{m}_{n} } {m^*}  \xrightarrow[]{p} 0 $. Since
\begin{align}
    \bm{P} \left[ \left|  n\log \frac{ \hat{m}_{n} } {m^*}  \right|> \epsilon  \right] 
     =  \bm{P} \left[   \hat{m}_{n}   > m^* \exp \left( \frac{  \epsilon}{n}\right)  \right]  + \bm{P} \left[   \hat{m}_{n}   < m^* \exp \left(- \frac{  \epsilon}{n}\right)  \right] \label{mo^n2parts} ,
\end{align}
it is sufficient to show (\ref{mo^n2parts}) converges to 0. To this end, recall that 
\begin{align*}
  \hat{m}_{n}-m^* 
  &= \frac{1}{J_n} \sum_{ j=1 }^{J_n} \left[   \frac{1}{n-\tau_n}  \sum_{l=\tau_n}^{n-1}     \left( \frac{ Z_{l+1,j}  } {Z_{l,j}}   -m^* \right)  \right] ,
\end{align*} 
by adding and subtracting $m_{l,j} $, we get $  \hat{m}_{n}-m^* = S_{1,n}+S_{2,n}$, where

\begin{align*}
    S_{1,n}=  \frac{1}{J_n}\sum_{ j=1 }^{J_n}  \frac{1}{n-\tau_n}  \sum_{l=\tau_n}^{n-1}     \left( \frac{ Z_{l+1,j}  } {Z_{l,j}}  - m_{l,j} \right) ~~ \text{and}~~ 
    S_{2,n}=   \frac{1}{J_n} \sum_{ j=1 }^{J_n}       \frac{1}{n-\tau_n}  \sum_{l=\tau_n}^{n-1}  \left(   m_{l,j} -m^*  \right) .
\end{align*}
Hence, using $ e^x-1>x$ for $x>0$,
\begin{align}
\bm{P} \left[   \hat{m}_{n}   > m^* \exp \left( \frac{  \epsilon}{n}\right)  \right]  &= 
     \bm{P} \left\{ S_{1,n}+ S_{2,n} > m^*   \left[ \exp \left( \frac{  \epsilon}{n}\right) -1 \right] \right\} \nonumber \\
    & \le       \bm{P} \left[ S_{1,n} \ge \frac{m^*   \epsilon}{2n} \right] + \bm{P} \left[  S_{2,n} >   \frac{m^*   \epsilon}{2n}  \right]  \label{S1S2}.
\end{align}
Consider the first term in (\ref{S1S2}). Using $\bm{E}[Y_{l,j} ]=0 $, Chebyshev's inequality, and Lemma \ref{varsumYl-finite}, it follows that, as $n \to \infty$,
\begin{align}
      \bm{P} \left[ S_{1,n} \ge \frac{m^*   \epsilon}{2n}  \right] 
    & \le \frac{n^2}{J_n (n-\tau_n)^2 (m^*\epsilon)^2 }  \bm{Var} \left[  \sum_{l=\tau_n}^{n-1}     Y_{l,j}\right]   
     \to 0,  \label{S1nJ-conv} 
\end{align}
where the convergence to 0 follows from the condition $ \frac{n^2}{J_n(n-\tau_n)} \to 0$. For the second term, again using Chebyshev's inequality  and the finiteness of $\sigma^{2*}$, it follows that as $n \to \infty$,
\begin{align}
   \bm{P} \left[  S_{2,n} >   \frac{m^*\epsilon}{2n} \right]   
  \le \frac{n^2 }{J_n(n-\tau_n) } \frac{ 4\sigma^{2*}}{(m^*\epsilon)^2} \to 0 \label{S2nJ-conv}.
\end{align}   
Now by (\ref{S1S2}), (\ref{S1nJ-conv}), and (\ref{S2nJ-conv}), the first term on the right-hand side (RHS) of (\ref{mo^n2parts}) converges to zero. A similar calculation also shows that the second term on (\ref{mo^n2parts}) converges to zero. This completes the proof of the Theorem.

\end{proof}

\section{Proofs of Main Results}
\label{sec-pf}

In this section, we provide the proofs of the theorems stated in Section \ref{sec-mr}.
To simplify the notations and when there is no scope for confusion, we will denote $\hat{m}_{\tau_n,n,J_n} $, $\hat{m}_{2,\tau_n,n,J_n}  $, and $\hat{r}_{\tau_n,n,J_n} $ as $\hat{m}_{n} $, $\hat{m}_{2,n}  $ and $\hat{r}_{n} $ respectively. Meanwhile, because $\tilde{m}_{n} $, $\tilde{m}_{2,n}  $ and $\tilde{r}_{n} $ in Theorem \ref{consis-CLT4ma} is a special condition of case (ii) in sampling scheme \ref{assum-tau_n}, all results applicable to $\hat{m}_{n} $, $\hat{m}_{2,n}  $ and $\hat{r}_{n} $ valid for these estimators. Since the proof of Theorem \ref{consis-CLT4mataunJ} requires the asymptotic property of $\hat{m}_{n}$, we begin with the proof of Theorem \ref{consis-CLT-mo}.
We begin by recalling that $Y_{l,j}=Z_{l,j}^{-1}Z_{l+1,j}-m_{l,j}$ and
\begin{align}
    \hat{m}_n - m^*&=\frac{1}{J_n} \sum_{ j=1 }^{J_n} \left(  \frac{1}{n-\tau_n}  \sum_{l=\tau_n}^{n-1}   Y_{l,j} \right) + \frac{1}{J_n} \sum_{ j=1 }^{J_n} \left[  \frac{1}{n-\tau_n}  \sum_{l=\tau_n}^{n-1}  \left( m_{l,j} - m^* \right) \right] \nonumber \\
    &= S_{1,n}+S_{2,n}.        \label{prf-consis-mo-1}
\end{align}
\subsection{Proof of Theorem  \ref{consis-CLT-mo} }

\begin{proof}
To prove the Theorem, first, we will show that $S_{1,n}$ and $S_{2,n}$ converge to zero in probability under our assumptions. To this end,
by Chebyshev's inequality

\begin{align}
      \bm{P} ( |S_{1,n} | > \epsilon ) \le  \frac{\bm{Var}  [  S_{1,n} ] }{\epsilon^2} =  \frac{1}{J_n(n-\tau_n)^2\epsilon^2}    \bm{Var} \left(  \sum_{l=\tau_n}^{n-1}   Y_{l,j} \right)     \to 0 , \label{consis-mo-p1}
\end{align}
where the convergence follows from $J_n(n-\tau_n)^2 \to \infty $ and Lemma \ref{varsumYl-finite}, yielding $S_{1,n} \xrightarrow[]{p} 0$.
Next, since $\{ m_{l,j}; l=0,1,2,\cdots,n; j=1,2,\cdots J_n\}$ is $(n+1)J_n$'s are i.i.d random variable, for any $\epsilon>0$,
\begin{align}
     \bm{P} ( |S_{2,n} | > \epsilon ) \le \frac{\bm{Var}(S_{2,n})}{\epsilon^2} = \frac{\bm{Var}[m_{0,1}]}{J_n(n-\tau_n)\epsilon^2}=\frac{\sigma^{2*}}{J_n(n-\tau_n)\epsilon^2} \to 0 , \label{consis-mo-p2}
\end{align}
yielding $S_{2,n} \xrightarrow[]{p} 0 $. This completes the proof of weak consistency.

\noindent
Now, we prove strong consistency under the condition $J_n(n-\tau_n)=O(n^{\alpha}) $ for $\alpha>1$. As $N \to \infty$, by using the inequality in (\ref{consis-mo-p1}) and Lemma \ref{varsumYl-finite}, we have
\begin{align*}
    \sum_{n=1}^{N} \bm{P} ( |S_{1,n} | > \epsilon ) &\le  \sum_{n=1}^{N}  \frac{\gamma^{*2} c_H }{J_n(n-\tau_n)^2(1-a)\epsilon^2}
    \le \frac{  \gamma^{*2} c_H c_\alpha }{ \epsilon^2 (1-a)}  \sum_{n=1}^N n^{-\alpha} < \infty
\end{align*}
for some constant $c_\alpha $ since $J_n(n-\tau_n)=O(n^{\alpha}) $ for $\alpha>1$. By Borel-Cantelli lemma, this yields $S_{1,n} \xrightarrow[]{a.s.} 0$. 
Also, by using the inequality in (\ref{consis-mo-p2}) and $J_n(n-\tau_n)=O(n^{\alpha}) $ for $\alpha>1$,
\begin{align*}
     \sum_{n=1}^{N} \bm{P} ( |S_{2,n} | > \epsilon ) &\le \sum_{n=1}^{N} \frac{\sigma^{2*}}{J_n(n-\tau_n)\epsilon^2} <  \frac{\sigma^{2*} c_\alpha}{\epsilon^2} \sum_{n=1}^{N} n^{-\alpha} <  \infty ,
\end{align*}
yielding $S_{2,n} \xrightarrow[]{a.s.} 0  $ by yet another application of Borel-Cantelli lemma.
Hence, strong consistency follows.

\noindent
Next, turning to the proof of the limit distribution,  using (\ref{prf-consis-mo-1}), we obtain
\begin{align}
    \sqrt{J_n(n-\tau_n)} ( \hat{m}_{n} - m^*)  
  = \sqrt{J_n(n-\tau_n)} S_{1,n} +  \sqrt{J_n(n-\tau_n)} S_{2,n}. \label{prf-consis-mo-2}
\end{align}
First, we consider the first term on the RHS of (\ref{prf-consis-mo-2}) under different cases of $\tau_n $ in sampling scheme \ref{assum-tau_n}. If $n-\tau_n \to \infty $, which is true for case (i) and case (iii), using the same argument as (\ref{consis-mo-p1}), we get
\begin{align}
   \bm{P} ( |\sqrt{J_n(n-\tau_n)} S_{1,n} | > \epsilon )   \le  \frac{1}{n-\tau_n} \bm{Var} \left(  \sum_{l=\tau_n}^{n-1}   Y_{l,j} \right) \to 0 .
\end{align} 
On the other hand, if $n-\tau_n $ is fixed as in case (ii), the inequality above also follows because $\bm{Var}\left(  \sum_{l=\tau_n}^{n-1}   Y_{l,j} \right) \to 0 $ by Lemma \ref{varsumYl-finite}. Now we turn to the second term on the RHS of (\ref{prf-consis-mo-2}).
Setting $\sqrt{J_n(n-\tau_n)} S_{2,n} \coloneqq  G_n $ and
$N_n = J_n(n-\tau_n)$,  the characteristic function of $G_n$, using the  independence over $j$ and $l$, is given by:
\begin{align*}
\phi_{G_n}(t) = E\left[ e^{ i t G_n } \right]
 = \left( \phi_X\left( \frac{ t }{ \sqrt{ J_n (n - \tau_n) } } \right) \right)^{ J_n ( n - \tau_n ) } = \left( \phi_X\left( \frac{ t }{ \sqrt{ N_n } } \right) \right)^{ N_n },
\end{align*}
where $X=m_{1,1} - m^* $. By Theorem 8.3.3 in \cite{Chow-Ti1978} it follows that
\begin{align*}
\phi_{G_n}(t) &= \left( 1 - \frac{ \sigma^{2*} t^2 }{ 2 N_n } + o\left( \frac{ t^2 }{ N_n } \right) \right)^{ N_n }\to e^{ - \frac{ \sigma^{2*} t^2 }{ 2 } } ~~ \text{as}~~ n \to \infty.
\end{align*}
Hence, it follows that $ G_n \xrightarrow{d} N(0, \sigma^{2*})$; Finally, by  (\ref{prf-consis-mo-2}) and  Slutsky's Theorem, the Theorem follows.

\end{proof}

\subsection{Proof of Theorem \ref{consis-CLT4mataunJ} }
In the rest of this section, we assume that \ref{AssumA1}-\ref{Assum-mo4+2delta} holds. As discussed above, the asymptotic property of $\hat{m}_{A,n}$ depends on the behavior of the CV, namely $r^{2*}=m_2^* /m^{*2}$.
Hence, to prove Theorem \ref{consis-CLT4mataunJ}, we need several results concerning the asymptotic behavior of the estimator of $m_{2}^* $, namely, 
\begin{align*}
        \hat{m}_{2,n}&=\frac{1}{J_n} \sum_{j=1}^{J_n}  \frac{1}{n-\tau_n} \sum_{l=\tau_n+1}^n \left(\frac{Z_{l,j}}{Z_{l-1,j}} \right)^2,
\end{align*}
Set $Y_{2,l,j}= \left(  \frac{ Z_{l+1,j}^2  } {Z_{l,j}^2}     -m_{l,j}^2   \right)  $ and denote the generic element of $Y_{2,l,j} $ by $Y_{2,l} $. Our first Lemma is similar in spirit to Lemma \ref{varsumYl-finite} concerning the variance of the sum of $Y_{2,l}$.

\begin{lemma}
\label{EY2l^2-finite}
 For some $0<a<1 $ and some constant $  c_\delta^{**}$,
    \begin{align*}
     \limsup_{n \to \infty} \bm{Var}\left[  \sum_{l=\tau_n+1}^n     Y_{2,l}  \right]  \le  \lim_{n \to \infty}
  c_\delta^{**} \sum_{l=\tau_n+1}^\infty  l (a^{\frac{1}{2}})^l .
\end{align*}
Moreover, there exists constant $ c_\tau^{\prime} $ such that
\begin{align*}
  &\limsup_{n \to \infty}  \bm{E} \left[   \sum_{l=\tau_n+1}^{n}   Y_{2,l,j} \right]^2 \le  c_\tau^{\prime}, \text{ if $\tau_n \equiv \tau$ is fixed (case (i)), and}\\
   &\lim_{n \to \infty} \bm{E} \left[   \sum_{l=\tau_n+1}^{n}   Y_{2,l,j} \right]^2 \to  0,   \text{ if $\tau_n \to \infty$  (case (ii) and (iii)).}
\end{align*}
\end{lemma}

\begin{proof}
Notice that 
\begin{align}
 \bm{E} \left[  \sum_{l=\tau_n+1}^n     Y_{2,l}  \right]^2  &=  \bm{Var} \left[  \sum_{l=\tau_n+1}^n     Y_{2,l}  \right] +
     \left[\bm{E} \left(  \sum_{l=\tau_n+1}^n     Y_{2,l}  \right)\right]^2 \label{EY2l^2-finite-0}.
\end{align}
We will prove the theorem by showing the limits of the two terms on the RHS of the above equation exist and are finite. Recall that $\mathcal{F}_{n} $ denotes the  $\sigma$-field generated by $(\zeta_{0},\zeta_{1},\cdots,\zeta_{n-1}, Z_{0},\cdots,Z_{n}) $.  For the first term in (\ref{EY2l^2-finite-0}), using that $\bm{E}   (  Y_{2,n}  | \mathcal{F}_n  ) = \frac{\sigma_{n}^2}{Z_n}  $, we have
\begin{align}
      \bm{Var} \left[  \sum_{l=\tau_n+1}^n     Y_{2,l}  \right] 
      &=  \bm{Var}\left[   \bm{E}   \left[ \sum_{l=\tau_n+1}^n   Y_{2,l}  \middle| \mathcal{F}_n  \right] \right] + \bm{E}\left[   \bm{Var}   \left[ \sum_{l=\tau_n+1}^n   Y_{2,l}  \middle| \mathcal{F}_n  \right] \right]  \nonumber \\
   &=   \bm{Var}\left[ \sum_{l=\tau_n+1}^{n-1}   Y_{2,l}+    \frac{\sigma_{n}^2}{Z_n} \right] + \bm{E}[   \bm{Var}   (    Y_{2,n}  | \mathcal{F}_n  ) ] \nonumber \\ 
   &= \bm{Var}\left[  \sum_{l=\tau_n+1}^{n-1}     Y_{2,l}  \right] +2 \bm{Cov} \left[ \sum_{l=\tau_n+1}^{n-1} Y_{2,l},\frac{\sigma_{n}^2}{Z_n} \right] + \bm{Var} \left[\frac{\sigma_{n}^2}{Z_n} \right] + \bm{E}[   \bm{Var}   (    Y_{2,n}  | \mathcal{F}_n  ) ] \nonumber \\
   &\coloneqq  \bm{Var}\left[  \sum_{l=\tau_n+1}^{n-1}     Y_{2,l}  \right] + L_{n,\tau_n} . \label{VARY2l^2-iter}
\end{align}
Now, by iterating (\ref{VARY2l^2-iter}) , it follows that
\begin{align}
      \bm{Var} \left[  \sum_{l=\tau_n+1}^n     Y_{2,l}  \right]  =   \sum_{l=\tau_n+1}^n L_{n,\tau_n} \le  
  c_\delta^{**} \sum_{l=\tau_n+1}^\infty  l (a^{\frac{1}{2}})^l \label{VARY2l^2-finite}
\end{align}
for some $c_\delta^{**}>0 $, where the inequality follows by using Lemma \ref{L_ntau_n} in the Appendix \ref{apd-c-EY2l^2-finite}. On the other hand, for the second term in (\ref{EY2l^2-finite-0}), because 
\begin{align}
    \bm{E} \left[  \left( \frac{ Z_{l+1}  } {Z_l}\right)^2 \middle|\mathcal{F}_l \right]  &=  \bm{Var} \left[  \left( \frac{ Z_{l+1}  } {Z_l}\right) \middle|\mathcal{F}_l \right] +\left[ \bm{E}   \left( \frac{ Z_{l+1}  } {Z_l} \middle|\mathcal{F}_l\right) \right]^2 \\
    &= \frac{1}{Z_l} \bm{Var} [  \xi_{l,1}|\mathcal{F}_l] + \left[  \bm{E} [  \xi_{l,1}|\mathcal{F}_l] \right]^2 =   \frac{\sigma_{l}^2}{Z_l} + m_l^2,
\end{align}
 we have $ \bm{E} [   Y_{2,l}|\mathcal{F}_l ]     =  \frac{\sigma_{l}^2}{Z_l} $.
Next, notice that 
\begin{align*}
  \bm{E} \left[  \sum_{l=\tau_n+1}^n     Y_{2,l}  \right] &= \bm{E} \left[ \bm{E} \left[  \sum_{l=\tau_n+1}^{n-1}     Y_{2,l} + Y_{2,n} \middle| \mathcal{F}_n  \right]\right]  
  = \bm{E} \left[ \sum_{l=\tau_n+1}^{n-1} Y_{2,l} \right] +  \bm{E} \left[ \frac{\sigma_{n}^2}{Z_n}  \right],
\end{align*}
by iterating above and using Lemma \ref{sum1/zn}, we obtain
\begin{align}
  \bm{E} \left[  \sum_{l=\tau_n+1}^n     Y_{2,l}  \right]   =\bm{E}\left[ \sum_{l=\tau_n+1}^n \frac{\sigma_{l}^2}{Z_l}\right] = \gamma^{2*} \sum_{l=\tau_n+1}^n c_H a^l \label{EY2l^2-finite-EE}.
\end{align}
Now, using (\ref{EY2l^2-finite-0}), (\ref{VARY2l^2-finite}), and (\ref{EY2l^2-finite-EE}), it follows that
\begin{align}
   \limsup_{n\to \infty} \bm{E} \left[  \sum_{l=\tau_n+1}^n     Y_{2,l}  \right]^2  
&\le  \lim_{n\to \infty}  c_\delta^{**} \sum_{l=1}^n l (a^{\frac{1}{2}})^l + \lim_{n\to \infty} \left(  \gamma^{2*} \sum_{l=\tau_n+1}^n c_H a^l \right)^2     < \infty \label{EY2l^2-finite-1},
\end{align}
and we conclude that for some $ 0 < c_\tau^{\prime} <\infty$:

\begin{align*}
  &\limsup_{n \to \infty}  \bm{E} \left[   \sum_{l=\tau_n+1}^{n}   Y_{2,l,j} \right]^2 \le  c_\tau^{\prime}, \text{ if $\tau_n \equiv \tau$ is fixed (case (i)), and}\\
   &\lim_{n \to \infty} \bm{E} \left[   \sum_{l=\tau_n+1}^{n}   Y_{2,l,j} \right]^2 \to  0,   \text{ if $\tau_n \to \infty$  (case (ii) and (iii)).}
\end{align*}
    
\end{proof}

Next, we use the above lemma to establish the consistency of $\hat{m}_{2,n} $. We begin by adding and subtracting $m_{l,j}^2 $ to $\hat{m}_{2,n} $, to obtain
\begin{align}
        \hat{m}_{2,n}- m_{2}^* &=\frac{1}{J_n} \sum_{j=1}^{J_n}  \frac{1}{n-\tau_n} \sum_{l=\tau_n+1}^n \left(\frac{Z_{l,j}}{Z_{l-1,j}} \right)^2- m_{2}^* \nonumber  \\
        &=\frac{1}{J_n} \sum_{ j=1 }^{J}  \left[   \frac{1}{n-\tau_n}  \sum_{l=\tau_n+1}^n     \left(  \frac{ Z_{l+1,j}^2  } {Z_{l,j}^2}     -m_{l,j}^2   \right) \right] + \frac{1}{J_n} \sum_{ j=1 }^{J}  \left[   \frac{1}{n-\tau_n}  \sum_{l=\tau_n+1}^n   ( m_{l,j}^2 - m_{2}^* ) \right] \nonumber \\
        &\coloneqq S_{1,n}^{\prime}+S_{2,n}^{\prime}. \label{consis-mo2-p0}
\end{align}

\begin{theorem}
\label{consis-mo2}
 Under \ref{AssumA1}-\ref{Assum-mo4+2delta}, as $ n \to \infty$, $\hat{m}_{2,n}  \xrightarrow[]{p} m_{2}^* $.
\end{theorem}

\begin{proof}
We prove the theorem by showing that the RHS of (\ref{consis-mo2-p0}) converges to 0 in probability.
For $S_{1,n} $, for any $\epsilon>0$, by using Chebyshev's inequality and Lemma \ref{EY2l^2-finite},
\begin{align}
  \bm{P}(|S_{1,n}^{\prime}|>\epsilon    )   &\le
 \frac{ \bm{E} [S_{1,n}^{\prime}]^2 }{\epsilon^2} \nonumber \\ 
 & =\frac{1}{\epsilon^2} [ \bm{Var} [S_{1,n}^{\prime}]^2   +  (\bm{E} [ S_{1,n}^{\prime}] )^2 ] \nonumber \\
  &= \frac{1}{J_n(n-\tau_n)^2\epsilon^2} \bm{Var} \left[  \sum_{l=\tau_n+1}^n     Y_{2,l}  \right] +  \left[  \frac{1}{(n-\tau_n)\epsilon^2}\bm{E} \left( \sum_{l=\tau_n+1}^n     Y_{2,l}  \right) \right]^2.\label{consis-mo2-p1}
\end{align}
Now for the first term in (\ref{consis-mo2-p1}), using (\ref{VARY2l^2-finite}), Lemma \ref{sum1/zn}, and $J_n(n-\tau_n)^2 \to \infty $, as $n \to \infty$, it follows that
\begin{align}
    \frac{1}{J_n(n-\tau_n)^2\epsilon^2} \bm{Var} \left[  \sum_{l=\tau_n+1}^n     Y_{2,l}  \right]   \le  \frac{c_\delta^{**}}{J_n(n-\tau_n)^2\epsilon^2} \sum_{l=\tau_n+1}^\infty  l (a^{\frac{1}{2}})^l \to 0.
\end{align}
For the second term in (\ref{consis-mo2-p1}), by (\ref{EY2l^2-finite-EE}) and Lemma \ref{sum1/zn}, if $ n- \tau_n \to \infty$ (as in case (i) and (iii)),  as $n \to \infty$
\begin{align*}
     \frac{1}{(n-\tau_n)\epsilon^2}\bm{E} \left( \sum_{l=\tau_n+1}^n     Y_{2,l} \right) \le  \frac{\gamma^{2*}}{(n-\tau_n)\epsilon^2}  \sum_{l=\tau_n+1}^n c_H a^l \to 0.
\end{align*}
Otherwise if $ n- \tau_n $ is fixed as in case (ii), the convergence above still holds since $ \sum_{l=\tau_n+1}^n c_H a^l \to 0$.
Next, for $S_{2,n} $, using \ref{Assum-mo4+2delta}, since  $\bm{E}  [m_{0,1}^2 -m_{2}^*]^2 \coloneqq \rho_4^{*2} < \infty $, we have
\begin{align*}
     \bm{P} ( |S_{2,n}^{\prime} | > \epsilon ) \le \frac{\bm{Var}(S_{2,n}^{\prime})}{\epsilon^2} = \frac{\rho_4^{*2}}{J_n(n-\tau_n)\epsilon^2} \to 0 , ~~\text{as}~~ n \to \infty.
\end{align*}
This concludes the proof of the Theorem.
\end{proof}

\begin{remark}[Estimation of $\sigma^{2*}$]
    Since $\sigma^{2*} = m_{2}^*-m^{*2} $, by Theorem \ref{consis-CLT-mo} and Theorem \ref{consis-mo2}, deriving a consistent estimator for $\sigma^{2*} $ is straightforward, given by $\hat{m}_{2,n}-\hat{m}_n^2 $. We will use this estimator to estimate the variance of $ \hat{m}_n$ in the data analysis described in Section  \ref{sec-da}.
\end{remark}

\noindent
We focus back on the asymptotic limit distribution of $\hat{m}_{A,n}$. First, recall from (\ref{def-r_taunJ}) that
\begin{align*}
     \hat{r}_{n} = \frac{\sqrt{ \hat{m}_{2,n}}}{\hat{m}_{n}}  , \quad 
      \hat{m}_{2,n}&=\frac{1}{J_n} \sum_{j=1}^{J_n} \frac{1}{n-\tau_n} \sum_{l=\tau_n+1}^n \left(\frac{Z_{l,j}}{Z_{l-1,j}} \right)^2.
\end{align*}
In the proof of the limit distribution, we will encounter the limit behavior of $(r^{*(-1)}  \hat{r}_{n})^n$. Our next Lemma establishes that the ratio $(r^{*(-1)}  \hat{r}_{n})^n$ converges to different values depending on the sampling scheme \ref{assum-tau_n}.
\begin{theorem}
\label{ro^n to 1}
Under \ref{AssumA1}-\ref{Assum-mo4+2delta}, if $\frac{n^2}{J_n}\to 0$, then as $n \to \infty$,
\begin{align*}
    \frac{ \hat{r}_{n}^n}{r^{*n}}  \xrightarrow[]{p}   \begin{cases}
              \exp \left(\frac{\lambda_{\tau}\gamma^{2*}}{2m_{2}^*}   \right) & \text{if } \tau_n  \equiv \tau\\
              1 \quad &\text{if } \tau_n \to \infty
          \end{cases}
\end{align*}
\end{theorem}

\begin{proof}
Notice that
\begin{align*}
     \frac{ \hat{r}_{n}^n}{r^{*n}} &=   \left( \frac{ \hat{m}_{2,n}}{m_{2}^*} \right) ^\frac{n}{2} \cdot \left(\frac{m^* }{ \hat{m}_{n}} \right)^n .
\end{align*}
Using Theorem \ref{mo^n to 1}, the second term in the above converges to one in probability. Hence, it is sufficient to show that, as $n \to \infty$,
\begin{align*}
     \left( \frac{ \hat{m}_{2,n}}{m_{2}^*} \right) ^n\xrightarrow[]{p}   \begin{cases}
              \exp \left(\frac{\lambda_{\tau}\gamma^{2*}}{m_{2}^*}   \right) & \text{if } \tau_n  \equiv \tau\\
              1 \quad &\text{if } \tau_n \to \infty.
          \end{cases}
\end{align*}
 Now, since
\begin{eqnarray*}
\bm{E}[ \hat{m}_{2,n} ] \coloneqq   m_{2,n}^*=   \bm{E} \left(  \frac{1}{n-\tau_n}  \sum_{l=\tau_n}^{n-1}     \frac{ Z_{l+1,j}^2  } {Z_{l,j}^2}   \right)= m_{2}^* +  \gamma^{2*} \frac{1}{n-\tau_n} \sum_{l=\tau_n}^{n-1} \bm{E}\left[\frac{1}{Z_l} \right] \coloneqq m_{2}^* + b_{2,n}^*,
\end{eqnarray*}
we can rewrite $    ( \frac{ \hat{m}_{2,n}}{m_{2}^*} ) ^n$ as
\begin{align*}
    \left( \frac{ \hat{m}_{2,n}}{m_{2}^*} \right) ^n &=  \left( \frac{ \hat{m}_{2,n}}{ m_{2,n}^*} \right) ^n \cdot \left( \frac{ m_{2,n}^*}{m_{2}^*} \right) ^n=  \left( \frac{ \hat{m}_{2,n}}{ m_{2,n}^*} \right) ^n \left( \frac{m_{2}^*+b_{2,n}^*}{m_{2}^*}\right)^n.
\end{align*}
It is sufficient to show that the first term on the RHS of the above equation converges to 1 in probability and that the second term converges to the limit specified in the statement of the Theorem. To this end, 
\begin{align}
\bm{P} \left[ \left|  n\log \frac{ \hat{m}_{2,n} } {m_{2,n}^*}  \right|> \epsilon  \right] 
     & = \bm{P} \left[   \hat{m}_{2,n}   > m_{2,n}^* \exp \left( \frac{  \epsilon}{n}\right)  \right]  + \bm{P} \left[   \hat{m}_{2,n}   < m_{2,n}^* \exp \left(- \frac{  \epsilon}{n}\right)  \right] \label{mo2^n2parts} .
\end{align}
Now, since $m_{2,n}^*=m_{2}^*+b_{2,n}^* $, we have $ \hat{m}_{2,n} -m_{2,n}^*=  ( \hat{m}_{2,n} - b_{2,n}^*) - m_{2}^* $. Then by adding and subtracting $  m_{l,j}^2 $, we have $  \hat{m}_{2,n} -m_{2,n}^* =  S_{1,n}^{\prime\prime}+S_{2,n}^{\prime}$, where
\begin{align*}
    S_{1,n}^{\prime\prime}=  \frac{1}{J_n}\sum_{ j=1 }^{J_n} \left[ \frac{1}{n-\tau_n}  \sum_{l=\tau_n}^{n-1}     \left( \frac{ Z_{l+1,j}^2  } {Z_{l,j}^2}  - m_{l,j}^2 \right) - b_{2,n}^* \right] ~~ \text{and}~~
    S_{2,n}^{\prime}=   \frac{1}{J_n} \sum_{ j=1 }^{J_n}       \frac{1}{n-\tau_n}  \sum_{l=\tau_n}^{n-1}  \left(   m_{l,j}^2 -m_{2}^*  \right) .
\end{align*}
Then for the first term on the RHS of (\ref{mo2^n2parts}), we have
\begin{align*}
    S_{1,n}^{\prime\prime}+ S_{2,n}^{\prime}+m_{2,n}^* > m_{2,n}^* \exp \left( \frac{  \epsilon}{n}\right) 
    \iff & S_{1,n}^{\prime\prime}+S_{2,n}^{\prime} > m_{2,n}^*   \left[ \exp \left( \frac{  \epsilon}{n}\right) -1 \right]
\end{align*}
Hence, since $ m_{2,n}^*> m_{2}^* $ for all $n$, using $ e^x-1>x$ for $x>0$,
\begin{align}
\bm{P} \left[   \hat{m}_{2,n}   > m_{2,n}^* \exp \left( \frac{  \epsilon}{n}\right)  \right] &= \bm{P} \left\{ S_{1,n}^{\prime\prime}+ S_{2,n}^{\prime} > m_{2,n}^*   \left[ \exp \left( \frac{  \epsilon}{n}\right) -1 \right] \right\} \nonumber \\
    & \le       \bm{P} \left[ S_{1,n}^{\prime\prime} \ge \frac{\epsilon m_{2,n}^*   }{2n} \right] + \bm{P} \left[  S_{2,n}^{\prime} >   \frac{\epsilon m_{2,n}^*   }{2n}  \right]  \nonumber \\
     & \le       \bm{P} \left[ S_{1,n}^{\prime\prime} \ge \frac{\epsilon m_{2}^*   }{2n} \right] + \bm{P} \left[  S_{2,n}^{\prime} >   \frac{\epsilon m_{2}^*   }{2n}  \right]  \label{S1pS2p}. 
\end{align}
For the first term in (\ref{S1pS2p}),  using Chebyshev's inequality, and definition of $Y_{2,l,j}  $,
\begin{align}
      \bm{P} \left[ S_{1,n}^{\prime\prime} \ge \frac{\epsilon m_{2}^*   }{2n}  \right] &=   
      \bm{P} \left[ \frac{1}{J_n}\sum_{ j=1 }^{J_n} \left[ \frac{1}{n-\tau_n}  \sum_{l=\tau_n}^{n-1}     \left( \frac{ Z_{l+1,j}^2  } {Z_{l,j}^2}  - m_{l,j}^2 \right)  - b_{2,n}^* \right]\ge \frac{\epsilon m_{2}^*   }{2n}  \right] \nonumber \\
      &=   \bm{P} \left[ \frac{1}{J_n(n-\tau_n)}\sum_{ j=1 }^{J_n}   \sum_{l=\tau_n}^{n-1}     \left( Y_{2,l,j}  - b_{2,n}^* \right) \ge \frac{\epsilon m_{2}^*   }{2n}  \right] \nonumber \\ 
      & \le \bm{Var} \left[ \frac{n}{J_n(n-\tau_n)}  \sum_{ j=1 }^{J_n}    \sum_{l=\tau_n}^{n-1}     Y_{2,l,j}      \right] (\frac{\epsilon m_{2}^*   }{2})^{-2}  \nonumber \\
    & =\frac{n^2}{J_n (n-\tau_n)^2 (\epsilon m_{2}^*   )^2 }  \bm{Var} \left[  \sum_{l=\tau_n}^{n-1}     Y_{2,l,j}\right]   
     \to 0 \label{S1nJp-conv},
\end{align}
where the last convergence follows, using Lemma \ref{EY2l^2-finite} and $\frac{n^2}{J_n(n-\tau_n)} \to 0 $. For the second term in (\ref{S1pS2p}), because $\bm{E}[m_{0,1} ]^4 < \infty $, setting $\bm{E}  [m_{0,1}^2 -m_{2}^*]^2 = \rho_4^{*2} $ we have
\begin{align}
   \bm{P} \left[ S_{2,n}^{\prime} >   \frac{\epsilon m_{2}^*   }{2n} \right]   
  &=  \bm{P} \left[  \frac{1}{J_n} \sum_{ j=1 }^{J_n}       \frac{1}{n-\tau_n}  \sum_{l=\tau_n}^{n-1}  \left(   m_{l,j}^2 -m_{2}^*  \right) >   \frac{\epsilon m_{2}^*   }{2n} \right]  \nonumber \\ 
  &\le \bm{Var} \left[     \frac{1}{J_n}  \sum_{ j=1 }^{J_n}    \frac{1}{n-\tau_n}  \sum_{l=\tau_n}^{n-1} Y_{2,l,j} \right]  \left( \frac{\epsilon m_{2}^*   }{2n} \right)^{-2} \nonumber  \\
  &= \frac{n^2 }{J_n(n-\tau_n) } \frac{ \rho_4^{*2}}{\epsilon m_{2}^*   } \to 0,  \quad (\text{since }\frac{n^2}{J_n(n-\tau_n)} \to 0 )\label{S2nJp-conv}.
\end{align}
Now by combining (\ref{S1pS2p}), (\ref{S1nJp-conv}), and (\ref{S2nJp-conv}), we have 
\begin{align}
     \bm{P} \left[   \hat{m}_{2,n}   > m_{2,n}^* \exp \left( \frac{  \epsilon}{n}\right)  \right] \to 0 . \label{S1S2p-cov1}
\end{align}
A similar calculation can show that the second term in (\ref{mo2^n2parts}), namely
\begin{align}
     \bm{P} \left[   \hat{m}_{2,n}   < m_{2,n}^* \exp \left( - \frac{  \epsilon}{n}\right)  \right] \to 0 . \label{S1S2p-cov2}
\end{align}
Now combining (\ref{S1S2p-cov1}) and (\ref{S1S2p-cov2}), it follows that (\ref{mo2^n2parts}) converges to 0. 
Finally, using Lemma \ref{sum1/zn}, we obtain
\begin{align*}
\left( \frac{ m_{2,n}^*}{m_{2}^*} \right)^n &=  \left( \frac{m_{2}^*+b_{2,n}^*}{m_{2}^*}\right)^n  = \left[ 1+  \frac{1}{n} \left(\frac{\gamma^{2*}}{m_{2}^*}  \sum_{l=\tau_n}^{n-1} \bm{E}\left[\frac{1}{Z_l}\right] \right)\right]^n 
           \to 
          \begin{cases}
              \exp \left(\frac{\lambda_{\tau}\gamma^{2*}}{m_{2}^*}   \right) & \text{if } \tau_n  \equiv \tau\\
              1 \quad &\text{if } \tau_n \to \infty
          \end{cases},
    \end{align*}
thus completing the proof of the Lemma.
\end{proof}

We now turn to the proof of Theorem \ref{consis-CLT4mataunJ}. To prove the consistency, recall that 
\begin{align*}
    \mathcal{N}_{n}^*=   m^{*\tau_n}(m^{*(n - \tau_n +1} -1))      (m^*-1)^{-1} ~~\text{and}~~   \hat{\mathcal{N}}_{n} = \hat{m}_n^{\tau_n}(\hat{m}_n^{ n - \tau_n +1} -1)      (\hat{m}_n-1)^{-1} .
\end{align*}
The estimator can be expressed as
\begin{align*}
\hat{m}_{A,n}=  \frac{\mathcal{N}_{n}^*}{\hat{\mathcal{N}}_{n}} \cdot \left( \frac{1}{J_n} \sum_{j=1}^{J_n} \frac{1}{\mathcal{N}_{n}^* } \sum_{l=\tau_n}^n Z_{l,j} \right).    
\end{align*}
We will show that
\begin{align*}
\frac{\mathcal{N}_{n}^*} {\hat{\mathcal{N}}_{n}} \xrightarrow[]{p} 1  \quad \text{and} \quad 
\left( \frac{1}{J_n} \sum_{j=1}^{J_n} \frac{1}{\mathcal{N}_{n}^* } \sum_{l=\tau_n}^n Z_{l,j} \right)     \xrightarrow[]{p} m_A.
\end{align*}
The behavior of $\frac{\mathcal{N}_n^*}{\hat{\mathcal{N}}_n}$ is required in several results, and we separate it as a Lemma below.

\begin{lemma}
\label{mathcal_N-p1}
Under the condition of Theorem \ref{mo^n to 1}, $     \frac{\mathcal{N}_{n}^*} {\hat{\mathcal{N}}_{n}} \xrightarrow[]{p} 1    $ as $n \to \infty $.
\end{lemma}
\begin{proof}
First observe that, by Theorem \ref{consis-CLT-mo}, 
$\frac{\hat{m}_{n}-1}{m^*-1 } \xrightarrow[]{p} 1$.
Also, by Theorem \ref{mo^n to 1} and the fact that $\tau_n < n $,  we have
$ \frac{m^{*\tau_n}   }{\hat{m}_{n}^{\tau_n}   } \xrightarrow[]{p} 1 $. It follows that $\frac{(m^{*(n - \tau_n +1)} -1)   }{(\hat{m}_{n}^{n - \tau_n +1} -1)        } \xrightarrow[]{p} 1$. Combining the above statements, the conclusion follows.
\end{proof}

\noindent
 Next Lemma, is required for the proof of the Theorem \ref{consis-CLT4mataunJ}.

\begin{lemma}
\label{CLT4maC1}      
Under the condition of \ref{consis-CLT4mataunJ}, as $ n \to \infty$,
\begin{align*}
  \frac{1}{r^{*n}}  \sqrt{J_n}\left( \frac{m^{*n}  } { \hat{m}_{n}^{n}} -1 \right)  \xrightarrow[]{p} 0 .
\end{align*}
Additionally 
    \begin{align*}
        \frac{1}{r^{*n}} \sqrt{J_n}\left( \frac{\mathcal{N}_{n}^*}{\hat{\mathcal{N}}_{n}} -1 \right)  \xrightarrow[]{p} 0.
    \end{align*}
\end{lemma}

\begin{proof}
Using algebra,
\begin{align*}  
\frac{1}{r^{*n}}\sqrt{J_n}\left( \frac{ \hat{m}_{n}^{n}}{m^{*n}  }  -1 \right) &= \left[ \frac{1}{n} \sum_{k=0}^{n-1} \left(\frac{\hat{m}_{n}}{ m^*} \right)^k \right] \left[  \sqrt{J_n (n-\tau_n)} (\hat{m}_{n}-m^*)\right] \frac{n}{r^{*n} m^* (n-\tau_n)^{\frac{1}{2}}   } \\
& \coloneqq R_n{(1)} \cdot R_n{(2)} \cdot R_n{(3)}
\end{align*}
By Theorem \ref{consis-CLT-mo}, $ R_n{(2)} $ converges in distribution to $ N(0, \sigma^{2*} )$.  Turning to $R_n{(1)} $, using $|a^k-1| \le |a^n -1|$, for any $a>0$ and $ 1\le k \le n$, it follows that using Theorem  \ref{mo^n to 1}
the first term converges to 1 in probability. Now using $r^*>1 $, $R_n{(3)} $ converges to 0, the proof of the Lemma follows. Turning to the behavior of $\frac{\mathcal{N}_{n}^*}{\hat{\mathcal{N}}_{n}} $ since $m^* >1$, we have for $ 1 \le  k \le n $, $|\hat{m}_{n}^k - m^{*k} |=  m^{*k}|\frac{\hat{m}_{n}^k}{ m^{*k}}-1 | \le m^{*n}|\frac{\hat{m}_{n}^n}{m^{*n}}-1|$. Hence 
\begin{align*}
    \frac{1}{r^{*n}} \sqrt{J_n}\left( \frac{\mathcal{N}_{n}}{\mathcal{N}_{n}^*} -1 \right)  
   &=\frac{1}{r^{*n}} \frac{1}{\sum_{k=\tau_n}^{n} m^{*k}} \sqrt{J_n}\left( \sum_{k=\tau_n}^{n} \hat{m}_{n}^k - m^{*k} \right)  \\
   &\le  \frac{n}{r^{*n}} \frac{m^{*n}}{\sum_{k=\tau_n}^{n} m^{*k}} \sqrt{J_n}\left| \frac{ \hat{m}_{n}^{n}}{m^{*n}  }  -1  \right|  \\
   & =\frac{m^{*n}}{\sum_{k=\tau_n}^{n} m^{*k}} \cdot  R_n{(1)} \cdot R_n{(2)} \cdot nR_n{(3)} \xrightarrow[]{p}0.
\end{align*}
The convergence follows from that $\frac{m^{*n}}{\sum_{k=\tau_n}^{n} m^{*k}} \le 1 $ and $nR_n(3)  \xrightarrow[]{p} 0$. 
This concludes the proof of lemma.

\end{proof}

We now turn to the proof of Theorem \ref{consis-CLT4mataunJ}. The proof will be based on the central limit theorem for Martingale arrays as described in Corollary \ref{PHclt-da-C} in Appendix \ref{apd-c-tfo}. Specifically, we will verify conditions (\ref{apd-C}.\ref{PHCLTAc5}) and (\ref{apd-C}.\ref{PHCLTAc4b}).
\begin{proof}
We first prove the consistency. Recall that 
\begin{align*}
     \hat{m}_{A,n}=  \left(\frac{\mathcal{N}_{n}^*}{\hat{\mathcal{N}}_{n}}\right) \cdot \frac{1}{\mathcal{N}_{n}^*} \left( \frac{1}{J_n} \sum_{j=1}^{J_n}  \sum_{l=\tau}^n Z_{l,j} \right) .
\end{align*}
Since, by Lemma \ref{mathcal_N-p1}, the first term on the RHS of the above equation converges to one in probability, it is sufficient to show the second term converges to $m_A$ in probability. From Lemma \ref{var-sum-zn}, we note that 
\begin{align*}
     \bm{Var}  \left[  \frac{1}{r^{n*}} \frac{1}{\mathcal{N}_{n}^*} \sum_{l=\tau_n}^n Z_{l,j}   \right] < \infty.
\end{align*}  
Because by Chebyshev's inequality, for any $\epsilon>0$,
\begin{align*}
    \bm{P} \left[ \left|\frac{1}{\mathcal{N}_{n}^*} \left( \frac{1}{J_n} \sum_{j=1}^{J_n}  \sum_{l=\tau_n}^n Z_{l,j} \right)  -m_A\right| > \epsilon \right] & \le \frac{  \bm{Var}  \left[ \frac{1}{\mathcal{N}_{n}^*} \left( \frac{1}{J_n} \sum_{j=1}^{J_n}  \sum_{l=\tau_n}^n Z_{l,j} \right)   \right]}{\epsilon } \\
    &= \frac{r^{*2n}}{J_n} \frac{  \bm{Var}  \left[ \frac{1}{r^{n*}} \frac{1}{\mathcal{N}_{n}^*}    \sum_{l=\tau_n}^n Z_{l,1}    \right]}{\epsilon } \to 0,
\end{align*}
where the convergence follows using $\frac{n}{ \log J_n} \to 0 $ and $ \frac{r^{*2n}}{J_n}    \to 0 $. This completes the proof of consistency.

\noindent
For limiting distribution, notice that
\begin{align}
     \frac{1}{r^{*n}} \sqrt{J_n}  (   \hat{m}_{A,n} -m_A ) 
    =  \frac{\mathcal{N}_{n}^*}{\hat{\mathcal{N}}_{n}}  \frac{1}{r^{*n}}  \frac{1}{\sqrt{J_n}}  \sum_{j=1}^{J_n} \left(  \frac{1}{\mathcal{N}_{n}^*}\sum_{l=\tau_n}^n   Z_{l,j}-m_A \right) + \frac{1}{r^{*n}} \sqrt{J_n}\left( \frac{\mathcal{N}_{n}^*}{\hat{\mathcal{N}}_{n}} -1 \right) m_A \label{consis-CLT4mataunJ-1}.
\end{align}
For the first term of the RHS of (\ref{consis-CLT4mataunJ-1}), let
\begin{align*}
      X_{nj}= \frac{1}{r^{*n}} \frac{1}{\sqrt{J_n}}  \left(  \frac{1}{\mathcal{N}_{n}^*}  \sum_{l=\tau_n}^n   Z_{l,j}-m_A \right).
\end{align*}
Notice because $\{X_{nj}, n \ge 1, 1\le j \le  J_n  \} $ is i.i.d collection in $j$ for every $n$, by Lemma \ref{var-sum-zn},
\begin{align*}
    \sum_{j=1}^{J_n} \bm{E}[X_{nj}^2|\mathcal{F}_{n,j-1}] =  \sum_{j=1}^{J_n} \bm{E}[X_{nj}^2] \to   \begin{cases}
        \sigma_{I}^2. & \text{ if } n-\tau_n \to \infty\\
      \sigma_{F}^2(\delta_\tau)  & \text{ if } n-\tau_n \equiv \delta_\tau \\
    \end{cases} .
\end{align*}
This verifies the condition (\ref{apd-C}.\ref{PHCLTAc4b}) in Appendix \ref{apd-C}. Also, 
\begin{align}
     \frac{1}{r^{*n}}\frac{1}{\mathcal{N}_{n}^*}  \sum_{l=\tau_n}^n   Z_{l,j} &=  \frac{m^*-1}{m^*-m^{*\tau_n-n} } \sum_{l=\tau_n}^n \frac{Z_l}{m_{2}^{*\frac{n}{2}}} \le  \sum_{l=\tau_n}^n \frac{Z_l}{m_{2}^{*\frac{n}{2}}}, \label{consis-CLT4mataunJ-2}
\end{align}
 let $\epsilon>0$ be fixed, because $ \mathcal{N}_{n}^{*(-1)}  \sum_{l=\tau_n}^n   Z_{l,j} >0$, there exists a constant $ n_0(\epsilon) $, for any $ n\ge n_0(\epsilon) $ and $ J_n>\frac{m_A^2}{\epsilon^2}$, by independence in $j$, 
\begin{align*}
 \sum_{j=1}^{J_n} \bm{E} [ (X_{nj} )^2 I(|X_{nj}|>\epsilon) ]   & \le  \bm{E} \left[ \frac{1}{r^{*2n}}  \left(  \frac{1}{\mathcal{N}_{n}^*}  \sum_{l=\tau_n}^n   Z_{l,j} \right)^2 I \left(\frac{1}{r^{*2n}} \left(  \frac{1}{\mathcal{N}_{n}^*}  \sum_{l=\tau_n}^n   Z_{l,j} \right) ^2> J_n\epsilon^2 \right) \right] \\
    & \le  \bm{E} \left[  \frac{1}{m_{2}^{*n} }\left( \sum_{l=\tau_n}^n Z_l \right)^2 I\left( \frac{1}{m_{2}^{*n} }\left( \sum_{l=\tau_n}^n Z_l \right)^2 > J_n\epsilon^2  \right) \right], ~~\text{by (\ref{consis-CLT4mataunJ-2})} \\
    & \le  \sum_{k \ge J_n}^\infty  \bm{E} \left[ \frac{1}{m_{2}^{*n} }\left( \sum_{l=\tau_n}^n Z_l \right)^2: k\epsilon^2 \le \frac{1}{m_{2}^{*n} }\left( \sum_{l=\tau_n}^n Z_l \right)^2  \le (k+1)\epsilon^2 \right] \\
    & \le   \sum_{k \ge J_n}^\infty (k+1) \bm{P}\left[   \frac{1}{m_{2}^{*n} }\left( \sum_{l=\tau_n}^n Z_l \right)^2 \ge k\epsilon^2  \right] .
\end{align*}
Using the fact that $Z_n $ is non-decreasing, Markov inequality, and Theorem \ref{HL-zn/c-limit} in Appendix \ref{apd-C}, setting $\frac{m_{4+2\delta}^*}{m_{2}^{*2+\delta} } \coloneqq \Delta $:
\begin{align*}
   \sum_{k \ge J_n}^\infty (k+1) \bm{P}\left[   \frac{1}{m_{2}^{*n} }\left( \sum_{l=\tau_n}^n Z_l \right)^2 \ge k\epsilon^2  \right]
   &\le 
 \sum_{k \ge J_n}^\infty (k+1) \bm{P}\left(   \frac{n^2 Z_n^2}{m_{2}^{*n} } \ge k\epsilon^2  \right)\\
    & \le  n^{4+2\delta} \sum_{k \ge J_n}^\infty \frac{(k+1)}{k^{2+\delta}} \frac{\bm{E} [Z_{n,1}^{4+2\delta}]}{ m_{2}^{*(2+\delta)n}\epsilon^{2+2\delta}  } \\
     & \le c_4 n^{4+2\delta}  \left(\frac{m_{4+2\delta}^*}{m_{2}^{*2+\delta} } \right)^n\sum_{k \ge J_n}^\infty \frac{(k+1)}{k^{2+\delta}}     \\
     &\le n^{4+2\delta} c_{\epsilon}' \Delta^n \int_{J_n}^\infty \frac{1}{x^{1+\delta}} dx   =c_{\epsilon}''\frac{\Delta^n }{J_n^\delta} n^{4+2\delta} 
\end{align*}
for some constants $c_{\epsilon}' $ and $c_{\epsilon}''$. Next, since
\begin{align*}
    \log (c_{\epsilon}''\frac{\Delta^n }{J_n^\delta}n^{4+2\delta} ) 
    &= \log(J_n) \left( \frac{n}{\log J_n} \Delta  + \frac{c_{\epsilon}'' }{\log J_n}  +(4+2\delta)\frac{\log n}{\log J_n}-\delta \right) \to -\infty ,
\end{align*}
we have $c_{\epsilon}''\frac{\Delta^n }{J_n^\delta} n^{4+2\delta} \to 0 $, verifying condition (\ref{apd-C}.\ref{PHCLTAc5}) in Appendix \ref{apd-C}. Then by Corollary \ref{PHclt-da-C} and Lemma \ref{var-sum-zn}, it follows that
\begin{align}
\label{CLT4ma2tau-1}
 \frac{1}{r^{*n}} \frac{1}{\sqrt{J_n}}  \left(  \frac{1}{\mathcal{N}_{n}^*}  \sum_{l=\tau_n}^n   Z_{l,j}-m_A \right)\xrightarrow[]{d}    \begin{cases}
     N(0,  \sigma_{I}^2 )    & \text{ if } n-\tau_n \to \infty\\
      N(0,  \sigma_{F}^2(\delta_\tau) )    & \text{ if } n-\tau_n \equiv \delta_\tau 
    \end{cases} .
\end{align}
Now combining (\ref{consis-CLT4mataunJ-1}), Lemma \ref{mathcal_N-p1}, and Lemma \ref{CLT4maC1}, we conclude that
\begin{align*}
  \frac{1}{r^{*n}} \sqrt{J_n}  ( \hat{m}_{A,\tau_n,n,J} -m_A ) \xrightarrow[]{d}  \begin{cases}
     N(0,  \sigma_{I}^2 )    & \text{ if } n-\tau_n \to \infty\\
      N(0,  \sigma_{F}^2(\delta_\tau) )    & \text{ if } n-\tau_n \equiv \delta_\tau \\
    \end{cases} .
\end{align*}
Finally, by Theorem \ref{ro^n to 1} it follows that
\begin{align*}
      \frac{1}{\hat{r}_{n}^{n}} \sqrt{J_n} ( \hat{m}_{A,n} -m_A ) \xrightarrow[]{d} N(0,  \sigma^2_{\tau} ).
 \end{align*}

\end{proof}

\subsection{Proof of Theorem \ref{consis-CLT4ma} }
The proof of Theorem \ref{consis-CLT4ma} for other data structures stated in Section \ref{sec-mr} can be proved by a slight modification of the proof of Theorem \ref{consis-CLT4mataunJ} along with the next lemma.
\begin{lemma}   
\label{CLT4ma-1}
Under the condition of Theorem \ref{consis-CLT4ma}, as $n \to \infty$
,
\begin{align*}
     \frac{1}{r^{*n}} \sqrt{J_n} ( \tilde{m}_{A,n} -m_A ) \xrightarrow[]{d} N(0,  \frac{\mathfrak{D} }{m_{2}^*} ).
\end{align*}
\end{lemma}
\begin{proof}

Notice that
\begin{align}
     \frac{1}{r^{*n}} \sqrt{J_n} ( \tilde{m}_{A,n} -m_A ) &= 
       \frac{m^{*n}}{ \tilde{m}_{n}^{n}}  \frac{1}{r^{*n}} \frac{1}{\sqrt{J_n}} \sum_{j=1}^{J_n}\left( \frac{Z_{n,j}}{m^{*n}} -m_A\right)  + \frac{1}{r^{*n}}\sqrt{J_n}\left( \frac{m^{*n}  } { \tilde{m}_{n}^{n}} -1 \right) m_A  \label{ma_sep2}. 
\end{align}
The first term of the RHS of (\ref{ma_sep2}) converges in distribution to $ N(0,  \frac{\mathfrak{D} }{m_{2}^*} ) $ by setting $\tau_n \equiv n $ in (\ref{CLT4ma2tau-1}), while the second term converges to 0 in probability by Lemma \ref{CLT4maC1}.
    
\end{proof}

Now we show the proof of Theorem \ref{consis-CLT4ma}.
\begin{proof}
Recall by (\ref{varzn1}) and (\ref{varzn1-D}) that
\begin{align*}
     \bm{Var} (Z_{n,1})=m_{2}^{*(n-1)} \mathfrak{D}_n + m^{*2n} \sigma_A^2  \quad \text{and} \quad  \mathfrak{D}_n \nearrow \mathfrak{D}< \infty.
\end{align*}
By Chebyshev's inequality, for any $\epsilon>0$,
\begin{align}
    \bm{P} \left[ \left|\frac{1}{J_n} \sum_{j=1}^{J_n}   \frac{Z_{n,j} }{m^{*n}} -m_A\right| > \epsilon \right] &\le \frac{  \bm{Var}  \left( \frac{1}{J_n} \sum_{j=1}^{J_n}   \frac{Z_{n,j} }{m^{*n}}   \right)}{\epsilon^2 } 
    = \frac{   r^{*2n}}{J_n} \left[ \frac{\mathfrak{D}_n}{m_{2}^*}   +  \sigma_A^2 \right] \to 0, \label{consis-CLT4ma-p1}
\end{align}
where the convergence to zero follows since $\frac{n}{ \log J_n} \to 0 $ (which implies   $ \frac{   r^{*2n}}{J_n} \to 0 $).
Now, using Theorem \ref{mo^n to 1} we have
$
          \tilde{m}_{A,n}=\tilde{m}_{n,J_n}^{-n} J_n^{-1} \sum_{j=1}^{J_n}  Z_{n,j}  \xrightarrow[]{p} m_A.
$
The proof of the theorem follows using Theorem \ref{ro^n to 1}, Lemma \ref{CLT4ma-1}, and Slutsky's theorem.
\end{proof}

\subsection{Proof of Theorem \ref{joint-CLT}}

The final result of this section concerns the proof of Theorem \ref{joint-CLT}. The proof relies on a covariance calculation which is provided in Lemma \ref{cov(TU)} in Appendix \ref{apd-c-joint-CLT}. Since the proof of Theorem \ref{joint-CLT-2} is similar, it is included in Appendix \ref{apd-c-joint-CLT-2}.

\begin{proof}
Let 
$
    \mathcal{N}_{\tau_2,n}^*= (m^*-1)^{-1} m^{*\tau_2}(m^{*(n-\tau_2 +1)} -1)     
$
and
\begin{align}
    T_{\tau_1,\tau_2,n}= \sqrt{\delta_{m}J_n}  ( \hat{m}_{\tau_1,\tau_2,n} - m^*) , \;
    S_{\tau_1,\tau_2,n}= \frac{1}{\hat{r}^n} \sqrt{J_n}  ( \hat{m}_{A,\tau_1,\tau_2,n} - m_A). \label{joint-CLT-p1}
\end{align}
For simplicity, we denote $(\hat{\mathcal{N}}_{\tau_1,\tau_2,n} , T_{\tau_1,\tau_2,n} , S_{\tau_1,\tau_2,n} )$ as $( \hat{\mathcal{N}}_{n} ,T_{\tau_1,n} , S_{\tau_1,n} )$. Now  
\begin{align}
    S_{\tau_1,n} 
     &= \left( \frac{r^{*n}}{\hat{r}^n}    \frac{\mathcal{N}_{\tau_2,n}^*}{\hat{\mathcal{N}}_{n}}  \right)
     \left[\frac{1}{r^{*n}}   \frac{1}{\sqrt{J}} \sum_{j=1}^{J}   \left(\mathcal{N}_{\tau_2,n}^*\sum_{l=\tau_2}^n Z_{l,j}- m_A \right) \right] + m_A \frac{r^{*n}}{\hat{r}^n}  \frac{1}{r^{*n}}\sqrt{J} \left( \frac{\mathcal{N}_{\tau_2,n}^*}{\hat{\mathcal{N}}_{n}}-1  \right) \nonumber \\
  &\coloneqq \mathcal{R}_n  U_{n,1} +U_{n,2}. \label{joint-CLT-p2}
\end{align}
To prove the Theorem, we use the Cramer-wold device. To this end, for any $\lambda_1,\lambda_2 $, we have 
\begin{align*}
 \begin{pmatrix} 
 \lambda_1 , \lambda_2
 \end{pmatrix}
      \begin{pmatrix}
         T_{\tau_1,n}  \\       S_{\tau_1,n}
      \end{pmatrix} 
 &= 
 \begin{pmatrix} 
 \lambda_1 , \lambda_2
 \end{pmatrix}
       \begin{pmatrix}
         T_{\tau_1,n} \\       \mathcal{R}_n  U_{n,1}
      \end{pmatrix} 
    +
     \begin{pmatrix} 
 \lambda_1 , \lambda_2
 \end{pmatrix}
        \begin{pmatrix}
         0 \\        U_{n,2}
      \end{pmatrix} .
\end{align*}   
By Lemma \ref{CLT4maC1} we know that $U_{n,2} \xrightarrow[]{p} 0$, as $n \to \infty $. Thus, consider the first term,
\begin{align*}
     \begin{pmatrix} 
 \lambda_1 , \lambda_2
 \end{pmatrix}
       \begin{pmatrix}
         T_{\tau_1,n} \\       \mathcal{R}_n  U_{n,1}
      \end{pmatrix}  
      &=
       \begin{pmatrix} 
 \lambda_1 , \lambda_2
 \end{pmatrix}
       \begin{pmatrix}
         T_{\tau_1,n} & 0  \\ 0 &  U_{n,1}
      \end{pmatrix} 
           \begin{pmatrix} 
 1 \\     \mathcal{R}_n  
 \end{pmatrix}.
\end{align*}
By using Lemma \ref{cov(TU)} in Appendix, we have $ \bm{Cov}(T_{\tau_1,n} , U_{n,1} ) \to 0$, as $n \to \infty $. 
Also, by Theorem \ref{consis-CLT-mo}, Theorem \ref{consis-CLT4mataunJ}, Lemma \ref{mathcal_N-p1}, and Theorem \ref{ro^n to 1}, it follows that
\begin{align*}
     \begin{pmatrix} 
 \lambda_1 , \lambda_2
 \end{pmatrix}
       \begin{pmatrix}
         T_{\tau_1,n} \\       \mathcal{R}_n  U_{n,1}
      \end{pmatrix}  
      &=
       \begin{pmatrix} 
 \lambda_1 ,\lambda_2
 \end{pmatrix}
       \begin{pmatrix}
         T_{\tau_1,n} & 0  \\ 0 &  U_{n,1}
      \end{pmatrix} 
           \begin{pmatrix} 
 1 \\     \mathcal{R}_n  
 \end{pmatrix} 
 \xrightarrow[]{d}N( 0, \lambda_1^2 \sigma^{2*} +\lambda_2^2 \sigma^2_{\tau} ),  ~~\text{as } n \to \infty.
\end{align*}
Then this completes the proof, since
\begin{align*}
     \begin{pmatrix}
         T_{\tau_1,n} \\      S_{\tau_1,n}
      \end{pmatrix}  
      \xrightarrow[]{d}
          N \left( \begin{pmatrix} 
    0 \\ 0 \end{pmatrix} ,
     \begin{pmatrix}
    \sigma^{2*} & 0 \\
    0 &  \sigma^2_{\tau}
    \end{pmatrix}
    \right).
\end{align*}
\end{proof}

\section{Estimation of the Variance}
\label{sec-var}

In this section, we discuss the estimation of the limiting variance of the estimator of $\tilde{m}_{A,n}$. By Theorem \ref{consis-CLT4ma}, the limiting variance of centered and scaled $\tilde{m}_{A,n} $ is $\frac{\mathfrak{D}}{m_{2}^*} $. 
\begin{align*} 
      \tilde{\kappa}_{n} \coloneqq  \frac{ 1}{\tilde{r}_{n}^{2n}}   \frac{1}{J_n} \sum_{j=1}^{J_n} \left( \frac{Z_{n,j}   }{ \tilde{m}_{n}^n } -\tilde{m}_{A,n}  \right)^2.
\end{align*}
\begin{theorem}
\label{asymp_varma}
Suppose \ref{AssumA1}-\ref{Assum-mo4+2delta} hold and  $ \frac{n}{ \log J_n} \to 0 $. Then, as $n \to \infty$, $\tilde{\kappa}_{n} \xrightarrow[]{p}\frac{\mathfrak{D}}{m_{2}^*} $.
\end{theorem} 
\begin{proof}
By adding and subtracting $m_A $, it follows that
\begin{align}
     \frac{1}{r^{*2n}} \frac{1}{J_n}  \sum_{j=1}^{J_n} \left( \frac{Z_{n,j}   }{ \tilde{m}_{n}^n } -\tilde{m}_{A,n}  \right)^2 
     &= \frac{1}{r^{*2n}} \frac{1}{J_n}  \sum_{j=1}^{J_n} \left( \frac{Z_{n,j}   }{ \tilde{m}_{n}^n } - m_A\right)^2 + \frac{1}{r^{*2n}} \frac{1}{J_n}  \sum_{j=1}^{J_n} \left( m_A-\tilde{m}_{A,n}   \right)^2 + \nonumber \\
     &\quad \frac{1}{r^{*2n}}\frac{1}{J_n}  \sum_{j=1}^{J_n} \left( \frac{Z_{n,j}   }{ \tilde{m}_{n}^n } - m_A\right)\left(m_A-\tilde{m}_{A,n}  \right) . \label{fkD3}
\end{align}
First notice that by (\ref{varzn1})
\begin{align}
    \bm{Var} (Z_{n,1})=m_{2}^{*(n-1)} \mathfrak{D}_n + m^{*2n} \sigma_A^2
\end{align}
and by $ \mathfrak{D}_n \nearrow \mathfrak{D} $ in (\ref{varzn1-D}) , we have
\begin{align}
     \bm{E}  \left[ \frac{1}{r^{*2n}}\left( \frac{Z_{n,j}   }{ m^{*n} } - m_A \right)^2  \right] = \frac{1}{r^{*2n}}\bm{Var}  \left[  \frac{Z_{n,j}   }{ m^{*n} }   \right] 
     &= \frac{ m^{*2n}}{m_{2}^{*n}} \frac{m_{2}^{*(n-1)} \mathfrak{D}_n}{m^{*2n}} + \frac{1}{r^{*2n}} \frac{m^{*2n} \sigma_A^2}{m^{*2n}} \nonumber \\
     &= \frac{\mathfrak{D}_n}{m_{2}^*} + \frac{\sigma_A^2}{r^{*2n}}  
 \to  \frac {\mathfrak{D} }{m_{2}^*}. \label{var(zn/m^n)-to-D}
\end{align}
The first term on the RHS of (\ref{fkD3}) can be decomposed as $T_n(1)+T_n(2)+T_n(3)$, where
\begin{align*}
   T_n(1)&=  \frac{1}{r^{*2n}} \frac{1}{J_n}  \frac{m^{*2n}  }{ \tilde{m}_{n}^{2n} } \sum_{j=1}^{J_n} \left( \frac{Z_{n,j}   }{ m^{*n} } - m_A\right)^2 ~~, ~~ 
T_n(2)= \frac{1}{r^{*2n}} \frac{1}{J_n} m_A^2  \sum_{j=1}^{J_n} \left( \frac{m^{*n}  }{ \tilde{m}_{n}^{n} } - 1 \right)^2, ~~\text{and} \\
T_n(3)&= \frac{1}{J_n} m_A  
\left( \frac{m^{*n}  }{ \tilde{m}_{n}^{n} } - 1 \right)  \frac{m^{*n}  }{ \tilde{m}_{n}^{n} }  \frac{1}{r^{*2n}} \sum_{j=1}^{J_n}  \left( \frac{Z_{n,j}   }{ m^{*n} } - m_A\right).
\end{align*}
By Theorem \ref{mo^n to 1}, Theorem \ref{consis-CLT4ma}, it follow that $T_n(2)+T_n(3)\xrightarrow[]{p}0 $. Now using Theorem \ref{mo^n to 1}, Theorem \ref{HL-zn/c-limit}, and Chebyshev's inequality, it follows that $ T_n(1) \xrightarrow[]{p} \frac{\mathfrak{D}}{m_{2}^*}$, hence we have
\begin{align*}
    T_n(1)+T_n(2)+T_n(3) \xrightarrow[]{p} \frac{\mathfrak{D}}{m_{2}^*}.
\end{align*}
The details of the above calculation are in Appendix \ref{apd-asymp_varma}. Turning to the second term on the RHS of (\ref{fkD3}), it follows from Theorem \ref{consis-CLT4ma}, that 
\begin{align*}
   \frac{1}{r^{*2n}} \frac{1}{J_n}  \sum_{j=1}^{J_n} \left( m_A-\tilde{m}_{A,n}   \right)^2 &=  \frac{1}{r^{*2n}}  \left( m_A-\tilde{m}_{A,n}   \right)^2  \xrightarrow[]{p} 0, ~~ \text{as}~~ n \to\infty.
\end{align*}
Finally, for the third term on the RHS of (\ref{fkD3}), note that
\begin{align*}
  &   \frac{1}{r^{*2n}} \frac{1}{J_n}  \sum_{j=1}^{J_n} \left( \frac{Z_{n,j}   }{ \tilde{m}_{n}^n } - m_A\right)\left(m_A-\tilde{m}_{A,n}  \right) \\
    =&  \left(m_A-\tilde{m}_{A,n}  \right)  \frac{m^{*n}}{\tilde{m}_{n}^n} \frac{1}{r^{*2n}} \frac{1}{J_n}\sum_{j=1}^{J_n} \left( \frac{Z_{n,j}   }{ m^{*n} } - m_A \right)   + \frac{1}{r^{*2n}} \left(m_A-\tilde{m}_{A,n}  \right)   m_A \left(  \frac{m^{*n}}{\tilde{m}_{n}^n} -1  \right).  
\end{align*}
Now, the first term on the RHS  converges to zero since $\tilde{m}_{A,n}$ is consistent for $m_A$, Theorem \ref{mo^n to 1}, $r^* >1$, and  (\ref{consis-CLT4ma-p1}). The same argument also yields the convergence of the second term.
This completes the proof of the Theorem.
\end{proof}
We next turn our attention to the estimation of $\sigma^2_A$. Notice that $\sigma^2_A$ is a complicated function of $\mathfrak{D}$, namely
\begin{align*}
\sigma_A^2= \frac{1-r^{*(-2)}}{m_{2}^*-m^{*2}   }\left[\mathfrak{D}- m_A \gamma^{2*} \left( \frac{1}{1-(m^*r^{*2})^{-1}} \right) \right] -m_A^2.
\end{align*}
In the above, 
a consistent estimator of $\mathfrak{D}$ is
\begin{align*}
    \tilde{\mathfrak{D}}_n=  \frac{\tilde{m}_{2,n}}{\tilde{r}_{n}^{2n}}   \frac{1}{J_n} \sum_{j=1}^{J_n} \left( \frac{Z_{n,j}   }{ \tilde{m}_{n}^n } -\tilde{m}_{A,n}  \right)^2 ,
\end{align*}
and $r$ can be  estimated using $\tilde{r}_n$ defined in (\ref{def-m_21-r1}) and its consistency follows from Theorem \ref{consis-CLT-mo} and Theorem \ref{consis-mo2}. Similarly, $m_2^*$ and $m*$ can be estimated using results from Section \ref{sec-mr}. However, using the BPRE data, it is not possible to estimate $\gamma^*$ without additional information. In cases when $\gamma ^*$ can be expressed explicitly in terms of the parameters of the offspring distribution, one can obtain explicit formulae for estimating the ancestral variance in the BPRE context. In the next subsection, we describe two explicit examples of offspring distributions frequently encountered in practice with an explicit representation of $\gamma^*$. 

\subsection{Beta-Binomial model}
We first assume the offspring distribution to be 1+$Bernoulli(p_{l,j})$ where $p_{l,j} $ are i.i.d following $Beta(\alpha,\beta))$. 
This model is useful for PCR model data analysis, discussed in Section \ref{sec-da} wherein the DNA molecule replication either produces one copy or none, which can be modeled by a Bernoulli distribution conditional on the environmental variable $p_{l,j}$. In this setting, using the moments of the Beta distribution, we have
\begin{align*}
    m^* =1+\bm{E}[ p_{l,j}] =\frac{2\alpha+\beta}{\alpha+\beta},   \quad
    \sigma^{2*} = \bm{Var}[ p_{l,j}]=\frac{\alpha\beta}{(\alpha+\beta)^2(\alpha+\beta+1)} ,
\end{align*}
and
\begin{align*}
    \gamma^{2*}
    = \bm{E}[ \bm{Var}(\xi_1|p_0)]
    =\bm{E}[p_0]-\bm{E}[p_0]^2 = 3m^*-m_{2}^*-2.
\end{align*}
Hence a consistent estimate of $\gamma^{2*} $ is given by 
\begin{align*}
\tilde{\gamma}^{2}=3\tilde{m}_{n}-\tilde{m}_{2,n}-2.
\end{align*}
Thus, the estimator of $\sigma^2_A$ is
\begin{align*}
    \tilde{\sigma}_A^2= \frac{1-\tilde{r}_{n}^{-2}}{\tilde{m}_{2,n}-\tilde{m}_{n}^2   }\left[\tilde{\mathfrak{D}}_n- \tilde{m}_{A,n} \tilde{\gamma}^{2} \left( \frac{1}{1-(\tilde{m}_n\tilde{r}_{n}^{2})^{-1}} \right) \right] -\tilde{m}_{A,n}^2.
\end{align*}

\subsection{Poisson-Gamma model}
As another example, consider the offspring distribution, which is conditioned on $\lambda_{l,j}$, $1+Poisson(\lambda_{l,j})$, where $\lambda_{l,j} $'s are i.i.d. $Gamma(\alpha,\beta)$. Then we have $ m^*=1+\alpha \beta$, $\sigma^{2*}=\alpha \beta^2$ and
\begin{align*}
    \gamma^{2*}
= \bm{E}[ \bm{Var}(\xi_1|\lambda_0)] 
    =\bm{E}[\lambda_0 ] =\alpha \beta .
\end{align*}
 Because $ \gamma^{2*}=m^*-1 $, using $ \tilde{\gamma}^{2} =\tilde{m}_{n}-1 $ our consistent estimator of $\sigma^2_A$ is
\begin{align*}
    \tilde{\sigma}_A^2= \frac{1-\tilde{r}_{n}^{-2}}{\tilde{m}_{2,n}-\tilde{m}_{n}^2   }\left[\tilde{\mathfrak{D}}_n- \tilde{m}_{A,n} \tilde{\gamma}^{2} \left( \frac{1}{1-(\tilde{m}_n\tilde{r}_{n}^{2})^{-1}} \right) \right] -\tilde{m}_{A,n}^2.
\end{align*}

\section{Numerical Experiment }
\label{sec-ne}

In this section, we implement our method and evaluate its performance using a series of numerical experiments. All our results are based on 5000 simulations, and all confidence levels are 95\%.
\subsection{Numerical experiment setting}
We consider two distinct settings. In the first setting, the number of ancestors is modeled as a lower truncated at zero Poisson random variable with a mean of $m_A=10$. The offspring distributions are  modeled as 1+Bernoulli$(p_{l,j} )$, where $p_{l,j} $ are i.i.d. following a $Beta(90, 10)$ distribution. In the second setting, the number of ancestors is modeled using the probability model $4+$ a negative binomial distribution with parameters $r=4$ and $p=0.4 $. The offspring distribution are conditionally 1+Poisson$(\lambda_{l,j})$, with $\lambda_{l,j}$s being i.i.d. $Gamma(10,0.03)$. 

For estimation, we use data from $\tau=12$ to $n=20$ generations. Following (\ref{apd-C}.\ref{ma-truevar}) the true value of the variance of $\hat{m}_{A,n}$ is given by 
\begin{align}
    \Lambda_{\tau,n,J}&= \frac{1}{J} \left[ \frac{m^*-1   }{ m^{*\tau}(m^{*(n - \tau +1)} -1)     } \right]^2  \left\{  \sum_{l=\tau}^n \left[ \left(1+2 \sum_{k=1}^{n-l}  m^{*k} \right) ( m_{2}^{*(l-1)} \mathfrak{D}_l + m^{*2l} \sigma_A^2  )\right]  \right\}   \label{lambda_taunJ}.
\end{align}
In Table  \ref{tab-PBB-matau-n=20-J<11} and Table \ref{tab-NGP-matau-n=20-J<11}, we provide a comparison of the true value of the variance of $\hat{m}_{A,n}$ to its estimate, namely,
\begin{align} 
     \hat{ \Lambda}_{\tau,n,J} =  \frac{1}{J-1} \sum_{j=1}^{J} \left[ \frac{1}{\hat{\mathcal{N}}_{\tau,n,J}}\left(\sum_{l=\tau}^n Z_{l,j} \right) -\hat{m}_{A,\tau,n,J}  \right]^2 .\label{e_lambda_taunJ}
\end{align}

The point estimate, the standard error based on the asymptotic limit distribution, and the 95\% confidence intervals are summarized in Table \ref{tab-PBB-matau-n=20-J<11} and Table \ref{tab-NGP-matau-n=20-J<11} for Poisson and negative binomial cases respectively.

We notice that, in general, the number of generations may not be large enough for large sample approximations. Additionally, the limiting variance has a complicated formula that can prevent an easy application in data analysis. Hence, we propose a useful bootstrap algorithm for point estimation and confidence interval for the ancestor mean. Considering that our theoretical results depend on a large generation size and various conditions in sampling scheme \ref{assum-tau_n}, we introduce a bootstrap method independent of these conditions as a competitive data analysis tool. We obtain $J$ replicates of the bootstrap BPRE by sampling with replacement. We denote the bootstrap samples of the BPRE by $\bm{Z}_{j^*}^{(b)} \coloneqq (Z_{\tau,j^*}^{(b)}, \ldots, Z_{n, j^*}^{(b)})$, where $j^*=1, \cdots, J$. We repeat this process $B$ times. In each iteration, we calculate the estimators $\{\hat{m}_{A,b} , 1\le b \le B \}$ using (\ref{def-ma_taunj}). Subsequently, we obtain the bootstrap variance and confidence intervals using the estimators $\{\hat{m}_{A,b}, 1\le b \le B\}$. Details of this bootstrap algorithm are outlined in Algorithm \ref{bootsrap_al}. 

\subsection{Results of the numerical experiment }
The full results of the numerical experiment are shown in Table \ref{tab-PBB-matau-n=20-J<11} and \ref{tab-NGP-matau-n=20-J<11} in Appendix \ref{apd-B}. $\bar{m}_{A,\tau,n,J} $ represents the mean of the point estimates $\hat{m}_{A,\tau,n,J}$. $\bar{m}_{A,B}$ and Var$_B$ represents the bootstrap mean and variance. All bootstrap results are computed based on 200 bootstrap samples. All confidence intervals (CI) have nominal $95\%$ coverage. Here CR gives the simulated coverage rates and ML gives the mean length of the confidence intervals over the 5000 simulations. B is for the bootstrap CI.  G is for the CI based on asymptotic normality; t is for the CI based on the asymptotic normality, using the t distribution. We also include an estimator $\bar{Z_0} $ using the observed ancestors $ Z_{0,j}$ as a benchmark for our method. 

\begin{algorithm}[ht]
\caption{Bootstrap Method}
\label{bootsrap_al}
\begin{algorithmic}[1]
\State Input B, Set b=1.
\Repeat
  \State   Draw bootstrap BPRE $\{ \bm{Z}_{1^*}^{(b)} ,\bm{Z}_{2^*}^{(b)} ,\ldots,\bm{Z}_{J^*}^{(b)}   \}$ by sampling with replacement.
    \State  Compute $\hat{m}_{n}^* $ and $\hat{m}_{A,n}^* $ with $\{ \bm{Z}_{1^*}^{(b)} ,\bm{Z}_{2^*}^{(b)} ,\ldots,\bm{Z}_{J^*}^{(b)}   \}$ using (\ref{def-mo_taunj}) and (\ref{def-ma_taunj}).
    \State  $ \hat{m}_{b} \gets \hat{m}_{n}^* $ and $\hat{m}_{A,b} \gets \hat{m}_{A,n}^* $.
  \State  $b \gets b+1$.
  \Until{$b> B $}
  \State   Let $\displaystyle 
    \bar{m}_{A,B}=\frac{1}{B}\sum_{b=1}^B\hat{m}_{A,b} , \quad  \text{Var}_B=\frac{1}{B} \sum_{b=1}^B ( \hat{m}_{A,b} -\bar{m}_{A,B} )^2  .
  $
  \State Let   $ \displaystyle 
        \Hat{F}_A(t)= \frac{1}{B} \sum_{b=1}^B I(   \hat{m}_{A,b} - \bar{m}_{A,B} <t).
    $ 
  \State Let   $
   \text{CI B} = (\bar{m}_{A,B} -  t_A(1-\alpha/2) , \bar{m}_{A,B} + t_A(\alpha/2) )
  $, 
   where $ t(\alpha)= \hat{F}^{-1}(\alpha) $ and $t_A(\alpha) = \hat{F}_A^{-1}(\alpha) $.
\end{algorithmic}
\end{algorithm}
Figure \ref{Es_plot} and \ref{CR-plot} compare our estimators of $\hat{m}_{A,\tau,n,J} $ with $m_A$ and unobservable benchmark for the Beta-Binomial model. From Figure \ref{Es_plot}, it is clear that for both small and large values of $J$, our estimator remains close to $m_A$. Figure \ref{CR-plot} indicates that the average coverage rate of confidence intervals approaches $95\%$ as $ J$ increases. A similar phenomenon also occurs for the Poisson-Gamma model (figure not presented). Also, from Table \ref{tab-PBB-matau-n=20-J<11} and \ref{tab-NGP-matau-n=20-J<11}, it is evident that for a smaller number of replicates $J$, the confidence interval (CI) utilizing the t-distribution exhibits a higher coverage rate than the CI using the Gaussian distribution.

\begin{figure}[ht]
\centering
\begin{minipage}{.5\textwidth}
\centering
\includegraphics[height=7cm]{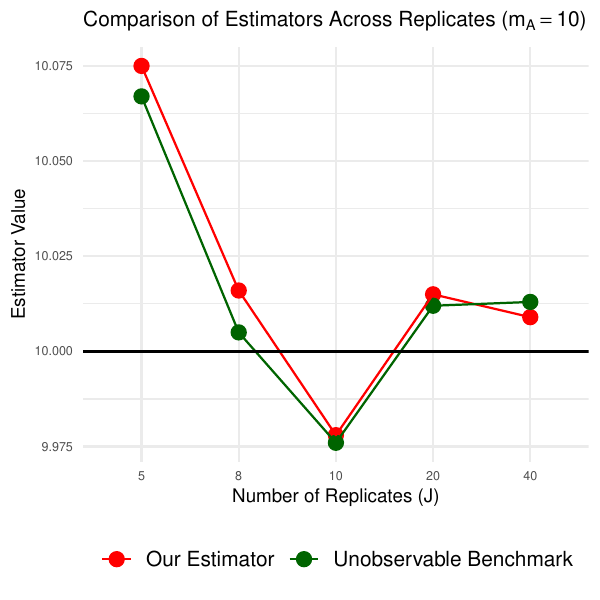} 
\caption{Estimator}
\label{Es_plot}
\end{minipage}\hfill
\begin{minipage}{.5\textwidth}
\centering
\includegraphics[height=7cm]{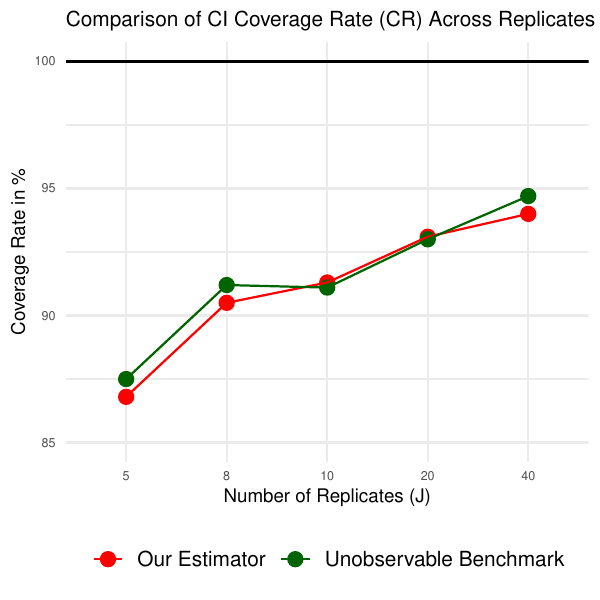} 
\caption{Coverage Rate}
\label{CR-plot}
\end{minipage}
\end{figure}

\subsection{Large offspring variance experiment}

By Theorem \ref{consis-CLT4ma}, we understand that the variance of our estimator diverges to infinity at a rate of $r^{*2n}$ as $n \to \infty$. This implies that if the offspring variance is too high compared to the offspring mean, the estimator may behave poorly when the number of replicates is small. To illustrate this, we conduct a numerical experiment where the offspring variance, $\sigma^{2*}$, is significantly larger than $m^*$.

In this experiment, as before, the number of ancestors follows a lower truncated  Poisson random variable with a mean of $m_A = 10$. The offspring distribution is, conditionally on $\lambda_{l,j}$, $1 + $Poisson$(\lambda_{l,j})$,  and $\lambda_{l,j}$ are i.i.d. $Gamma(0.4, 3)$. As a result, $r^* = 1.32$.

As shown in Table \ref{tab-PGP-matau-n=20-J<11}, the estimator's variance is high, resulting in a low coverage rate for the confidence interval. According Theorem \ref{consis-CLT4ma}, for our estimator to remain consistent, $\frac{n}{\log J_n} \to 0$ because we require $\frac{r^{*2n}}{J_n} \to 0$. However, since $r^*$ is large, achieving this necessitates many replicates.

\subsection{Comparing  the CIs from different sampling schemes in \ref{assum-tau_n}}

In this subsection, we provide a comparison of the limiting variances of $\hat{m}_{A,n}$ and the coverage rate (CR) of the CIs using bootstrap methods across all scenarios in sampling scheme \ref{assum-tau_n}. The comparisons are illustrated in Figure \ref{bvar-plot-c} and Figure \ref{bCR-plot-c}. From Figure \ref{bvar-plot-c} the variance estimators of case (i) and case (iii) are close to each other, while the variance for case (ii) is larger. Turning to the coverage rates, Figure \ref{bCR-plot-c} shows that case (iii) has the best coverage and hence is a preferred sampling scheme.

\begin{figure}[H]
    \begin{minipage}{0.45\textwidth}
        \centering
        \includegraphics[width=\textwidth]{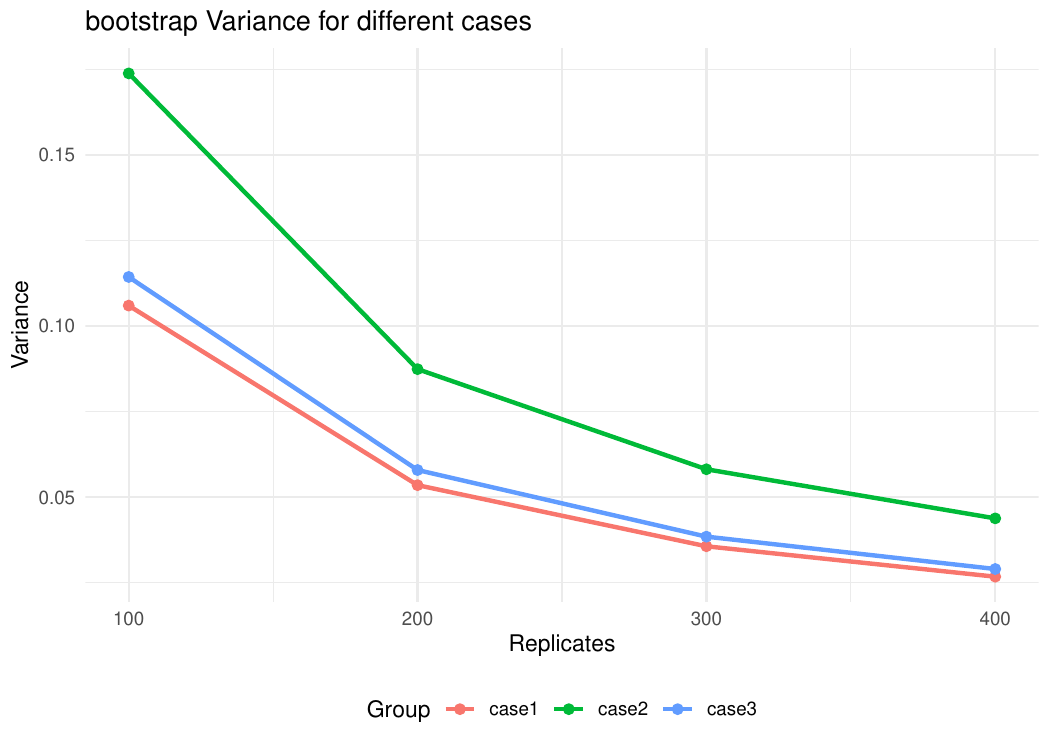}
        \caption{Bootstrap Variance}
        \label{bvar-plot-c}
    \end{minipage}\hfill
    \begin{minipage}{0.45\textwidth}
        \centering
        \includegraphics[width=\textwidth]{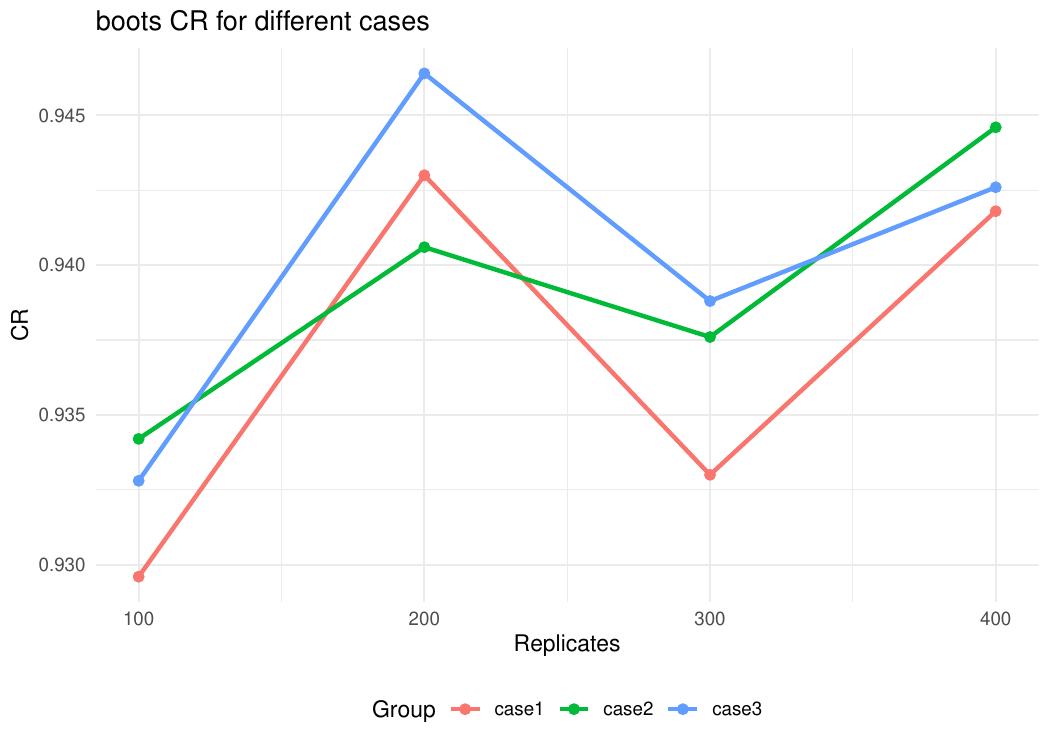}
        \caption{CR of Bootstrap CI}
        \label{bCR-plot-c}
    \end{minipage}
\end{figure}

\subsection{Comparison of learning methods}

For the learning method, we compare variances using Figures \ref{var-hist-Learning-case(i)} - Figure \ref{var-hist-Learning-case(iii)}. In these Figures, each set of bars represents the replicate size, and within a set, the orange represents the variance estimate from the entire data set (Theorem \ref{consis-CLT4mataunJ}), purple represents the variance of the estimate of $\hat{m}_{A,n}$ after learning $m^*$ from the first $\tau_1$ to $\tau_2$ generations (Theorem \ref{joint-CLT}), and blue represents the variance of the estimate of $\hat{m}_{A,n}$ where we learn $m^*$ from $\tau_2$ to $n$ generations (Theorem \ref{joint-CLT-2}). The results clearly illustrate the results from Theorem \ref{consis-CLT4mataunJ}, Theorem \ref{joint-CLT}, and Theorem \ref{joint-CLT-2}, namely that variance is minimized in the last case (blue).

We find that using earlier generations to estimate $m_A$ reduces the estimator’s variance while preserving the same convergence rate of the confidence interval for cases. This effect is more pronounced in cases (i) and (iii), where more samples are employed in the estimation. This tells that the learning method indeed helps reduce the variance of the estimator.

\begin{figure}[ht]
\centering
\begin{minipage}{.3\textwidth}
\centering
\includegraphics[height=3.6cm]{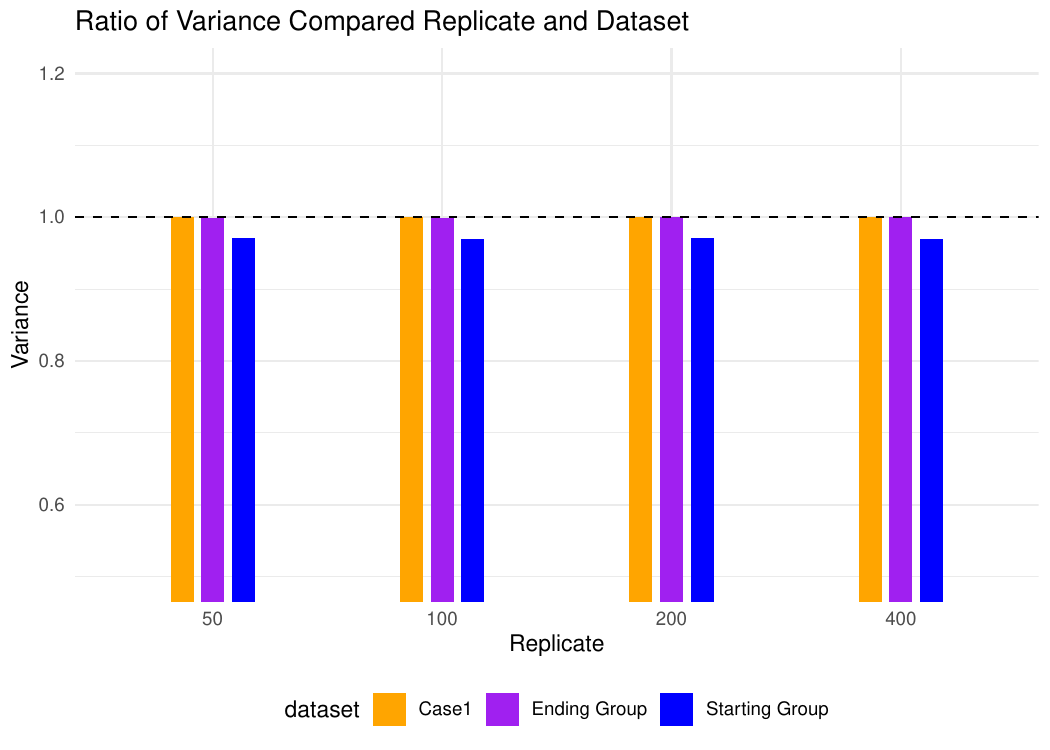} 
\caption{Case (i)}
\label{var-hist-Learning-case(i)}
\end{minipage}\hfill
\begin{minipage}{.3\textwidth}
\centering
\includegraphics[height=3.6cm]{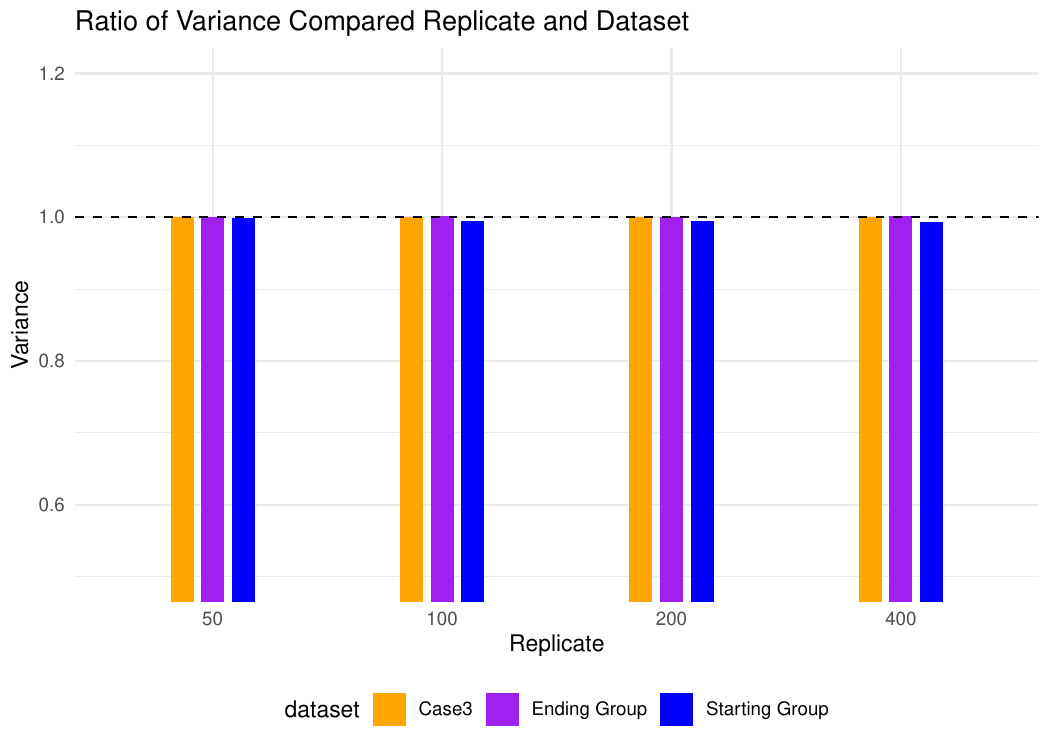} 
\caption{Case (ii)}
\label{var-hist-Learning-case(ii)}
\end{minipage}\hfill
\begin{minipage}{.3\textwidth}
\centering
\includegraphics[height=3.6cm]{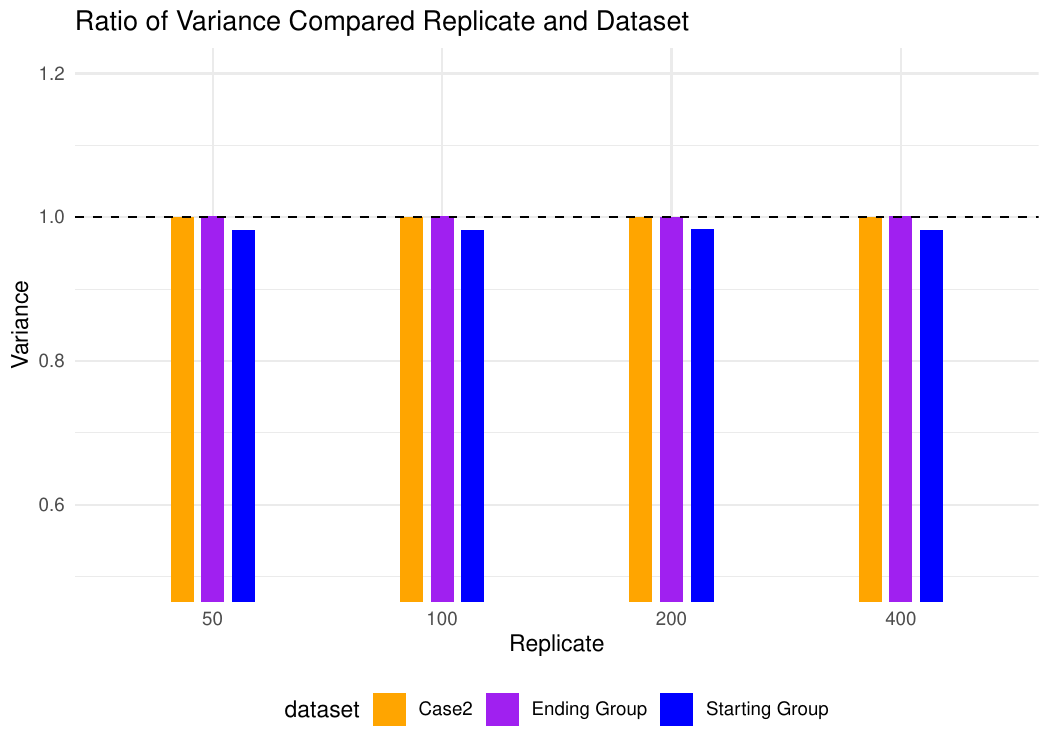} 
\caption{Case (iii)}
\label{var-hist-Learning-case(iii)}
\end{minipage}
\end{figure}

\section{Data Analysis}
\label{sec-da}

\subsection{Polymerase chain reaction experiments}
We return to the PCR experiments discussed in Section \ref{sec-intro}. 
We let $Z_{l,j} $ denote the number of target DNA molecules in the $j^{\text{th}}$ well after $ l$ cycles. Following the ideas in \cite{Bret-Vidya2011} and \cite{Li-Vidya2019}, we assume the relationship between the fluorescence intensity $F_{n,j}$ and the number of target DNA molecules $Z_{n,j}$ follows the formula $F_{n,j}=cZ_{n,j} $, where $c$ depends on the amplicon size and the calibration factor, which represents the number of nanograms of double-stranded DNA per fluorescence unit. Since thermal cycling may either lead to duplication or not, and this depends on several experimental factors, we model it as a Bernoulli distribution whose success probability in the $l^{\text{th}}$ cycle of the $j^{\text{th}}$ well $p_{l,j}$ as a random variable. Hence, the PCR evolution is modeled using a BPRE with conditional offspring distribution $ \bm{p}( \cdot|\zeta_{n,j})$, which accounts for both within and between reaction variability.

We analyze two experimental datasets from \cite{Bret-Vidya2011}. In both experiments, the true value of the relative quantitation is known because the calibrator is obtained by diluting the target by a known factor $R$.  Using Theorem \ref{relative-theorem}, we estimate $R$ by $\hat{R}_{n} $ defined in (\ref{r_est}) and its variance by
\begin{align}
 \hat{ \Lambda}_{R,\tau,n,J} = \frac{1}{J} \hat{R}_{n}^2 \left( \frac{\hat{ \Lambda}_{T,\tau,n,J} }{\hat{m}_{A,\tau,n,J,T}^2 }+   \frac{\hat{ \Lambda}_{C,\tau,n,J}}{\hat{m}_{A,\tau,n,J,C}^2}   \right),
\end{align}
where $\hat{\Lambda}_{T,\tau,n,J}  $ and $ \hat{\Lambda}_{C,\tau,n,J}$ are determined by ($\ref{e_lambda_taunJ}$).

\subsubsection{LSTAR Model}
An alternative approach to quantitation has been suggested in a recent work (\cite{Hsu-Sherina20}) and is referred to as the Logistic Smooth Transition Auto-regressive (LSTAR) model. We compare our approach with the LSTAR method. The basic idea behind the LSTAR approach is to incorporate different auto-regressive structures to represent distinct regimes within a PCR. This is achieved using a lagged time series. The model aims to smoothly transition between regimes using a logistic function as the transition mechanism. However, the model does not provide explicit estimates of the parameters of the distribution of an initial number of DNA molecules.

The class of two-regime LSTAR models is given by:
\begin{align*}
    y_t= \bm{\rho}_1 \bm{X}_t (1-G(q_t,\gamma_L,c_L))+ \bm{\rho}_2  \bm{X}_t G(q_t,\gamma_L,c_L)+\epsilon_t, \quad{\text{where}}
\end{align*}
$\bm{X}_t=(1,y_{t-1},\ldots,y_{t-p})$ and for $i=1,2$, $\bm{\rho}_i= (\rho_{i,0},\rho_{i,1},\ldots,\rho_{i,p})$. The transitioning logistic function is given by
\begin{align*}
    G(q_t,\gamma_L,c_L) =\frac{1}{1+\exp(-\gamma_L (q_t-c_L )}.
\end{align*}
The quantity $p$ represents the order of the auto-regression and $q_t$ represents the threshold variable, which allows for a gradual transition between regimens one and two. The threshold value, $c_L $, is the point of equality between the lower and upper regimes. $\gamma_L$ will specify how abruptly the switch between the two regimes occurs at $q_t=c$. The details of how these parameters are determined are described in Section 2.3 of \cite{Hsu-Sherina20}.

Turning to our dataset, we denote the fluorescence intensity of the $t^{\text{th}}$ cycle by $y_t$.  
For our analysis, we use the model:
\begin{align*}
     y_t=(\rho_{1,0}+\rho_{1,1} y_{t-1}+\rho_{1,2} y_{t-2}  ) + (\rho_{2,0}+\rho_{2,1} y_{t-1}+\rho_{2,2} y_{t-2}  )  G(y_{t-1},\gamma_L,c_L)+\epsilon_t,
\end{align*}
which we fit using the concentrated least squares methodology proposed in \cite{Leybourne1998}. We use the \textit{tsDyn} pacakge in R (\cite{tsDyn2009}).
Using the fitted model, we obtain estimators for $y_t$, denoted by $\hat{y}_t $ for $ t_1 \le t \le t_2 $. For $t< t_1 $, we can use the model again to estimate the $ y_t$ with the following function:
\begin{align*}
    \hat{y}_t= \frac{\hat{y}_{t+2}- (\hat{\rho}_{1,0}+\hat{\rho}_{1,1} \hat{y}_{t+1}) - (\hat{\rho}_{2,0}+\hat{\rho}_{2,1}  \hat{y}_{t+1})G(\hat{y}_{t+1},\hat{\gamma_L},\hat{c_L})}{ \hat{\rho}_{1,2} +  \hat{\rho}_{2,2} G(\hat{y}_{t+1},\hat{\gamma_L},\hat{c_L}) } 
\end{align*}
For PCR quantitation, we iteratively apply the above estimation method until $y_0$ is estimated. In the following sections, we use methods from the paper and the LSTAR method to analyze the aforementioned datasets.

\subsubsection{Luteinizing Hormone}
The first dataset is from a qPCR experiment of target material from Luteinizing Hormone taken from a mouse pituitary gland (\cite{TIEN2005105}). The datasets consist of 16 replicates of two dilution sets denoted as LH1 and LH2, where the LH2 product was obtained by diluting the LH1 product by a factor of $R=2.9505$. For our method, to specify the starting cycles $\tau_{1,j} $ and ending cycles $\tau_{2,j} $ of the exponential phase for every replicate, we use the algorithm that sets $F^*=0.2 $, $ m_c=1.5$ from section (6.1) in \cite{Bret-Vidya2011}. We refer to the result using GW process model as ``GW". One replicate of each set is not included for not reaching the detectable level (See Appendix \ref{apd-A}).

For the LSTAR model, we use only the positive fluorescence intensity data.  Additionally, since the duplication rate of DNA molecules can at most be 2, we fit the model using data collected after the growth rate has decreased to 2 or below.
Furthermore, since the estimated $y_t$ can be used, our method can also be applied to this estimate to assess any improvements. We refer to the result of this method in the tables as "LSTAR+BPRE." The results are summarized in Table \ref{LH-tab-lstar} and Table \ref{LH-tab-lstar-mo} below.

\begin{table}[H]
\centering
\begin{tabular}{|c|c|c|}
\hline  
&  Estimator  & CI \\
\hline  
Truth   & 2.95  & - \\
\hline  
   $\hat{R}_{n}$ & 2.9328 &(2.4947, 3.3708)   \\
\hline                  
LSTAR  &     4.459    & -   \\
\hline  
LSTAR+BPRE & 23.322       &  (18.656, 27.988)  \\
\hline
GW & 2.8221   & (1.8950, 3.7493) \\
\hline  
\end{tabular}
\caption{Comparison with LSTAR for Luteinizing Hormone (LH) experiment}
\label{LH-tab-lstar}
\end{table}

\begin{table}[H]
\centering
\begin{tabular}{|c|c|c|}
\hline  
   &   $\hat{m}_{n}$ (SE) & LSTAR (SE)\\
\hline
Group 1  & 1.760 (0.0221) & 1.144(0.0596) \\
\hline  
Group 2  & 1.764 (0.0184) &1.247(0.0675) \\
\hline
\end{tabular}
    \caption{Estimator for $m^* $ for Luteinizing Hormone (LH) experiment}
    \label{LH-tab-lstar-mo}
\end{table}

\subsubsection{Strongylus Vulgaris}
In the second dataset, the target material is Strongylus vulgaris (S.vulgaris) ribosomal DNA (\cite{NIELSEN2008443}), which consists of 10 replicates of two dilution sets, denoted as SV1 and SV2. The SV2 product was obtained by diluting the SV1 product by a factor of $R=10$. Using section (4.1.2) in \cite{Li-Vidya2019}, we set  $m_c=1.55 $ to determine the starting and ending cycles of the exponential phase.

For the LSTAR model, we will use the entire dataset because the positive portion alone is insufficient to adequately fit the model. This approach may result in negative fitted values for some replicates. To address this issue, we will disregard the negative fitted values for each replicate and focus only on the positive fitted values for quantitation. The results of our analysis are summarized below in Table \ref{SV-tab-lstar} and Table \ref{SV-tab-lstar-mo}.

\begin{table}[H]
\centering
\begin{tabular}{|c|c|c|}
\hline
   &  Estimator & CI \\
  \hline  
Truth & 10 &- \\
\hline 
   $\hat{R}_{n}$ &10.2180 & (8.995, 11.441)   \\
\hline                  
LSTAR  &     0.368    & -   \\
\hline  
LSTAR+BPRE & 1.154       &  (1.068 1.240)  \\
\hline  
GW & 11.8033 &  (-0.524, 24.130)\\
\hline  
\end{tabular}
    \caption{Comparison with LSTAR for Strongylus Vulgaris (SV) experiment}
    \label{SV-tab-lstar}
\end{table}

\begin{table}[H]
\centering
\begin{tabular}{|c|c|c|}
\hline  
   &   $\hat{m}_{n}$ (SE) & LSTAR (SE)\\
\hline
Group 1  & 1.672 (0.0178) & 1.416 (0.0786) \\
\hline  
Group 2  & 1.716 (0.0165) & 1.403 (0.1263) \\
\hline
\end{tabular}
    \caption{Estimator for $m^* $ for Strongylus Vulgaris (SV) experiment}
    \label{SV-tab-lstar-mo}
\end{table}

\subsubsection{PCR Data Analysis Discussion}
From the above results summary, we see that our model outperforms the LSTAR model. One potential reason is that the methodology used in the LSTAR model relies on fitting the model based on observed data and then extrapolating back to the ancestor. However, this method can lead to biased estimations if the observed dataset is unreliable, such as data from the initial phase. 
Moreover, when our method uses the fitted values from the LSTAR model, it tends to underestimate the marginal offspring mean because the LSTAR model includes data from the plateau phase, where the reaction's efficiency decreases. Another advantage of our method is that it provides a confidence interval for the estimator, offering insight into the precision and reliability of our estimation.

One challenge we encounter when fitting the LSTAR model is that, since the \textit{tsDyn} package in R, our only available tool, does not specifically cater to PCR quantitation, the algorithm does not always yield results for our dataset. Additionally, the fitted LSTAR model cannot guarantee that the estimated initial fluorescence intensity will be positive, thus failing to provide quantitation results for some replicates.

\subsection{COVID-19}
We first consider the COVID-19 dataset of two groups from Virginia, where each group consists of 8 counties (Figure \ref{gpp_plot}). 
As previously mentioned, only data from the stable growth period is considered. Therefore, we identify the starting point $\tau_j$ and endpoint $n_j$ of this stable growth period for each county based on the available data. We only utilize data after the accumulated cases exceed a threshold, $T_1$. However, for counties with large populations, the growth rate after this threshold might still be irregular,  so we ensure that the growth rate drops below $r_1$ after exceeding the threshold. This verification guarantees that cases from prior weeks have been resolved. To determine the end of the stable growth stage, when the growth rate falls below  $r_2$. For both analyses, we have $T_1=100 $, $r_1=2 $, and $r_2=1.05 $.
 In Table \ref{VA-pop}, Pop represents the population size considered for analysis and the Mean Pop is the average population in the particular group.$\hat{m}_n$ represents the marginal rate of growth of COVID-19 infections, and $\hat{m}_{A,n}$ represents the estimated initial number of subjects infected with COVID-19.
\begin{table}[H]
    \centering
    \begin{tabular}{|c|c|c|c|c|c|}
    \hline
         & Counties & Pop  & Mean Pop  &  $\hat{m}_n$ (SE) & $\hat{m}_{A,n}$ (SE) \\ \hline
       Group 1 & 8 & 89k - 1.1m & 98k  & 1.1788 (0.0110) & 24.7939 (0.1709) \\ \hline
       Group 2 & 8 & 37k - 41k  & 38.5k  & 1.1710 (0.0176) & 8.3633 (0.2424) \\ \hline
    \end{tabular}
    \caption{Virginia Population}
    \label{VA-pop}
\end{table}

It seems reasonable to assume that areas with higher population densities would have more potentially infecting subjects. Therefore, we compare the mean number of ancestral subjects in the group with the larger population to that in the group with the smaller population within each state. The results of this analysis are summarized in Table \ref{VA-ana}.

\begin{table}[H]
    \centering
    \begin{tabular}{|c|c|c|c|}
      \hline
       State  &    Population Ratio  & Relative Quantitation  &  CI bound      \\ \hline
       VA & 2.5436 & 2.9646 &   (1.2412 , 4.6880  )    \\ \hline
    \end{tabular}
    \caption{Virginia analysis}
    \label{VA-ana}
\end{table}
We repeat this analysis for the COVID-19 dataset from Maryland. The results of this analysis are summarized in Table \ref{MD-pop} and \ref{MD-ana}.

\begin{table}[H]
    \centering
    \begin{tabular}{|c|c|c|c|c|c|}
    \hline
         & Counties & Pop  & Mean Pop  &  $\hat{m}_n$ (SE) & $\hat{m}_{A,n}$ (SE) \\ \hline
       Group 1 & 11 & 150k - 1m & 480k  & 1.25525 (0.0178) & 199.70656 (0.3660) \\ \hline
       Group 2 & 10 & 30k - 120k  & 69k  & 1.197068 (0.0221) & 28.799474 (0.2627) \\ \hline
    \end{tabular}
    \caption{Maryland Population}
    \label{MD-pop}
\end{table}

\begin{table}[H]
    \centering
    \begin{tabular}{|c|c|c|c|}
      \hline
       State  &    Population Ratio  & Relative quantitation  &  CI bound      \\ \hline
       MD & 6.9805 &  6.93438 &   (0.81187, 13.05690  )    \\ \hline
    \end{tabular}
    \caption{Maryland analysis}
    \label{MD-ana}
\end{table}
From the analysis of two COVID-19 datasets, we note that areas with higher population densities had larger numbers of potentially infecting subjects during the early stages of the pandemic. This could be due to several factors; for example, counties with large populations may have more crowded gatherings, facilitating the spread of infection, especially before anti-epidemic measures were implemented.

We also include an analysis for comparing data from two states by analyzing counties with similar populations (90k–300k). Specifically, we examine the spreading rate and the mean of ancestral subjects to identify potential differences between the states. The results of this analysis are summarized in Table \ref{MD-VA-pop} and \ref{MD-VA-ana}.
\begin{table}[H]
    \centering
    \begin{tabular}{|c|c|c|c|c|c|}
    \hline
         & Counties & Pop  & Mean Pop  &  $\hat{m}_n$ (SE) & $\hat{m}_{A,n}$  (SE) \\ \hline
      MD Group & 9 & 92k - 260k &  157k &  1.2391 (0.0209) & 47.6166 (0.1487) \\ \hline
      VA Group & 16 & 92k - 245k  & 151k  & 1.2022 (0.0095) & 43.5927 (0.2989) \\ \hline
    \end{tabular}
    \caption{Two States Population}
    \label{MD-VA-pop}
\end{table}
\begin{table}[H]
    \centering
    \begin{tabular}{|c|c|c|c|}
      \hline
       State  &    Population Ratio  & Relative quantitation  &  CI bound      \\ \hline
       MD/VA & 1.039175 &  1.092308 &   (0.377669 1.806947  )    \\ \hline
    \end{tabular}
    \caption{Two States analysis}
    \label{MD-VA-ana}
\end{table}

The results show that Virginia generally had a lower growth rate in cases and slightly fewer ancestral subjects. This may be attributed to Virginia implementing stricter and more effective measures during the first year of the pandemic.
\section{Discussion}
\label{sec-dis}

In this paper, we discuss the theoretical findings and real-world applications of the replicated BPRE model. We demonstrate that, under appropriate conditions, the number of replicates and generations diverges to infinity, and our estimators of offspring mean and ancestral mean asymptotically converge to two uncorrelated normal distributions. The numerical experiment results illustrate that the proposed model and method yield substantial improvements under different settings of ancestral and offspring distributions.

A key component of the proposed model is that it accounts for both between and within reaction variability in PCR experiments. In the context of PCR experiments, it yields substantial improvements compared to other existing methods.
The method can be applied to other situations, as illustrated by our application to COVID-19 evolution.  

A natural extension is to extend our approach to the BPRE model with immigration, given that the growth in the number of cases includes both internal infections and imported cases. Overall, our methods show promise for routine use in applications requiring ancestral inference.

\newpage


\begin{thebibliography}{}

\bibitem[Atanasov et~al., 2021]{Atanasov2021}
Atanasov, D., Stoimenova, V., and Yanev, N.~M. (2021).
\newblock Branching process modelling of covid-19 pandemic including immunity and vaccination.
\newblock {\em Stochastics and Quality Control}, 36(2):157--164.

\bibitem[Athreya and Ney, 1972]{athreya2004branching}
Athreya, K.~B. and Ney, P.~E. (1972).
\newblock {\em Branching processes}, volume Band 196 of {\em Die Grundlehren der mathematischen Wissenschaften}.
\newblock Springer-Verlag, New York-Heidelberg.

\bibitem[Chow and Teicher, 1988]{Chow-Ti1978}
Chow, Y.~S. and Teicher, H. (1988).
\newblock {\em Probability Theory: Independence, Interchangeability, Martingales}.
\newblock Springer texts in statistics. Springer.

\bibitem[Curry et~al., 2002]{Curry2002FactorsIR}
Curry, J.~D., Mchale, C., and Smith, M.~T. (2002).
\newblock Factors influencing real-time rt-pcr results: Application of real-time rt-pcr for the detection of leukemia translocations.
\newblock {\em Molecular Biology Today}, 3:79--84.

\bibitem[{Fabio Di Narzo} et~al., 2009]{tsDyn2009}
{Fabio Di Narzo}, A., Aznarte, J., and Stigler, M. (2009).
\newblock {\em tsDyn: Time series analysis based on dynamical systems theory}.
\newblock R package version 0.7.

\bibitem[Follestad et~al., 2010]{Follestad2010}
Follestad, T., Jørstad, T.~S., Erlandsen, S.~E., Sandvik, A.~K., Bones, A.~M., and Langaas, M. (2010).
\newblock A bayesian hierarchical model for quantitative real-time pcr data.
\newblock {\em Statistical Applications in Genetics and Molecular Biology}, 9(1):3.

\bibitem[Francisci and Vidyashankar, 2024]{Fran-Vidya2024}
Francisci, G. and Vidyashankar, A.~N. (2024).
\newblock Branching processes in random environments with thresholds.
\newblock {\em Advances in Applied Probability}, 56(2):495–544.

\bibitem[Goll et~al., 2006]{Goll2006}
Goll, R., Olsen, T., Cui, G., and Florholmen, J. (2006).
\newblock Evaluation of absolute quantitation by nonlinear regression in probe-based real-time pcr.
\newblock {\em BMC Bioinformatics}, 7:107 -- 107.

\bibitem[Grama et~al., 2017]{Grama-Liu2017}
Grama, I., Liu, Q., and Miqueu, E. (2017).
\newblock Harmonic moments and large deviations for a supercritical branching process in a random environment.
\newblock {\em Electron. J. Probab.}, 22:1 -- 23.

\bibitem[Hall and Heyde, 1980]{B-Hall-Heyde1980}
Hall, P. and Heyde, C.~C. (1980).
\newblock {\em Martingale limit theory and its application}.
\newblock Probability and Mathematical Statistics. Academic Press, Inc. [Harcourt Brace Jovanovich, Publishers], New York-London.

\bibitem[Hanlon, 2009]{hanlon2009contributions}
Hanlon, B. (2009).
\newblock {\em Contributions To Ancestral Inference For Supercritical Branching Processes And High-Dimensional Data Analysis}.
\newblock PhD thesis, Cornell University.

\bibitem[Hanlon and Vidyashankar, 2011]{Bret-Vidya2011}
Hanlon, B. and Vidyashankar, A.~N. (2011).
\newblock Inference for quantitation parameters in polymerase chain reactions via branching processes with random effects.
\newblock {\em J. Amer. Statist. Assoc.}, 106(494):525--533.

\bibitem[Hsu et~al., 2020]{Hsu-Sherina20}
Hsu, B., Sherina, V., and McCall, M.~N. (2020).
\newblock {Autoregressive modeling and diagnostics for qPCR amplification}.
\newblock {\em Bioinformatics}, 36(22-23):5386--5391.

\bibitem[Huang and Liu, 2012]{Huang-Liu2012}
Huang, C. and Liu, Q. (2012).
\newblock Moments, moderate and large deviations for a branching process in a random environment.
\newblock {\em Stochastic Process. Appl.}, 122(2):522--545.

\bibitem[Huang et~al., 2023]{Huang-Wang2023}
Huang, C., Wang, C., and Wang, X. (2023).
\newblock Moments and asymptotic properties for supercritical branching processes with immigration in random environments.
\newblock {\em Stoch. Models}, 39(1):21--40.

\bibitem[Kersting and Vatutin, 2017]{Kersting2017DiscreteTB}
Kersting, G. and Vatutin, V. (2017).
\newblock {\em Discrete Time Branching Processes in Random Environment}.
\newblock Dover Publications.

\bibitem[Lalam, 2009]{L09A}
Lalam, N. (2009).
\newblock A quantitative approach for polymerase chain reactions based on a hidden {M}arkov model.
\newblock {\em Journal of Mathematical Biology}, 59(4):517--533.

\bibitem[Leybourne et~al., 1998]{Leybourne1998}
Leybourne, S., Newbold, P., and Vougas, D. (1998).
\newblock Unit roots and smooth transitions.
\newblock {\em Journal of Time Series Analysis}, 19(1):83--97.

\bibitem[Livak and Schmittgen, 2001]{LIVAK2001402}
Livak, K.~J. and Schmittgen, T.~D. (2001).
\newblock Analysis of relative gene expression data using real-time quantitative pcr and the 2[delta][delta]ct method.
\newblock {\em Methods}, 25(4):402--408.

\bibitem[Nielsen et~al., 2008]{NIELSEN2008443}
Nielsen, M.~K., Peterson, D.~S., Monrad, J., Thamsborg, S.~M., Olsen, S.~N., and Kaplan, R.~M. (2008).
\newblock Detection and semi-quantification of strongylus vulgaris dna in equine faeces by real-time quantitative pcr.
\newblock {\em International Journal for Parasitology}, 38(3):443--453.

\bibitem[Saha et~al., 2007]{Saha2007}
Saha, N., Watson, L., Kafadar, K., Ramakrishnan, N., Onufriev, A., Mane, S., and Vasquez-Robinet, C. (2007).
\newblock Validation and estimation of parameters for a general probabilistic model of the pcr process.
\newblock {\em Journal of computational biology : a journal of computational molecular cell biology}, 14:97--112.

\bibitem[Tanny, 1977]{Tanny1977}
Tanny, D. (1977).
\newblock Limit theorems for branching processes in a random environment.
\newblock {\em The Annals of Probability}, 5(1):100--116.

\bibitem[Tien et~al., 2005]{TIEN2005105}
Tien, J.~H., Lyles, D., and Zeeman, M.~L. (2005).
\newblock A potential role of modulating inositol 1,4,5-trisphosphate receptor desensitization and recovery rates in regulating ovulation.
\newblock {\em Journal of Theoretical Biology}, 232(1):105--117.

\bibitem[Vidyashankar and Li, 2019]{Li-Vidya2019}
Vidyashankar, A.~N. and Li, L. (2019).
\newblock Ancestral inference for branching processes in random environments and an application to polymerase chain reaction.
\newblock {\em Stoch. Models}, 35(3):318--337.

\bibitem[Yanev et~al., 2020]{Yanev2020}
Yanev, N., Stoimenova, V., and Atanasov, D. (2020).
\newblock Stochastic modeling and estimation of covid-19 population dynamics.
\newblock {\em arXiv: Methodology}.

\end{thebibliography}

\newpage

\appendix
\appendixpage

\section{Starting and Ending Cycles for data anlyasis} \label{apd-A}

\subsection{Luteinizing Hormone data}
\begin{table}[H]
\centering
\begin{tabular}{|c|ccccccccccccccc |}
\hline 
j &  1 & 2 & 3 & 4 & 5 & 6 & 7 & 8 & 9 & 10 & 11 & 12 & 13  & 14 & 15 \\
\hline
LH1  $\tau_{1,j}$  &  18 &  19 &  19  & 19  & 18   & 19 &  18  & 18  & 18  &  19  &  18  &  19  &  18   & 19   & 19  \\    
LH1  $\tau_{2,j}$  & 21 &  22  & 22 &  22  & 22  & 22  & 22  & 22  & 21  &  22   & 22  &  22   & 22 &   22   & 22 \\
\hline \hline  j &  1 & 2 & 3 & 4 & 5 & 6 & 7 & 8 & 9 & 10 & 11 & 12 & 13  & 14 & 15 \\     
\hline
LH2  $\tau_{1,j}$ &   21  & 21 &  20 &  20  & 20 &  20 &  20  & 20  & 20 &   19  &  20  &  20  &  20 &   20  &  20 \\ 
LH2 $\tau_{2,j}$ & 24  & 23  & 23 &  23 &  23  & 23  & 24  & 23  & 24   & 23   & 24  &  24  &  24  &  24  &  24 \\
\hline
\end{tabular}
\caption{The Starting cycles and Ending cycles for LH1 and LH2}
\end{table}

\subsection{Strongylus Vulgaris data}
\begin{table}[H]
\centering
\begin{tabular}{|c|cccccccccc |}
\hline 
j &  1 & 2 & 3 & 4 & 5 & 6 & 7 & 8 & 9 & 10 \\
\hline
SV1  $\tau_{1,j}$ & 31  & 31 &  30  & 30  & 31 &  31  & 31  & 30 &  30 &   31 \\
SV1 $\tau_{2,j}$  &  33 &  32  & 32 &  32  & 32 &  32  & 32 &  32  & 32   & 32 \\ 
\hline  \hline  SV2 j &  1 & 2 & 3 & 4 & 5 & 6 & 7 & 8 & 9 & 10 \\
\hline SV1  $\tau_{1,j}$ &  34 &  33 &  33 &  33 &  33 &  34 &  34  & 33 &  34  &  34 \\ 
SV1 $\tau_{2,j}$ & 35  & 35  & 34 &  35 &  35 &  35 &  35  & 35   &35  &  35 \\
\hline
\end{tabular}
\caption{The Starting cycles and Ending cycles for SV1 and SV2}
\end{table}

\pagebreak

\section{Numerical Experiment Result Tables} \label{apd-B}
\begin{table}[H]
\begin{center}
\begin{tabular}[ht]{|c| c | c  | c | c | c | c  | c | c | } 
\hline
J & 5 & 6 & 8 & 10 & 20 & 30 & 40 & 50  \\ 
\hline 
 $\bar{m}_{A,\tau,n,J}$ & 10.075 & 10.035 & 10.016 & 9.978  & 10.015 & 9.999 & 10.009 & 10\\ 
\hline 
$\hat{\Lambda}_{\tau,n,J}$  & 2.203 & 1.859 & 1.378 & 1.111  & 0.55 & 0.367 & 0.276 & 0.219 \\ 
\hline 
$ \Lambda_{\tau,n,J}$ & 2.204 & 1.836 & 1.377 & 1.102  & 0.551 & 0.367 & 0.275  & 0.220\\ 
\hline 
G CR & 0.868 & 0.888 & 0.905 & 0.913 & 0.931 & 0.931 & 0.94 & 0.937 \\ 
\hline 
G ML & 5.445 & 5.062 & 4.429 & 4.007  & 2.866 & 2.354 & 2.044 & 1.823 \\ 
\hline 
t CR & 0.945 & 0.944 & 0.95 & 0.944  & 0.947 & 0.942 & 0.947 & 0.941\\ 
\hline 
t ML & 7.713 & 6.639 & 5.344 & 4.625  & 3.06 & 2.456 & 2.109 & 1.869\\ 
\hline 
$\hat{m}_{A,b}$ & 10.08 & 10.041 & 10.016 & 9.979 & 10.015 & 9.999 & 10.009 & 10 \\ 
\hline 
Var$_B$ & 1.816 & 1.595 & 1.244 & 1.022  & 0.538 & 0.364 & 0.276 & 0.221\\ 
\hline 
B CR & 0.827 & 0.85 & 0.884 & 0.887 & 0.918 & 0.92 & 0.931 & 0.931 \\ 
\hline 
B ML & 4.765 & 4.528 & 4.084 & 3.734 & 2.765 & 2.29 & 1.998 & 1.789 \\ 
\hline 
$\bar{Z_0}$ & 10.067 & 10.028 & 10.005 & 9.976 & 10.012 & 9.991 & 10.013 & 10  \\ 
\hline 
var$(Z_0)$  & 2.002 & 1.68 & 1.246 & 1.004  & 0.498 & 0.333 & 0.251 & 0.199  \\ 
\hline 
$(Z_0)$ G CR  & 0.875 & 0.887 & 0.912 & 0.911  & 0.93 & 0.942 & 0.947 & 0.941 \\ 
\hline 
$(Z_0)$ G ML & 5.202 & 4.823 & 4.219 & 3.815  & 2.729 & 2.242 & 1.949 & 1.738\\ 
\hline 
$(Z_0)$ t CR & 0.95 & 0.948 & 0.95 & 0.944 & 0.945 & 0.951 & 0.954 & 0.945 \\ 
\hline 
$(Z_0)$ t ML & 7.369 & 6.325 & 5.09 & 4.403 & 2.915 & 2.34 & 2.012 & 1.782  \\ 
\hline 
\end{tabular}
\end{center} 
    \caption{Summary of the results for the simulation experiment of Beta-Binomial with Poisson Ancestor}
\label{tab-PBB-matau-n=20-J<11}
\end{table}

\begin{table}[H]
\begin{center}
\begin{tabular}[H]{|c| c | c  | c | c | } 
\hline
J & 20 & 30 & 40 & 50 \\ 
\hline 
$\bar{m}_{A,\tau,n,J}$ & 10.013 & 10.007 & 10.016 & 9.998 \\ 
\hline 
$\hat{ \Lambda}_{\tau,n,J}$  & 1.689 & 1.125 & 0.845 & 0.67 \\ 
\hline 
$ \Lambda_{\tau,n,J}$  & 1.811 & 1.207 & 0.905 & 0.724 \\ 
\hline 
G CR & 0.89 & 0.898 & 0.901 & 0.9 \\ 
\hline 
G ML & 4.917 & 4.057 & 3.536 & 3.159 \\ 
\hline 
t CR & 0.907 & 0.909 & 0.906 & 0.908 \\ 
\hline 
t ML & 5.25 & 4.233 & 3.649 & 3.239 \\ 
\hline 
$\hat{m}_{A,b}$ & 10.036 & 10.021 & 10.027 & 10.007 \\ 
\hline 
Var$_B$ & 2.011 & 1.357 & 1.031 & 0.828 \\ 
\hline 
B CR & 0.911 & 0.923 & 0.926 & 0.927 \\ 
\hline 
B ML & 5.264 & 4.37 & 3.829 & 3.442 \\ 
\hline 
$\bar{Z_0}$ & 9.996 & 9.998 & 9.988 & 9.993 \\ 
\hline 
var$(Z_0)$  & 0.75 & 0.498 & 0.373 & 0.299 \\ 
\hline 
$(Z_0)$ G CR  & 0.927 & 0.933 & 0.94 & 0.94 \\ 
\hline 
$(Z_0)$ G ML & 3.323 & 2.726 & 2.369 & 2.127 \\ 
\hline 
$(Z_0)$ t CR & 0.941 & 0.945 & 0.946 & 0.946 \\ 
\hline 
$(Z_0)$ t ML & 3.549 & 2.844 & 2.444 & 2.18 \\ 
\hline 
\end{tabular}
\end{center}
    \caption{Numerical experiment result for Gamma-Poisson with Negative-Binomial ancestor}
\label{tab-NGP-matau-n=20-J<11}
\end{table}

\begin{table}[H]
\begin{center}
\begin{tabular}[H]{|c| c | c  | c | c | } 
\hline
J & 20 & 30 & 40 & 50 \\ 
\hline 
$\bar{m}_{A,\tau,n,J}$ & 10.861 & 11.593 & 9.608 & 11.627 \\ 
\hline 
$\hat{ \Lambda}_{\tau,n,J}$  & 3208.119 & 7057.371 & 2061.778 & 3727.59 \\ 
\hline 
${ \Lambda}_{\tau,n,J}$  & 218061.3 & 145374.2 & 109030.7 & 87224.53 \\ 
\hline 
G CR & 0.335 & 0.364 & 0.375 & 0.386 \\ 
\hline 
G ML & 35.013 & 37.203 & 29.259 & 36.422 \\ 
\hline 
t CR & 0.344 & 0.369 & 0.381 & 0.39 \\ 
\hline 
t ML & 37.39 & 38.821 & 30.195 & 37.344 \\ 
\hline 
$\hat{m}_{A,b}$ & 15.322 & 14.473 & 11.584 & 12.966 \\ 
\hline 
Var$_b$ & 11353.52 & 22888.1 & 6976.552 & 5264.922 \\ 
\hline 
B CR & 0.684 & 0.664 & 0.654 & 0.647 \\ 
\hline 
B ML & 80.026 & 67.887 & 48.354 & 51.055 \\ 
\hline 
$\bar{Z_0}$ & 10.002 & 9.997 & 9.993 & 9.992 \\ 
\hline 
var$(Z_0)$  & 0.501 & 0.333 & 0.249 & 0.2 \\ 
\hline 
$(Z_0)$ G CR  & 0.937 & 0.942 & 0.945 & 0.945 \\ 
\hline 
$(Z_0)$ G ML & 2.737 & 2.242 & 1.943 & 1.744 \\ 
\hline 
$(Z_0)$ t CR & 0.952 & 0.95 & 0.954 & 0.951 \\ 
\hline 
$(Z_0)$ t ML & 2.923 & 2.34 & 2.005 & 1.789 \\ 
\hline 
\end{tabular}
\end{center}
    \caption{Numerical experiment result for Gamma-Poisson with Poisson ancestor}
\label{tab-PGP-matau-n=20-J<11}
\end{table}

Here we also include some experiment results for relative quantitation. Let the ancestors generated by 1+Poisson random variables with the mean of target group $m_{A,T}=1000$ and calibrator $m_{A,C}=100 $ hence $R=10 $.
For both groups, the offspring distribution are conditionally 1+Binomial with mean $p_{l,j} $ where $p_{l,j} $'s are i.i.d. $Beta(90,10)$.
We estimate $R $ by $ \hat{R}_{n,J} $ defined in (\ref{r_est}). 
All simulation results are based on 5000 repetitions.

\begin{table}[H]
\begin{center}
\begin{tabular}[H]{|c| c | c  | c | c | } 
 \hline
J & 20 & 30 & 40 & 50 \\ 
\hline 
 $\hat{R}$ & 10.005 & 10.003 & 10.003 & 9.999 \\ 
\hline 
$\hat{ \Lambda}_{R,\tau,n,J}$  & 0.104 & 0.069 & 0.052 & 0.041 \\ 
\hline 
$\Lambda_{R,\tau,n,J}$  & 0.103 & 0.069 & 0.052 & 0.041 \\ 
\hline 
G CR & 0.928 & 0.938 & 0.942 & 0.941 \\ 
\hline 
G ML & 1.251 & 1.024 & 0.887 & 0.794 \\ 
\hline 
t CR & 0.946 & 0.947 & 0.949 & 0.947 \\ 
\hline 
t ML & 1.336 & 1.069 & 0.916 & 0.815 \\ 
\hline 
$\hat{m}_{A,b}$ & 9.748 & 10.175 & 9.729 & 9.942 \\ 
\hline 
Var$_b$ & 0.104 & 0.07 & 0.053 & 0.043 \\ 
\hline 
b CR & 0.912 & 0.926 & 0.932 & 0.93 \\ 
\hline 
b ML & 1.213 & 1.007 & 0.875 & 0.786 \\ 
\hline 
\end{tabular}
\end{center}
\caption{Numerical experiment result Relative Quantitation ($n=20 $, $\tau=10$)}
\end{table}

\begin{table}[H]
\begin{center}
\begin{tabular}[H]{|c| c | c  | c | c | } 
 \hline
J & 20 & 30 & 40 & 50 \\ 
\hline 
 $\hat{R}$ & 10.009 & 10.002 & 10.002 & 10.002 \\ 
\hline 
$\hat{ \Lambda}_{1,n,J}$  & 0.128 & 0.086 & 0.064 & 0.051 \\ 
\hline 
Asy Var  & 0.128 & 0.085 & 0.064 & 0.051 \\ 
\hline 
G CR & 0.936 & 0.941 & 0.939 & 0.943 \\ 
\hline 
G ML & 1.391 & 1.142 & 0.987 & 0.885 \\ 
\hline 
t CR & 0.952 & 0.951 & 0.946 & 0.951 \\ 
\hline 
t ML & 1.486 & 1.192 & 1.019 & 0.907 \\ 
\hline 
$\hat{m}_{A,b}$ & 9.972 & 10.038 & 10.392 & 10.168 \\ 
\hline 
Var$_b$ & 0.127 & 0.086 & 0.065 & 0.052 \\ 
\hline 
b CR & 0.918 & 0.922 & 0.93 & 0.934 \\ 
\hline 
b ML & 1.343 & 1.112 & 0.971 & 0.874 \\ 
\hline
\end{tabular}
\end{center}
\caption{Numerical experiment result Relative Quantitation ($n=30 $, $\tau=15$)}
\end{table}

Let the ancestors generated by 1+Poisson random variables with mean of target group $m_{A,T}=3000$ and calibrator $m_{A,C}=100 $ hence $R=30 $

$n=30 $, $\tau=15$:

\begin{table}[H]
\begin{center}
\begin{tabular}[H]{|c| c | c  | c | c | } 
 \hline
J & 20 & 30 & 40 & 50 \\ 
\hline 
 $\hat{R}$ & 30.014 & 29.985 & 30.002 & 30.01 \\ 
\hline 
$\hat{ \Lambda}_{R,\tau,n,J}$  & 1.117 & 0.745 & 0.561 & 0.45 \\ 
\hline 
$\Lambda_{R,\tau,n,J}$  & 1.12 & 0.747 & 0.56 & 0.448 \\ 
\hline 
G CR & 0.935 & 0.949 & 0.939 & 0.943 \\ 
\hline 
G ML & 4.108 & 3.365 & 2.923 & 2.621 \\ 
\hline 
t CR & 0.952 & 0.959 & 0.944 & 0.949 \\ 
\hline 
t ML & 4.386 & 3.511 & 3.016 & 2.687 \\ 
\hline 
$\hat{m}_{A,b}$ & 29.859 & 29.408 & 28.476 & 29.945 \\ 
\hline 
Var$_b$ & 1.11 & 0.747 & 0.57 & 0.459 \\ 
\hline 
b CR & 0.917 & 0.935 & 0.924 & 0.929 \\ 
\hline 
b ML & 3.97 & 3.282 & 2.871 & 2.578 \\ 
\hline 
\end{tabular}
\end{center}
\caption{Numerical experiment result Relative Quantitation ($n=30 $, $\tau=15$)}
\end{table}

\newpage

\section{Supplement Proofs}
\label{apd-C}

In this appendix, we provide supplementary proofs for the results stated in the main paper. 

\subsection{Supplementary proof for Theorem \ref{joint-CLT}}
\label{apd-c-joint-CLT}
We begin with a supplementary proof of Theorem \ref{joint-CLT}. Recall that $ T_{\tau_1,\tau_2,n}= \sqrt{\delta_{m}J_n}  ( \hat{m}_{\tau_1,\tau_2,n} - m^*) $ and $  U_{n,1}=\frac{1}{r^{*n}}   \frac{1}{\sqrt{J}} \sum_{j=1}^{J}   \left(\mathcal{N}_{\tau_2,n}^*\sum_{l=\tau_2}^n Z_{l,j}- m_A \right)$. The next lemma shows that their covariance converges to 0 as $n \to \infty$.

\begin{lemma}
\label{cov(TU)}
Under the condition of Theorem \ref{joint-CLT}, as $n \to \infty$,
    $ \bm{Cov}(T_{\tau_1,\tau_2,n} , U_{n,1} ) \to 0$.
\end{lemma}
\begin{proof}  
For simplicity, we denote $(\hat{\mathcal{N}}_{\tau_1,\tau_2,n} , T_{\tau_1,\tau_2,n} , S_{\tau_1,\tau_2,n} )$ as $( \hat{\mathcal{N}}_{n} ,T_{\tau_1,n} , S_{\tau_1,n} )$ if otherwise not specified. First, consider $\bm{Cov}(T_{\tau_1,n} , U_{n,1} ) $, by replicates independence,
\begin{align}
    \bm{Cov}(T_{\tau_1,n} , U_{n,1} )   
   &=\frac{1}{J_nr^{*n} \sqrt{\delta_{m}} } \bm{Cov} \left[ \sum_{j=1}^{J_n}   \sum_{l=\tau_1+1}^{\tau_2} (\frac{Z_{l,j}}{Z_{l-1,j}} - m^*) , \sum_{j=1}^{J_n}  \left(  \frac{1}{\mathcal{N}_{\tau_2,n}^*} \sum_{l=\tau_2}^{n}   Z_{l,j}-m_A \right)\right] \nonumber \\
   &= \frac{1 }{r^{*n}\sqrt{\delta_{m}}} \bm{Cov} \left[   \sum_{l=\tau_1+1}^{\tau_2} (\frac{Z_{l,1}}{Z_{l-1,1}} - m^*) ,  \left(  \frac{1}{\mathcal{N}_{\tau_2,n}^*}\sum_{l=\tau_2}^{n}   Z_{l,1}-m_A \right)\right] \nonumber \\
      &=\frac{ 1 }{\mathcal{N}_{\tau_2,n}^*r^{*n}\sqrt{\delta_{m}}} \bm{E} \left[\left(   \sum_{l=\tau_1+1}^{\tau_2} \frac{Z_{l,1}}{Z_{l-1,1}}  \right)  \left( \sum_{l=\tau_2}^{n}   Z_{l,1} \right)\right] - \frac{  \sqrt{\delta_{m}} m^* m_A }{r^{*n}} .  \label{cov(TU)-p1}
\end{align}
Now using $r^*>1 $, the second term on the RHS of (\ref{cov(TU)-p1}) above converges to 0. To calculate the first term, notice that
\begin{align*}
    \bm{E} \left[\left(   \frac{Z_{n}^2}{Z_{n-1}}  \right) \right] 
    &=\bm{E} \left[ \bm{E} \left(  \frac{Z_n^2}{Z_{n-1}}    \middle| \mathcal{F}_{n-1} \right) \right]  = \bm{E} \left[ \frac{1}{Z_{n-1}}\bm{E} \left( (\sum_{k=1}^{Z_{n-1}} \xi_{n-1,k} )^2    \middle| \mathcal{F}_{n-1} \right) \right] \\
    &= \bm{E} \left[ \frac{1}{Z_{n-1}}\bm{E} \left( \sum_{k=1}^{Z_{n-1}} \xi_{n-1,k} ^2 + 2\sum_{1 \le k_1 < k_2 \le Z_{n-1}} \xi_{n-1,k_1} \xi_{n-1,k_2}    \middle| \mathcal{F}_{n-1} \right) \right] \\
    &=\bm{E} \left[  \sigma^2_{n-1,1}  + Z_{n-1}  m_{n-1,1}^2\right]
    = \gamma^{2*}+ m_{2}^*m^{*(n-1)}m_A.
\end{align*}
Follow the same logic,
\begin{align}
    \bm{E} \left[\left(   \frac{Z_{n}}{Z_{n-1}}  \right)  \left( Z_{n+k}\right)\right] =m^{*k} (  \gamma^{2*}+ m_{2}^*m^{*(n-1)}m_A ) \label{cov(TU)-p2}
\end{align}
Now let $\mathcal{M}_{\tau_1,\tau_2}^* = \sum_{l_1=\tau_1+1}^{\tau_2} m^{-*l_1} $ and recall that $\delta_m \coloneqq \delta_{m,n}=\tau_{2,n}-\tau_{1,n} $, for the first term on the RHS of (\ref{cov(TU)-p1}), we have
\begin{align*}
  & \quad \frac{1 }{ \mathcal{N}_{\tau_2,n}^* r^{*n}\sqrt{\delta_{m}}} \bm{E} \left[\left(   \sum_{l_1=\tau_1+1}^{\tau_2} \frac{Z_{l_1,1}}{Z_{l_1-1,1}}  \right)  \left( \sum_{l_2=\tau_2}^{n}   Z_{l_2,1} \right)\right] \\
    &=  \frac{1 }{ \mathcal{N}_{\tau_2,n}^* r^{*n}\sqrt{\delta_{m}}} \sum_{l_1=\tau_1+1}^{\tau_2} \sum_{l_2=\tau_2}^{n} \bm{E} \left[\left(  \frac{Z_{l_1,1}}{Z_{l_1-1,1}}  \right)  \left(    Z_{l_2,1} \right)\right]   \\
     &= \frac{1 }{ \mathcal{N}_{\tau_2,n}^* r^{*n}\sqrt{\delta_{m}}} \sum_{l_1=\tau_1+1}^{\tau_2} \sum_{l_2=\tau_2}^{n}  m^{*(l_2-l_1)} (  \gamma^{2*}+ m_{2}^*m_{o}^{*(l_1-1)}m_A ), ~~\text{by using (\ref{cov(TU)-p2})} \\
    &= \frac{1 }{ \mathcal{N}_{\tau_2,n}^* r^{*n}\sqrt{\delta_{m}}} \left( \sum_{l_1=\tau_1+1}^{\tau_2} \sum_{l_2=\tau_2}^{n}  m^{*(l_2-l_1)}  \gamma^{2*} +  \sum_{l_1=\tau_1+1}^{\tau_2} \sum_{l_2=\tau_2}^{n} m^{*(l_2-1)}m_{2}^*m_A \right)\\
    &= \frac{1 }{ \mathcal{N}_{\tau_2,n}^* r^{*n}\sqrt{\delta_{m}}}  \left( \gamma^{2*} \sum_{l_1=\tau_1+1}^{\tau_2} m^{-*l_1} \sum_{l_2=\tau_2}^{n}  m^{*l_2}   + \frac{1}{m^*} (\tau_2-\tau_1)m_{2}^*m_A \sum_{l_2=\tau_2}^{n} m^{*l_2} \right)\\
    &= \frac{1 }{ \mathcal{N}_{\tau_2,n}^* r^{*n}\sqrt{\delta_{m}}}  \mathcal{N}_{\tau_2,n}^{*} ( \gamma^{2*} \mathcal{M}_{\tau_1,\tau_2}^*   + \frac{\delta_{m} m_{2}^* m_A}{m^*}  ), ~~ \text{ by definition of }  \mathcal{N}_{\tau_2,n}^* \text{ and } \mathcal{M}_{\tau_1,\tau_2}^* \\
     &=  \frac{ 1 }{ r^{*n}\sqrt{\delta_{m}} } ( \gamma^{2*}   \mathcal{M}_{\tau_1,\tau_2}^*   + \frac{\delta_{m} m_{2}^* m_A }{m^*}  ) \\
    &=  \frac{ \gamma^{2*}  \mathcal{M}_{\tau_1,\tau_2}^* }{ r^{*n}\sqrt{\delta_{m}}  }  + \frac{m_{2}^* m_A}{m^*} \frac{\sqrt{\delta_{m}}}{ r^{*n}  } 
    \to 0.
\end{align*}
where the convergence above follows from $ r*>1$, $\mathcal{M}_{\tau_1,\tau_2}^* \le (m^*-1)^{-1}m^*  $, and $\delta_m < n $. This completes the proof of the lemma.

\end{proof}

\subsection{Proof for Theorem \ref{joint-CLT-2}}
\label{apd-c-joint-CLT-2}
Next, we provide a proof of Theorem \ref{joint-CLT-2}.

\begin{proof}
Similar to (\ref{joint-CLT-p1}) and (\ref{joint-CLT-p2}), define 
\begin{align*}
  &  \mathcal{N}_{\tau_1,\tau_2}^*= \frac{m^*-1   }{ m^{*\tau_1}(m^{*(\tau_2-\tau_1 +1)} -1)     } , \quad  
    T_{\tau_1,\tau_2,n}^{\prime}= \sqrt{\delta_{m}J_n}  ( \hat{m}_{\tau_1,\tau_2,n}^{\prime} - m^*) \\
  &  S_{\tau_1,\tau_2,n}^{\prime}= \frac{1}{\hat{r}^{\tau_2}} \sqrt{J_n}  ( \hat{m}_{A,\tau_1,\tau_2,n}^{\prime}  - m_A) , \quad 
    \mathcal{R}_n^{\prime} =  \frac{r^{*\tau_2}}{\hat{r}^{\tau_2}}     \frac{\hat{\mathcal{N}}_{\tau_1,\tau_2,n}^{\prime}}{\mathcal{N}_{\tau_1,\tau_2}^*} \\
  &  U_{n,1}^{\prime} = \left[\frac{1}{r^{*\tau_2}}   \frac{1}{\sqrt{J_n}} \sum_{j=1}^{J_n}   \left(\frac{1}{\mathcal{N}_{\tau_1,\tau_2}^*}\sum_{l=\tau_1}^{\tau_2} Z_{l,j}- m_A \right) \right], \quad  \text{and }
    U_{n,2}^{\prime} = m_A  \frac{r^{*\tau_2}}{\hat{r}^{\tau_2}}   \frac{1}{r^{*\tau_2}}\sqrt{J_n} \left( \frac{\hat{\mathcal{N}}_{\tau_1,\tau_2,n}^{\prime}}{\mathcal{N}_{\tau_1,\tau_2}^*}-1  \right).
\end{align*}
Then the rest of the proof follows, as before, as in the proof of Theorem \ref{joint-CLT}. Also while in Theorem \ref{joint-CLT}, $\bm{Cov}(T_{\tau_1,\tau_2,n} , U_{n,1} ) \to 0 $, in the current theorem we have for any $n$,
\begin{align*}
    \bm{Cov}(T_{\tau_1,\tau_2,n}^{\prime} , U_{n,1}^{\prime}  )  
   &=\frac{\sqrt{\delta_{m} }}{J_nr^{*\tau_2}} \bm{Cov} \left[ \sum_{j=1}^{J_n}  \frac{1}{\delta_{m}} \sum_{l=\tau_2+1}^n (\frac{Z_{l,j}}{Z_{l-1,j}} - m^*) , \sum_{j=1}^{J_n}  \left(  \frac{1}{\mathcal{N}_{\tau_1,\tau_2}^*} \sum_{l=\tau_1}^{\tau_2}   Z_{l,j}-m_A \right)\right] \\
   &=\frac{1 }{r^{*\tau_2}\sqrt{\delta_{m}}} \bm{E} \left[\left(   \sum_{l=\tau_2+1}^n (\frac{Z_{l,1}}{Z_{l-1,1}} - m^*) \right)  \left(  \frac{1}{\mathcal{N}_{\tau_1,\tau_2}^*}\sum_{l=\tau_1}^{\tau_2}   Z_{l,1}-m_A \right)\right] \\
   &= \frac{1 }{r^{*\tau_2}\sqrt{\delta_{m}}} \bm{E} \left[\left(  \frac{1}{\mathcal{N}_{\tau_1,\tau_2}^*}\sum_{l=\tau_1}^{\tau_2}   Z_{l,1}-m_A \right) \bm{E} \left[   \left(   \sum_{l=\tau_2+1}^n (\frac{Z_{l,1}}{Z_{l-1,1}} - m^*) \right) \middle| \mathcal{F}_{\tau_2}\right] \right]=0.
\end{align*}
This completes the proof of the Theorem.

\end{proof}

\subsection{Proof of Theorem \ref{relative-theorem}}
\label{apd-C-relative-theorem}
\begin{proof}
Let $ r_{T}^*>r_{C}^* $.  Notice that,
\begin{align*}
  \frac{\sqrt{J_n} (\hat{R}_{n}-R)}{r_{T}^{*n}+r_{C}^{*n}}  &=  \frac{1}{r_{T}^{*n}+r_{C}^{*n}}  \frac{\hat{m}_{A,T}m_{A,C}- \hat{m}_{A,C}m_{A,T}  }{\hat{m}_{A,C}m_{A,C}  }\\
 &= \frac{Rr_{T}^{*n}}{r_{T}^{*n}+r_{C}^{*n}}\left[ \frac{ m_{A,C} }{\hat{m}_{A,C} m_{A,T}}  \frac{ \sqrt{J_n}}{r_{T}^{*n}} (\hat{m}_{A,T}- m_{A,T}) - \frac{ r_{C}^{*n}  }{r_{T}^{*n} }\frac{1}{\hat{m}_{A,C} }  \frac{\sqrt{J_n}}{r_{C}^{*n}}(\hat{m}_{A,C}-m_{A,C}) \right]  \\
 &\xrightarrow[]{d} N( 0,  \frac{R^2 \sigma_{\tau_{T}}^2 }{m_{A,T}^2 }),
\end{align*}
where above convergence follows from Theorem \ref{consis-CLT4mataunJ} and Slutsky's theorem.

\noindent
Next if $ r_{T}^*< r_{C}^* $, the limiting distribution is,
\begin{align*}
    \frac{\sqrt{J_n} (\hat{R}_{n}-R)}{r_{T}^{*n}+r_{C}^{*n}}   \xrightarrow[]{d} N( 0,  \frac{R^2 \sigma_{\tau_{C}}^2 }{m_{A,C}^2 }).
\end{align*}
Finally, if $ r_{T}^*=r_{C}^* $, then
\begin{align*}
    \frac{\sqrt{J_n} (\hat{R}_{n}-R)}{r_{T}^{*n}+r_{C}^{*n}} \xrightarrow[]{d} N\left[ 0, \frac{R^2}{4}\left(\frac{\sigma_{\tau_{T}}^2 }{m_{A,T}^2}+\frac{\sigma_{\tau_{C}}^2 }{m_{A,C}^2 } \right) \right].
\end{align*}

\end{proof}

\subsection{Proof of Lemma \ref{var-sum-zn}}
\label{apd-c-var-sum-zn}
Next, we provide a proof of Lemma \ref{var-sum-zn}, preceded by a presentation of a required preliminary lemma.

\begin{lemma}
\label{varsumZl}
   Under condition of Lemma \ref{var-sum-zn}, for any $\tau $, 
\begin{align*}
     \bm{Var}   ( \sum_{l=\tau}^n Z_{l})  = \sum_{l=\tau}^n \left[ \left(1+2 \sum_{k=1}^{n-l}  m^{*k} \right)\bm{Var} (Z_l)\right]
\end{align*}
\end{lemma}

\begin{proof}
     Notice that,
\begin{align}
   \bm{Cov} \left( Z_{n} ,  Z_{n-1} \right) 
  &=  \bm{E}[ Z_{n-1} \bm{E}( Z_{n}   |\mathcal{F}_{n-1} )] - m^* [\bm{E}( Z_{n-1} )]^2 \nonumber \\
  &= \bm{E}[ Z_{n-1}^2  ] \bm{E}[m_{n}] - m^* [\bm{E}( Z_{n-1} )]^2 
  = m^* \bm{Var} (Z_{n-1}). \label{cov_zn,zn-1}
\end{align}
By induction, we get
\begin{align*}
      \bm{Cov} \left( Z_{n} ,  Z_{n-k} \right) =2 m^{*(n-k)}\bm{Var} (Z_{n-k}).
\end{align*}
Hence, we get
\begin{align*}
       \bm{Var}   ( \sum_{l=\tau}^n Z_{l})  &= \sum_{l=\tau}^n \bm{Var}   [ Z_{l}]  + 2 \sum_{\tau \le l<j\le n} \bm{Cov} (Z_l,Z_j)\\
       &= \sum_{l=\tau}^n \left[ \left(1+2 \sum_{k=1}^{n-l}  m^{*k} \right)\bm{Var} (Z_l)\right].
\end{align*}
\end{proof}

\noindent
Now we show the proof of Lemma \ref{var-sum-zn}.
\begin{proof}
   Recall that 
   \begin{align*}
   \bm{Var}(Z_l)=  m_{2}^{*(l-1)} \mathfrak{D}_l + m^{*2l} \sigma_A^2 .
\end{align*}
Now first consider the case that $ n-\tau_n \to \infty$, by Lemma \ref{varsumZl},
\begin{align}
    \bm{Var} \left[\frac{1}{r^{*n}} \frac{1}{\mathcal{N}_{n}^*} \sum_{l=\tau}^n Z_{l,j} \right] &= \frac{1 }{\mathcal{N}_{n}^{*2} r^{*2n}}  \left\{  \sum_{l=\tau}^n \left[ \left(1+2 \sum_{k=1}^{n-l}  m^{*k} \right) ( m_{2}^{*(l-1)} \mathfrak{D}_l + m^{*2l} \sigma_A^2  )\right]  \right\} \label{ma-truevar} \\
    &=  \frac{1 }{\mathcal{N}_{n}^{*2} r^{*2n}}  \left[  \sum_{l=\tau}^n \left(1+2 \sum_{k=1}^{n-l}  m^{*k} \right)  m_{2}^{*(l-1)} \mathfrak{D}_l  \right] \nonumber \\
    & \quad +\frac{1 }{\mathcal{N}_{n}^{*2} r^{*2n}}   \left[ \sum_{l=\tau}^n  \left(1+2 \sum_{k=1}^{n-l}  m^{*k} \right)  m^{*2l} \sigma_A^2  \right] \nonumber \\
    &\coloneqq  \mathcal{K}_{\tau,n,J,1 }+\mathcal{K}_{\tau,n,J,2 } ,\nonumber
\end{align}
consider $ \mathcal{K}_{\tau,n,J,2 }$, since 
\begin{align*}
\sum_{l=\tau}^n \left[ \left(1+2 \sum_{k=1}^{n-l}  m^{*k} \right) (
    m^{*2l} \sigma_A^2  )\right]  
  &= \sigma_A^2 \sum_{l=\tau}^n \frac{2m^{*(n+l+1)}-m^{*(2l+1)}-1 }{m^*-1 } \\
  &\le \sigma_A^2 \sum_{l=\tau}^n \frac{2m^{*(n+l+1)} }{m^*-1 }  
  =\frac{ 2\sigma_A^2 m^{*(n+\tau+1)} ( m^{*(n-\tau)}-1 )}{(m^*-1)^2},
\end{align*}
using $r^*>1 $, we have
\begin{align}
    \mathcal{K}_{\tau,n,J,2} 
    \le \frac{1 }{\mathcal{N}_{n}^{*2} r^{*2n}} \frac{ 2\sigma_A^2 m^{*(n+\tau+1)} ( m^{*(n-\tau)}-1 )}{(m^*-1)^2} \le \frac{1 }{ r^{*2n}}  \frac{ 2\sigma_A^2 m^{*2n+1} }{ (m^{*(2n +2)} -m^{*2\tau})   } \to 0. \label{k_taunJ1-c0}
\end{align}
Now consider $\mathcal{K}_{\tau,n,J,1 } $,  we have
\begin{align}
\mathcal{K}_{\tau,n,J,1 }   &= \frac{1 }{\mathcal{N}_{n}^{*2} r^{*2n}}  \left[ \frac{2m^{*n}  }{ m^*-1 }  \sum_{l=\tau}^n (m^*r^{2*})^{l-1}\mathfrak{D}_l - \frac{m^*+1}{ m^*-1 } \sum_{l=\tau}^n  m_{2}^{*(l-1)} \mathfrak{D}_l \right] \nonumber \\
    &=\frac{1}{r^{*2n}}  \frac{2m^{*n} (m^*-1)   }{ (m^{*( n  +1)} -m^{*\tau})^2     }\sum_{l=\tau}^n (m^*r^{2*})^{l-1}\mathfrak{D}_l - \frac{1}{r^{*2n}}  \frac{ (m^{*2}-1)}{(m^{*( n  +1)} -m^{*\tau})^2 } \sum_{l=\tau}^n (m^{*2}r^{2*})^{l-1}\mathfrak{D}_l \nonumber \\
    & \to \mathfrak{D} \left( \frac{2(m^*-1)}{ m^* (m_{2}^*-m^*)} -\frac{(m^{2*}-1)  }{ m^{2*} (m_{2}^*-1)}  \right) =  \frac{\mathfrak{D}(m^*-1)^2(m_{2}^*+m^*)}{m^{2*}(m_{2}^*- m^*)(m_{2}^*-1) } = \sigma_{I}^2.  \label{k_taunJ1-c1}
\end{align}
Hence, by (\ref{k_taunJ1-c0}) and (\ref{k_taunJ1-c1}), we have
\begin{align*}
    \bm{Var}\left[\frac{1}{r^{*n}} \frac{1 }{\mathcal{N}_{n}^*} \sum_{l=\tau}^n Z_{l,j} \right] \to \sigma_{I}^2.
\end{align*}
Now for $n-\tau_n $ is fixed, the proof follows the same logic, except that for $ \mathcal{K}_{\tau,n,J,1 } $ in this case, we have as $ n \to \infty $,
\begin{align}
  \mathcal{K}_{\tau,n,J,1 } &\to   \frac{2\mathfrak{D} (m^*-1)   }{ m^*r^{2*}-1 } \frac{m^{*2\delta_\tau}r^{*(2\delta_\tau+2)} - m^{*\delta_\tau-1} }{ (m^{*(\delta_\tau  +1)} -1)^2  r^{*(2\delta_\tau+2)}    } - \frac{\mathfrak{D}(m^{2*}-1)}{m^{2*}r^{2*}-1} \frac{m^{*(2\delta_\tau+2)} r^{*(2\delta_\tau+2)} -1}{r^{*(2\delta_\tau+2)}(m^{*(\delta_\tau  +2)} -m^*)^2} \label{k_taunJ1-c2} \\
     &=\frac{\mathfrak{D}(m^*-1)m^{*(2\delta_\tau+2)} }{m_{2}^{*(\delta_\tau+1)}(m^{*( \delta_\tau  +2)} -m^*)^2} \left( \frac{2m^* m_{2}^{*(\delta_\tau+1)} - 2m^{*(\delta_\tau+2)}}{ m_{2}^*-m^*  }  - \frac{(m^*+1)(m_{2}^{*(\delta_\tau+1)}-1 ) }{m_{2}^*-1} \right)  
     = \sigma_{F}^2(\delta_\tau).  \nonumber
\end{align}
The rest of the proof again follows Slutsky's theorem.
\begin{align*}
    \bm{Var}\left[\frac{1}{r^{*n}} \frac{1 }{\mathcal{N}_{n}^*} \sum_{l=\tau}^n Z_{l,j} \right] \to \sigma_{F}^2(\delta_\tau).
\end{align*}
Notice that as $\delta_\tau \to \infty $, the RHS of (\ref{k_taunJ1-c2}) is increasing and converges to the RHS of (\ref{k_taunJ1-c1}), implies  $\sigma_{F}^2(\delta_\tau) \to \sigma_{I}^2 $. Meanwhile, if $\delta_\tau=0 $, $ \sigma_{F}^2(0)=\frac{\mathfrak{D}}{m_{2}^*}$.
\end{proof}

\subsection{Additional details of Lemma \ref{EY2l^2-finite}}
\label{apd-c-EY2l^2-finite}
 Recall that 
\begin{align*}
     L_{n, \tau_n} = 2 \bm{Cov}[ \sum_{l=\tau_n+1}^{n-1} Y_{2,l},\frac{\sigma_{n}^2}{Z_n} ] + \bm{Var} [\frac{\sigma_{n}^2}{Z_n} ] + \bm{E}[   \bm{Var}   (    Y_{2,n}  | \mathcal{F}_n  ) ],
\end{align*}
our next lemma shows that $L_{n, \tau_n}$ is bounded by a constant $c_\delta^{**} $ a term that converges to zero exponentially fast: $n(a^{\frac{1}{2}})^n   $.

\begin{lemma}
    \label{L_ntau_n}
    Under conditions and settings of Lemma \ref{EY2l^2-finite}, there exists $ c_\delta^{**}>0 $, and $0<a<1 $, 
    \begin{align*}
         L_{n, \tau_n}   
     \le  c_\delta^{**}   n(a^{\frac{1}{2}})^n   .
    \end{align*}
\end{lemma}

\begin{proof}
Let $   L_{n,\tau_n} = 2 L_{n,\tau_n,1}+L_{n,2}+L_{n,3} $ where $L_{n,\tau_n,1}= \bm{Cov}[ \sum_{l=\tau_n+1}^{n-1} Y_{2,l}, Z_n^{-1}\sigma_{n}^2]$, $L_{n,2}= \bm{Var} [Z_n^{-1}\sigma_{n}^2 ]$, and $L_{n,3} =\bm{E}[   \bm{Var}   (    Y_{2,n}  | \mathcal{F}_n  ) ] $. 
For $L_{n,\tau_n,1}$, we have 
\begin{align}
   L_{n,\tau_n,1}
   &= \sum_{l=\tau_n+1}^{n-1}\bm{Cov}\left[  \left( \frac{Z_{l+1}}{Z_l}\right)^2,\frac{\sigma_{n}^2}{Z_n} \right]  -  \sum_{l=\tau_n+1}^{n-1}\bm{Cov}\left[ m_{2,l},\frac{\sigma_{n}^2}{Z_n} \right] \nonumber \\
   & \le \left| \sum_{l=\tau_n+1}^{n-1}\bm{Cov}\left[  \left( \frac{Z_{l+1}}{Z_l}\right)^2,\frac{\sigma_{n}^2}{Z_n} \right]  \right| + \left|  \sum_{l=\tau_n+1}^{n-1}\bm{Cov}\left[ m_{2,l},\frac{\sigma_{n}^2}{Z_n} \right] \right| \nonumber \\
   &\coloneqq  | L_{n,\tau_n,1,1} |+ | L_{n,\tau_n,1,2}| \label{Ln12-p1}
\end{align}
now for $| L_{n,\tau_n,1,1} | $ because of nonnegativity, let $ \sqrt{\bm{E}[ m_{2,l}^2 ]} = \sigma_{2}^* $, by Cauchy Schwarz Inequality,  
\begin{align*}
  | L_{n,\tau_n,1,1} | &\le \sum_{l=\tau_n+1}^{n-1}\bm{E}\left[ \frac{m_{2,l} \sigma_{n}^2}{Z_n} \right] 
  = \sum_{l=\tau_n+1}^{n-1}\bm{E}[\sigma_{n}^2]\bm{E}\left[\frac{ m_{2,l}}{Z_n} \right] \\
   &\le  \gamma^{2*} \sum_{l=\tau_n+1}^{n-1} \sqrt{\bm{E}[ m_{2,l}^2 ]}\sqrt{\bm{E}\left[\frac{1}{Z_n^2} \right]} \le (n-\tau_n) \gamma^{2*}  \sigma_{2}^*\sqrt{\bm{E}\left[\frac{1}{Z_n} \right]}
\end{align*}
and for $| L_{n,\tau_n,1,2}|$, by Cauchy Schwarz Inequality and Theorem 2.8 in \cite{Huang-Wang2023}, there exists $ c_\delta>0$ such that
\begin{align*}
 | L_{n,\tau_n,1,2}|
    &\le  \sum_{l=\tau_n+1}^{n-1}\bm{E}\left[\left( \frac{Z_{l+1}}{Z_l}\right)^2\frac{\sigma_{n}^2}{Z_n} \right]=  \gamma^{2*} \sum_{l=\tau_n+1}^{n-1}\bm{E}\left[\left( \frac{Z_{l+1}}{Z_l} -m_{l}+m_{l}\right)^2\frac{1}{Z_n} \right]  \\
  & \le \gamma^{2*} \sum_{l=\tau_n+1}^{n-1} 2 \left(\sqrt{\bm{E}\left[ \frac{Z_{l+1}}{Z_l} -m_{l}\right]^4 } + \sqrt{\bm{E}[  m_{l}^4 ]} \right) \sqrt{\bm{E}\left[\frac{1}{Z_n}\right]} \\
  & \le  \gamma^{2*} \sum_{l=\tau_n+1}^{n-1} 2(c_\delta    + \sqrt{m_{4}^*}) \sqrt{\bm{E}\left[\frac{1}{Z_n}\right]}.
\end{align*}
Then for some constant $ c_\delta^* >  \gamma^{2*} (\sigma_{2}^*+2c_\delta+2\sqrt{m_{4}^*} )$  and (\ref{Ln12-p1}) we have 
\begin{align*}
    L_{n,\tau_n,1} 
    &\le   (n-\tau_n)  \gamma^{2*}  \sigma_{2}^*\sqrt{\bm{E}\left[\frac{1}{Z_n} \right]} + 2(n-\tau_n)  \gamma^{2*}(c_\delta     + \sqrt{m_{4}^*}) \sqrt{\bm{E}\left[\frac{1}{Z_n}\right]} \\
    & \le (n-\tau_n)  \gamma^{2*} (\sigma_{2}^*+2c_\delta+2\sqrt{m_{4}^*} )\sqrt{\bm{E}\left[\frac{1}{Z_n}\right]} \coloneqq (n-\tau_n)  c_\delta^* \sqrt{\bm{E}\left[\frac{1}{Z_n}\right]}.
\end{align*}
For $L_{n,2} $, let $\bm{E}[\sigma_{n}^2]^2=\gamma_{2}^{2*}  $, we have
\begin{align*}
 L_{n,2} \le 
     \bm{E} \left[ \frac{\sigma_{n}^2}{Z_n} \right]^2 \le \bm{E}\left[\frac{1}{Z_n}\right] \gamma_{2}^{2*} 
\implies
    L_{n,2} \le \bm{E}\left[\frac{1}{Z_n}\right] \gamma_{2}^{2*} .
\end{align*}
Finally, for $L_{n,3}$, let $ \bm{E}[\bm{Var}[\xi_{n,1}^2|\mathcal{F}_n]] =\rho_{2}^{2*} $, since
\begin{align*}
     \bm{Var}  \left[ Z_{n+1}^2 \middle|  \mathcal{F}_n  \right]
  & =  \bm{Var}  \left[ \left( \sum_{k=1}^{Z_n} \xi_{n,k} \right)^2\middle|  \mathcal{F}_n  \right] 
   = \bm{Var}  \left[ \left( \sum_{k=1}^{Z_n} \xi_{n,k}^2 \right) + 2\sum_{1\le k_1<k_2 \le Z_n} \xi_{n,k_1}  \xi_{n,k_2} \middle|  \mathcal{F}_n  \right]\\
   &=  Z_n \bm{Var}[\xi_{n,1}^2|\mathcal{F}_n] + 4 Z_n (Z_n-1)\bm{Var}[\xi_{n,1}\xi_{n,2}|\mathcal{F}_n] + 2 Z_n (Z_n-1) \bm{Cov}[ \xi_{n,1}^2,\xi_{n,1}\xi_{n,2}]  \\
   &\le  6 Z_n^2 \bm{Var}[\xi_{n,1}^2|\mathcal{F}_n],
\end{align*}
we have
\begin{align*}
     L_{n,3} =  \bm{E}\left[   \frac{1}{Z_n^4} \bm{Var}  \left[ Z_{n+1}^2 \middle|  \mathcal{F}_n  \right] \right] \le  \bm{E}\left[\frac{6}{Z_n^2} \right] \bm{E}[\bm{Var}[\xi_{n,1}^2|\mathcal{F}_n]] \le
       \bm{E}\left[\frac{6}{Z_n}\right]\rho_{2}^{2*} .
\end{align*}
Then for some constant $c_\delta^{**} > c_\delta^*+\gamma_{2}^{2*} +\rho_{2}^{2*}  $, by Lemma \ref{sum1/zn},
\begin{align*}
   L_{n, \tau_n}   
     \le  c_\delta^*  (n-\tau_n) \sqrt{\bm{E}\left[\frac{1}{Z_n}\right]} + \gamma_{2}^{2*}   \bm{E}\left[\frac{1}{Z_n}\right]  +\rho_{2}^{2*}   \bm{E}\left[\frac{1}{Z_n}\right] \le c_\delta^{**}   n(a^{\frac{1}{2}})^n   .
\end{align*}

\end{proof}

\subsection{Additional details of Theorem \ref{asymp_varma}}
\label{apd-asymp_varma}
Here we prove that the first term on the RHS of (\ref{fkD3}), $\frac{1}{r^{*2n}} \frac{1}{J_n}  \sum_{j=1}^{J_n}\left( \frac{Z_{n,j}   }{ \tilde{m}_{n}^n } - m_A\right)^2  $, converges to $ \frac{\mathfrak{D}}{m_2^*} $ in probability.

\begin{proof}
Notice
    \begin{align*}
  \frac{1}{r^{*2n}} \frac{1}{J_n}  \sum_{j=1}^{J_n}\left( \frac{Z_{n,j}   }{ \tilde{m}_{n}^n } - m_A\right)^2 
  &= \frac{1}{r^{*2n}} \frac{1}{J_n}  \frac{m^{*2n}  }{ \tilde{m}_{n}^{2n} } \sum_{j=1}^{J_n}\left( \frac{Z_{n,j}   }{ m^{*n} } - \frac{ \tilde{m}_{n}^{2n} }{m^{*2n}  } m_A\right)^2 \\
  &= \frac{1}{r^{*2n}} \frac{1}{J_n}  \frac{m^{*2n}  }{ \tilde{m}_{n}^{2n} } \sum_{j=1}^{J_n}\left( \frac{Z_{n,j}   }{ m^{*n} } - m_A+m_A- \frac{ \tilde{m}_{n}^{2n} }{m^{*2n}  } m_A\right)^2 \\
  &= \frac{1}{r^{*2n}} \frac{1}{J_n}  \frac{m^{*2n}  }{ \tilde{m}_{n}^{2n} } \sum_{j=1}^{J_n}\left( \frac{Z_{n,j}   }{ m^{*n} } - m_A\right)^2+ 
 \frac{1}{r^{*2n}} \frac{m_A^2}{J_n}   \sum_{j=1}^{J_n}\left( \frac{m^{*n}  }{ \tilde{m}_{n}^{n} } - 1 \right)^2 \\
  & \quad  + \frac{1}{J_n} m_A  
\left( \frac{m^{*n}  }{ \tilde{m}_{n}^{n} } - 1 \right)  \frac{m^{*n}  }{ \tilde{m}_{n}^{n} }  \frac{1}{r^{*2n}} \sum_{j=1}^{J_n} \left( \frac{Z_{n,j}   }{ m^{*n} } - m_A\right) \\
&=T_n(1)+T_n(2)+T_n(3);
 \end{align*}
Now by Theorem \ref{mo^n to 1},
\begin{align*}
    T_n(2) = \frac{1}{r^{*2n}}  m_A^2  \left( \frac{m^{*n}  }{ \tilde{m}_{n}^{n} } - 1 \right)^2 \xrightarrow[]{p} 0.
\end{align*}
By Theorem \ref{mo^n to 1}, Theorem \ref{consis-CLT4ma}, 
\begin{align*}
    T_n(3)=  m_A  
\left( \frac{m^{*n}  }{ \tilde{m}_{n}^{n} } - 1 \right)  \frac{m^{*n}  }{ \tilde{m}_{n}^{n} } \left[ \frac{1}{r^{*2n}} \frac{1}{J_n} \sum_{j=1}^{J_n} \left( \frac{Z_{n,j}   }{ m^{*n} } - m_A\right)  \right] \xrightarrow[]{p} 0.
\end{align*}
For $T_n(1)$, let $X_{nj}=\frac{1}{r^{*2n}} \left( \frac{Z_{n,j}   }{ m^{*n} } - m_A\right)^2 $, then 
\begin{align*}
 T_n(1)=  \frac{m^{*2n}  }{ \tilde{m}_{n}^{2n} }   \frac{1}{J_n}  \sum_{j=1}^{J_n}  X_{nj}.
\end{align*}
By Theorem \ref{mo^n to 1}, $ \frac{m^{*2n}  }{ \tilde{m}_{n}^{2n} } \xrightarrow[]{p} 1$, hence it is sufficient to show that $\frac{1}{J_n}  \sum_{j=1}^{J_n}  X_{nj} \xrightarrow[]{p} \frac{\mathfrak{D}}{m_2^*}  $. Recall from (\ref{var(zn/m^n)-to-D}) that $\bm{E}[X_{nj}]=\frac{\mathfrak{D}_n}{m_{2}^*} + \frac{\sigma_A^2}{r^{*2n}} \coloneqq C_{nj} $, by Chebyshev's inequality, and independence in $j$,
\begin{align}
    \bm{P} \left[ \left| \frac{1}{J_n}  \sum_{j=1}^{J_n}  X_{nj} - C_{nj} \right|>\epsilon \right]   &\le  \frac{1}{J_n} \bm{Var}[ X_{nj}] \nonumber \\
    & \le \frac{1}{J_n r^{*4n} }\bm{E} \left[  \left( \frac{Z_{n,j}   }{ m^{*n} } - m_A\right)^4 \right] \text{, by using }r^{*2}= \frac{m_2^{*}}{m^{*2}} \nonumber \\
    &\le \frac{\bm{E} [  Z_{n,j}^4 ]}{J_n m_2^{*2n}m^{*2n}  }   +  \frac{6\bm{E} [  Z_{n,j}^2 ]}{J_n m_2^{*2n} } \nonumber \\
    & \le  \frac{  c_4 m_4^{*n} }{J_n m_2^{*2n}m^{*2n}  }   +  \frac{6  c_2 m_2^{8n} }{J_n m_2^{*2n} } \label{asymp_varma-sup1},
\end{align}
where the last inequality follows from Theorem \ref{HL-zn/c-limit} for large $n$ and $c_4$ and $c_2$ are positive constants. Now since $\frac{n}{ \log J_n} \to 0 $ implies $\frac{a^n}{J_n}$ converges to 0 for any $a >0$, we have taking limits that each term on the RHS of (\ref{asymp_varma-sup1}) converges to zero, implying that  $ \left|\frac{1}{J_n}  \sum_{j=1}^{J_n}  X_{nj} - C_{nj} \right|\xrightarrow[]{p}0  $. Also, from (\ref{var(zn/m^n)-to-D}), since  $C_{nj} \to \frac{\mathfrak{D}}{m_2^*} $ as $ n \to \infty$  we have that $T_n(1) \xrightarrow[]{p} \frac{\mathfrak{D}}{m_2^*} $ as $n \to \infty$.  Hence,
\begin{align*}
    T_n(1)+T_n(2)+T_n(3) \xrightarrow[]{p} \frac{\mathfrak{D}}{m_2^*}. 
\end{align*}

\end{proof}

\subsection{Results from literature}
\label{apd-c-tfo}
The following results are cited from different papers. Since they are necessary for the proof of multiple results in this paper, we state the theorems to make the script self-contained.

\begin{theorem}[Theorem 3.2 in \cite{B-Hall-Heyde1980}]
    \label{PHclt-da}
 Let $\{ S_{ni}, \mathcal{F}_{ni}, 1\le i \le k_n, n\ge 1  \} $ be a zero-mean, square-integrable martingale array with differences $X_{ni}$, and let $\eta^2$ be an s.s. finite r.v. Suppose that 
\begin{gather}
    \max_{1\le i \le k_n} |X_{ni}| \xrightarrow[]{P} 0, \label{PHCLTAc1} \\
    \sum_i^{k_n} X_{ni}^2 \xrightarrow[]{P} \eta^2, \label{PHCLTAc2}\\
    \bm{E} (\max_{1\le i \le k_n} X_{ni}^2 ) \text{ is bounded in } n, \label{PHCLTAc3}\\
    \text{ and the $\sigma$-fields are nested: }\mathcal{F}_{n,i} \subseteq \mathcal{F}_{n+1,i} \quad \text{ for } 1\le i \le k_n , \; n  \ge 1.  \label{PHCLTAc4}
\end{gather}
Then $\displaystyle S_{n k_n}=\sum_i^{k_n} X_{ni} \xrightarrow[]{d} Z$(stably), where the r.v. $Z$ has characteristic function $\bm{E}[\exp(-\frac{1}{2} \eta^2t^2)]$.

\end{theorem}

\begin{corollary}[Corollary 3.1 in \cite{B-Hall-Heyde1980}]
\label{PHclt-da-C}
If (\ref{PHCLTAc1}) and (\ref{PHCLTAc3}) are replaced by the Linderberg condition:
\begin{align}
    \forall \; \epsilon>0, \quad \sum_i^{k_n} \bm{E}[X_{ni}^2I(|X_{ni}|>\epsilon) |\mathcal{F}_{n,i-1}] \xrightarrow[]{P} 0 , \label{PHCLTAc5}
\end{align}
If (\ref{PHCLTAc2}) is replaced by the analogous condition on the conditional variance:
\begin{align}
    V_{nk_n}^2= \sum_{i=1}^{k_n} \bm{E}[X_{ni}^2|\mathcal{F}_{n,i-1}] \xrightarrow[]{P} \eta^2, \label{PHCLTAc4b}
\end{align}
and if (\ref{PHCLTAc4}) holds, then the conclusion of Lemma \ref{PHclt-da} remains true.
    
\end{corollary}

\begin{theorem}[Theorem 1.3 in \cite{Huang-Liu2012}]
 \label{HL-zn/c-limit}
Let $t \in \mathbb{R} $ Suppose that $\bm{E}[ Z_1^t]< \infty $, for some constant $c_t \in (0, \infty) $,
\begin{align}
    \lim_{n \to \infty} \frac{ \bm{E}[Z_n^t]}{ m_{t}^{*n}} \to c_t.
\end{align}
    
\end{theorem}

\end{document}